\documentclass[a4paper,10pt]{amsart}
\usepackage{amssymb, amsmath,stmaryrd, svninfo, enumerate, diagrams, mathtools}
\usepackage[all]{xy}
\usepackage[pagebackref, colorlinks]{hyperref}

\svnInfo $Id: localmc.tex 696 2013-03-07 00:33:34Z davidloeffler $

 \let\mathbb\mathbf

 \newcommand{\mybf}{\mathbf}
 \renewcommand{\AA}{\mybf{A}}
 \newcommand{\BB}{\mybf{B}}
 \newcommand{\CC}{\mybf{C}}
 \newcommand{\DD}{\mybf{D}}
 \newcommand{\FF}{\mybf{F}}
 \newcommand{\QQ}{\mybf{Q}}
 \newcommand{\ZZ}{\mybf{Z}}
 \newcommand{\NN}{\mybf{N}}
 \newcommand{\cD}{\mathcal{D}}
 \newcommand{\cH}{\mathcal{H}}
 \newcommand{\cL}{\mathcal{L}}
 \newcommand{\cK}{\mathcal{K}}
 \newcommand{\cO}{\mathcal{O}}

 \newcommand{\into}{\hookrightarrow}
 \newcommand{\vp}{\varphi}

 \newcommand{\Qp}{\QQ_p}
 \newcommand{\Zp}{\ZZ_p}
 \newcommand{\Qpi}{\QQ_{p, \infty}}
 \newcommand{\Qpn}{\QQ_{p, n}}
 \newcommand{\Qpnr}{\Qp^{\mathrm{nr}}}
 \newcommand{\Qpb}{\overline{\QQ}_p}
 \newcommand{\Cp}{\CC_p}

 \newcommand{\Dcris}{\DD_{\cris}}
 \newcommand{\DdR}{\DD_{\dR}}
 \newcommand{\Bcris}{\BB_{\cris}}
 \newcommand{\BdR}{\BB_{\dR}}
 \newcommand{\Brig}{\BB^+_{\rig,\Qp}}
 \newcommand{\Dpst}{\DD_{\pst}}
 \newcommand{\Lt}{\widetilde{L}}
 \newcommand{\Rt}{\widetilde{R}}
 \newcommand{\cOt}{\widetilde{\cO}}

 \newcommand{\cris}{\mathrm{cris}}
 \newcommand{\dR}{\mathrm{dR}}
 \newcommand{\Iw}{\mathrm{Iw}}
 \newcommand{\rig}{\mathrm{rig}}
 \newcommand{\perf}{\mathrm{perf}}
 \newcommand{\pst}{\mathrm{pst}}

 \DeclareMathOperator{\gr}{gr}
 \DeclareMathOperator{\ev}{ev}
  \DeclareMathOperator{\mysp}{sp}
  \renewcommand{\sp}{\mysp}
 \DeclareMathOperator{\Fil}{Fil}
 
 \DeclareMathOperator{\Det}{Det}
 \DeclareMathOperator{\Gal}{Gal}
 \DeclareMathOperator{\id}{id}
 
 \DeclareMathOperator{\Ch}{Ch}
 \DeclareMathOperator{\Hom}{Hom}
 \DeclareMathOperator{\Isom}{Isom}
 \DeclareMathOperator{\cores}{cores}
 \DeclareMathOperator{\Char}{char}
 \DeclareMathOperator{\tw}{Tw}

\newcommand{\htimes}{\mathop{\widehat\otimes}}

\renewcommand{\det}{\operatorname{det}}

 \newtheorem{definition}{Definition}[subsection]
 \newtheorem{proposition}[definition]{Proposition}
 \newtheorem{theorem}[definition]{Theorem}
 \newtheorem{corollary}[definition]{Corollary}
 
 \newtheorem{lemma}[definition]{Lemma}

 \theoremstyle{remark}
 \newtheorem{remark}[definition]{Remark}

 \renewcommand{\u}{\mathbf{1}}
 \renewcommand{\H}{\mathrm{H}}
 \newcommand{\lu}{\mathbb{U}}
 \newcommand{\T}{\mathbb{ T}}

\begin{document}

\title{Local epsilon isomorphisms}

\author{David Loeffler}
\address[Loeffler]{Mathematics Institute\\
Zeeman Building\\
University of Warwick\\
Coventry CV4 7AL, UK}
\email{d.a.loeffler@warwick.ac.uk}
\urladdr{\url{http://www2.warwick.ac.uk/fac/sci/maths/people/staff/david_loeffler/}}

\author{Otmar Venjakob}
\address[Venjakob]{Mathematisches Institut \\ Universit\"at Heidelberg \\ Im Neuenheimer Feld 288 \\69120 Heidelberg, Germany}
\email{otmar@mathi.uni-heidelberg.de}
\urladdr{\url{http://www.mathi.uni-heidelberg.de/~otmar/}}

\author{Sarah Livia Zerbes}
\address[Zerbes]{Department of Mathematics\\
University College London\\
Gower Street, London WC1E 6BT, UK}
\email{s.zerbes@ucl.ac.uk}

\thanks{The first author's research is supported by a Royal Society University Research Fellowship, the second author's by an ERC grant and the third author's by the EPSRC First Grant EP/J018716/1.}


\begin{abstract}
 In this paper, we prove the ``local $\varepsilon$-isomorphism'' conjecture of Fukaya and Kato \cite{fukayakato06} for a particular class of Galois modules obtained by tensoring a $\Zp$-lattice in a crystalline representation of $G_{\Qp}$ with a representation of an abelian quotient of $G_{\Qp}$ with values in a suitable $p$-adic local ring $R$. This can be regarded as a local analogue of the Iwasawa main conjecture for abelian $p$-adic Lie extensions of $\Qp$. We show that such an $\varepsilon$-isomorphism can be constructed using the Perrin-Riou regulator map, or its extension to the 2-variable case due to the first and third author.
\end{abstract}

\maketitle
\tableofcontents

\section{Introduction}

 \subsection{Aims}

  In this paper, we prove a special case of the ``local $\varepsilon$-isomorphism conjecture'' of Fukaya and Kato (Conjecture 3.4.3 of \cite{fukayakato06}). This conjecture asserts the existence of a canonical trivialization of the determinant of a cohomology complex associated to any representation $M$ of the Galois group $G_{\Qp} = \Gal(\Qpb / \Qp)$ with coefficients in a $p$-adically complete local ring $R$, which should be compatible with a specific ``standard'' trivialization of the corresponding complex for representations with coefficients in finite extensions of $\Qp$ obtained by specializing $M$ at ideals of $R$; this ``standard'' trivialization contains information about the $\varepsilon$-factor of the corresponding Weil--Deligne representation, hence the terminology ``$\varepsilon$-isomorphism''.

  The main result of this paper is a proof of this conjecture for a specific class of modules $M$: we consider the case where $M$ is obtained by tensoring a lattice in a crystalline Galois representation with an $R$-linear representation of the abelianization $\mathcal{G} = \Gal(\Qp^{\mathrm{ab}} / \Qp)$. This specific instance of the local $\varepsilon$-isomorphism conjecture can be thought of as a ``local Iwasawa main conjecture'' for $T$ over the extension $\Qp^{\mathrm{ab}} / \Qp$.

  The key ingredient in the construction of the local $\varepsilon$-isomorphism is the two-variable regulator map introduced by the first and third authors \cite{loefflerzerbes11}. Essentially, the local $\varepsilon$-isomorphism is obtained by taking the determinant of the two-variable regulator, after dividing out by certain correction factors determined purely by the Hodge--Tate weights of the Galois representation. The details of the construction are slightly tortuous, so we give an outline below.

  In particular, our results generalize those of \cite{benoisberger08} (and contain independent proofs of those): via the usual functorial properties of determinants, our results imply the conjectures $C_{IW}(K(\mu_{p^\infty})/K,V)$ and $C_{EP}(F/K,V)$ of \emph{loc.\ cit.} for all finite subextensions $\Qp \subseteq K \subseteq F \subseteq \Qp^{ab}$. However, we also obtain compatibility with the standard $\varepsilon$-isomorphisms for unramified characters which are not necessarily of finite order, a case which is not considered in \cite{benoisberger08}.

 \subsection{Outline of the paper}

  In sections \ref{sect:notation}-\ref{sect:epsfactors} we fix notation, recall the formalism of determinants used in the formulation of the conjecture, and recall some definitions and results concerning $\varepsilon$-factors of de Rham representations. In section \ref{sect:deRhamepsilon} we describe the ``standard'' $\varepsilon$-isomorphism, for representations with coefficients in a finite extension of $\Qp$, with which the general $\varepsilon$-isomorphism is required to be compatible. In section \ref{sect:specialcases} we give an alternative, more convenient description of this isomorphism in the cases relevant to us.

  In section \ref{sect:regulators} we recall the definition and properties of the cyclotomic regulator map and of the two-variable version defined in \cite{loefflerzerbes11}. Some of the formulae we need for these maps are a little stronger than those that can be found in the literature, so we give the proofs of these formulae in appendix \ref{sect:app-formulary}.

  In section \ref{sect:constructionG} we construct our determinant isomorphism for the ``universal'' case when $R$ is the Iwasawa algebra of an abelian $p$-adic Lie quotient $G$ of $G_{\Qp}$. We begin (in section \ref{sect:construction1}) by constructing an isomorphism of determinants over a rather enormous ring (the total ring of quotients of the distribution algebra of $G$); in section \ref{sect:descenttoLambda} we show that this descends to the Iwasawa algebra (but only after inverting $p$). The final descent to an isomorphism over $\Lambda_{\cO}(G)$ is accomplished in section \ref{sect:powerofp}. We then obtain isomorphisms for more general $R$ by (derived) base change.

  Section \ref{sect:evaluation} deals with the compatibility of our determinant isomorphism over $\Lambda_{\cO}(G)$ with the ``standard'' isomorphisms for $V$ and all of its twists by de Rham characters of $G$. We divide up the set of characters into classes depending on $V$, which we refer to as \emph{good}, \emph{somewhat bad}, and \emph{extremely bad}; we deal with each of these classes separately. The first two cases can be handled by direct computation using the properties of the one-variable and two-variable regulator maps; for the extremely bad characters, we use induction on the dimension and the results of the second author in the case of one-dimensional representations (cf.~\cite{venjakob11}).

\section{Preliminaries}

 \subsection{Notation}
  \label{sect:notation}

  Let $p$ be an odd prime. For any $p$-adic analytic group $H$ without $p$-torsion, and $L$ a complete discretely valued extension of $\Qp$ with ring of integers $\cO$, we write $\Lambda_{\cO}(H)$ and $\Lambda_L(H)=L\otimes \Lambda_\cO(H)$ for the Iwasawa algebras of $H$ with $\cO$ and $L$ coefficients, and $\cH_L(H)$ for the algebra of $L$-valued distributions on $H$. We shall only use these constructions in cases where $H$ is abelian and $p$-torsion-free, in which case all of these algebras are reduced commutative semi-local rings, and can be interpreted as algebras of functions on the $p$-adic analytic space parametrising characters of $H$.

  We shall also need the notation $\cK_{L}(H)$, signifying the total ring of quotients of $\cH_L(H)$, which is a finite direct product of fields.

  Let $\Qpi = \Qp(\mu_{p^\infty})$ and $\Gamma = \Gal(\Qpi / \Qp)$, and let $\chi: \Gamma \to \Zp^\times$ be the cyclotomic character. For $j \in \ZZ$, we define the element $\ell_j$ of $\cH_{\Qp}(\Gamma)$ by
  \[ \ell_j = \frac{\log(\gamma)}{\log \chi(\gamma)}-j\]
  for any non-torsion element $\gamma \in \Gamma$ (the element $\ell_j$ is independent of this choice). We also fix a norm-compatible system of $p$-power roots of unity $\xi = (\xi_n)_{n \ge 1} \in \Qpi$.

  If $\mu \in \cH_L(\Gamma)$, and $\eta$ is a character of $\Gamma$, we define $\mu'(\eta)$ by
  \[ \mu'(\eta) = \lim_{s \to 0} \frac{\mu(\eta \langle \chi\rangle^s) - \mu(\eta)}{s},\]
  where $\langle . \rangle$ denotes the projection $\Zp^\times \to 1 + p\Zp$. This limit exists for all $\mu$, and we have
  \begin{equation}
   \label{derivation}
   \mu'(\eta) = \int_{\Gamma} \eta(\tau) \log \chi(\tau)\, \mathrm{d}\mu(\tau).\end{equation}
  If $\mu$ is not a zero-divisor, we may define $\mu^*(\eta)$ to be the value of the lowest non-vanishing derivative of $\mu$ at $\tau$.

  We write $\gamma_{-1}$ for the unique element of $\Gamma$ such that $\chi(\gamma_{-1}) = -1$.

 \subsection{K-theory and determinants}
  \label{sect:dets}

  Let $R$ be a ring. We define $K_0(R)$ and $K_1(R)$ in the usual way, as in \cite[\S 1.1]{fukayakato06}. If $R$ is commutative, then there is a canonical surjective map $K_1(R) \to R^\times$, and the kernel of this is the special Whitehead group $SK_1(R)$.

  The following statements are well known:

  \begin{proposition}
   Let $L$ be a complete discretely valued subfield of $\Cp$ and $\cO$ its ring of integers. Then for any $p$-torsion-free abelian $p$-adic Lie group $G$, the Iwasawa algebras $\Lambda_L(G)$ and $\Lambda_{\cO}(G)$ have trivial $SK_1$.
  \end{proposition}

  \begin{remark}
   We do not know whether, in the above setting, the distribution algebra $\cH_L(G)$ also has trivial $SK_1$. If this were true, it would allow certain arguments below to be shortened somewhat.
  \end{remark}

  For $R$ a ring, let $\Ch(R)_{\perf}$ denote the category of perfect complexes of $R$-modules (chain complexes quasi-isomorphic to a bounded complex of finitely-generated projective $R$-modules). We denote by $\underline{\Det}(R)$ the category denoted by $\mathcal{C}_R$ in \cite{fukayakato06}, which is equivalent to the universal Picard category for the category of finitely-generated projective $R$-modules.

  We denote by $\Det$ the canonical functor $\Ch(R)_{\perf} \to \underline{\Det}(R)$. This factors through the derived category $D(R)_{\perf}$ of perfect complexes.

 \subsection{Epsilon-factors}
  \label{sect:epsfactors}

  We recall the definition of $\varepsilon$-factors associated to representations of the Weil group of $\Qp$, for which the canonical reference is \cite{deligne73}. These are constants
  \[ \varepsilon_E(\Qp, D, \psi, \mathrm{d}x) \in E^\times \]
  where $E$ is a field of characteristic 0 containing $\mu_{p^\infty}$, $\psi$ is a locally constant $E$-valued character of $\Qp$, $\mathrm{d}x$ is a Haar measure on $\Qp$, and $D$ is a finite-dimensional $E$-linear representation of the Weil group $W(\Qpb / \Qp)$ which is locally constant (i.e.~the image of the inertia group $I(\Qpb / \Qp)$ is finite).

  Following \cite[\S 3.2]{fukayakato06}, we shall restrict to the case when $\mathrm{d}x$ is the usual Haar measure giving $\Zp$ measure 1, and $\psi$ has kernel equal to $\Zp$; the data of such a character $\psi$ is equivalent to the data of a compatible system of $p$-power roots of unity $\xi = (\xi_n)_{n \ge 1}$, via the map sending $\psi$ to $\big(\psi(p^{-n})\big)_{n \ge 1}$. Since $\mathrm{d}x$ and $\Qp$ are fixed, and $\psi$ is determined by $\xi$, we shall drop them from the notation and write the $\varepsilon$-factor as $\varepsilon_E(D, \xi)$.

  \begin{remark}
   Note that our conventions here, which were chosen for compatibility with \cite{fukayakato06}, differ slightly from the conventions of \cite{loefflerzerbes11}, which were chosen for compatibility with \cite{cfksv}: in \cite{loefflerzerbes11} we defined $\varepsilon$-factors using the additive character mapping $p^{-n}$ to $\xi_n^{-1}$, which is more convenient for global purposes.
  \end{remark}

  We are interested in the case when $D = \Dpst(W)$ for a de Rham representation $W$ of $G_{\Qp}$, with the linearized action of the Weil group given as in \cite{fontaine94b}. If $W$ is an $L$-linear representation of dimension $d$, for $L$ a finite extension of $\Qp$, then $\Dpst(W)$ is naturally a free module of rank $d$ over $\Qpnr \otimes_{\Qp} L$, and we may obtain the necessary roots of unity by extending scalars to $\Qpb \otimes_{\Qp} L$; but this is, of course, not a field but rather a finite product of fields indexed by embeddings $f: L \into \Qpb$. Following \cite[\S 3.3.4]{fukayakato06}, we define
  \[ \varepsilon_L(\Dpst(W), \xi) = \left( \varepsilon_{\Qpb} \left(\Qpb \otimes_{(L \otimes \Qpnr, f)} \Dpst(W), \xi\right) \right)_{f} \in \left(\Qpb \otimes_{\Qp} L\right)^\times = \prod_{f} \Qpb^\times.\]

  Any such character can be uniquely written as $\eta = \chi^j \eta_0 \eta_1$, where $j \in \ZZ$, $\eta_0$ is a finite-order character factoring through $\Gamma$, and $\eta_1$ is an unramified charater.

  \begin{proposition}
   The functor $\Dpst$ has the following properties:
   \begin{itemize}
    \item it commutes with tensor products;
    \item if $V$ is crystalline, then the linearized action of $W_{\Qp}$ on $\Dpst(V) \cong \Qpnr \otimes_{\Qp} \Dcris(V)$ is unramified, and the action of arithmetic Frobenius $\sigma \in W_{\Qp} / I_{\Qp}$ coincides on $\Dcris(V)$ with the inverse of the crystalline Frobenius;
    \item if $\eta$ is finitely ramified, then $W_{\Qp}$ acts on $\Dpst(\eta)$ via the character $\eta$;
    \item if $\eta = \chi$, then $W_{\Qp}$ acts on $\Dpst(\eta)$ via the unramified character mapping arithmetic Frobenius to $p$.
   \end{itemize}
  \end{proposition}

  \begin{proof}
   The compatibility of $\Dpst$ with tensor products follows from the corresponding statement for $\DD_{\mathrm{st}}$, which is standard. The remaining statements follow immediately from the definition of the linearized action of the Weil group on $\Dpst$.
  \end{proof}

  \begin{proposition}
   \label{prop:gaussums}
   Let $\eta$ be a de Rham character of $G_{\Qp}$ with values in $L$. Write $\eta = \eta_0\eta_1 \chi^j$ for some finite-order character $\eta_0$ of $\Gamma$ of conductor $n$, some unramified character $\eta_1$, and some $j \in \ZZ$. Then we have
   \[ \varepsilon_L(\Dpst(L(\eta)), \xi) = \eta_1(\sigma)^{-n} p^{-nj} \tau(\eta_0, \xi),\]
   where $\sigma$ denotes the arithmetic Frobenius of $\Gal(\Qpnr / \Qp)$ and the Gauss sum $\tau(\eta_0, \xi)$ is defined by
   \[ \tau(\eta_0, \xi) \coloneqq \sum_{\gamma \in \Gamma / \Gamma_n} \eta_0(\sigma)^{-1} \xi_n^\sigma.\]
  \end{proposition}

  \begin{proof}
   From the previous proposition, the action of $W_{\Qp}$ on $\Dpst(L(\eta))$ is given by the character of $W_{\Qp}$ whose restriction to $\Gal(\Qpb / \Qpnr)$ coincides with $\eta_0$, and which takes the value $p^j \eta_1(\sigma)$ on the arithmetic Frobenius element of $\Gal(\Qpnr(\mu_{p^\infty}) / \Qp(\mu_{p^\infty}))$.

   Thus we may apply property (7) of local $\varepsilon$-factors in \cite[\S 3.2.2]{fukayakato06} to see that\footnote{There appears to be a minor error in \cite{fukayakato06} in item (7) of \S 3.2.2: the factor $\chi(\tau)^n$ is not well-defined, since $\tau$ is chosen as an arbitrary element of $W(\Qpb/\Qp)$ such that $v(\tau) = 1$, so $\tau$ is only determined up to multiplication by an element of $I(\Qpb/\Qp)$; one needs to assume that $\tau$ acts trivially on $\QQ_{p, n}$. With this modification, the formula is correct modulo a sign error: the factor $\chi(\tau)^n$ should be $\chi(\tau)^{-n}$, as one sees by comparison with (5).}
   \[ \varepsilon_L(\Dpst(L(\eta)), \xi) = \eta_1(\sigma)^{-n} p^{-nj} \tau(\eta_0, \xi) = \eta_1(\sigma)^{-n} p^{-nj} \varepsilon_L(L(\eta_0), \xi).\]
  \end{proof}

  We shall write $\varepsilon_L(\eta, \xi)$ for $\varepsilon_L(\Dpst(L(\eta), \xi))$; this should not cause confusion, since $\varepsilon_L(\Dpst(L(\eta), \xi))$ agrees with $\varepsilon_L(L(\eta), \xi)$ whenever the latter is defined.

  \begin{proposition}
   Let $V$ be a $d$-dimensional $L$-linear crystalline representation of $G_{\Qp}$, and let $\eta$ be a de Rham character of $G_{\Qp}$ with values in $L$. Write $\eta = \eta_0\eta_1 \chi^j$ as in the previous proposition. Then we have
   \[ \varepsilon_L(\Dpst(V(\eta)), \xi) = \varepsilon_L(\eta, \xi)^d
   \cdot \det_L\left(\vp : \Dcris(V) \to \Dcris(V)\right)^n.\]
  \end{proposition}

  \begin{proof}
   Using the property (5) of $\varepsilon$-factors stated in \cite[\S 3.2.2]{fukayakato06}, we have
    \begin{align*}
    \varepsilon_L(\Dpst(V(\eta)), \xi) &= \varepsilon_L(\Dpst(V) \otimes_{\Qpnr \otimes L} \Dpst(L(\eta)), \xi)\\
    &= \varepsilon_L(\Dpst(L(\eta)), \xi)^d \cdot \det_{(\Qpnr \otimes L)}(\sigma^{-1}: \Dpst(V) \to \Dpst(V))^n\\
    &= \varepsilon_L(\Dpst(L(\eta)), \xi)^d \cdot \det_L(\vp: \Dcris(V) \to \Dcris(V))^n
   \end{align*}
   since the arithmetic Frobenius $\sigma$ on $\Dpst(V)$ acts as the inverse of the crystalline Frobenius $\vp$ on $\Dcris(V)$.
  \end{proof}

 \subsection{Epsilon-isomorphisms for de Rham representations}
  \label{sect:deRhamepsilon}
   Let $V$ be a de Rham representation of $G_{\Qp}$ with coefficients in a finite extension $L / \Qp$. Let $\Lt = L \otimes_{\Qp} \widehat{\Qp^{\mathrm{nr}}}$, and let $\xi = (\xi_n)_{n \ge 0}$ be a compatible system of $p$-power roots of unity as before.

  Then Fukaya and Kato have shown in \cite[\S 3.3]{fukayakato06} how to construct a canonical isomorphism
  \[ \varepsilon_{L, \xi}(V): \Det_{\Lt}(0) \rTo^\cong \Lt\otimes_L \left\{ \Det_L(R\Gamma(\Qp, V)) \cdot \Det_L(V)\right\}.\]

  This isomorphism is defined as a product of three terms
  \begin{equation}\label{eq:epsilon}
   \varepsilon_{L, \xi}(V) = \Gamma_L(V) \cdot \varepsilon_{L, \xi, \dR}(V) \cdot \theta_L(V),
   \end{equation}
  where
  \begin{gather*}
   \theta_L(V) : \Det_L(0) \rTo^\cong \Det_L(R\Gamma(\Qp, V)) \cdot \Det_L(\DdR(V)),\\
   \varepsilon_{L, \xi, \dR}(V) : \Lt \otimes_L \Det_L(\DdR(V)) \rTo^\cong \Lt \otimes_{L} \Det_L(V),\\
   \Gamma_L(V) \in \QQ^\times.
  \end{gather*}

  As it will be important for the remainder of the present paper, let us recall in detail the definitions of these terms.

  Firstly, we define $\Gamma_L(V)$, which depends only on the Hodge--Tate weights of $V$. For $r \in \ZZ$ let
  \[ n(r) = \dim_L \gr^{-r} \DdR(V),\]
  so $n_r$ is the multiplicity of $r$ as a Hodge--Tate weight\footnote{We adopt the convention in this paper that the Hodge--Tate weight of the cyclotomic character is 1.} of $V$. We define
  \[ \Gamma^*(r) = \begin{cases} (r-1)! & \text{if $r > 0$,}\\ \frac{(-1)^r}{(-r)!} & \text{if $r \le 0$,}\end{cases}\]
  the leading coefficient of the Taylor series of $\Gamma(s)$ at $s = r$.
  Then
  \[ \Gamma_L(V) = \prod_{r \in \ZZ} \Gamma^*(r)^{-n(r)}.\]

  Secondly, we define $\varepsilon_{L,\xi,\dR}(V)$. Let $\varepsilon_L(\DD_{\mathrm{pst}}(V), \xi) \in \Qpb \otimes_{\Qp} L$ be the $\varepsilon$-factor of $\Dpst(V)$, as defined in the previous section. Then \( \varepsilon_L(\DD_{\mathrm{pst}}(V), \xi) \) clearly lies in $\Qpi \otimes_{\Qp} L$, and it transforms under $\Gamma$ via
  \[ \sigma \cdot \varepsilon_L(\DD_{\mathrm{pst}}(V), \xi) = \varepsilon(\DD_{\mathrm{pst}}(V), \xi^\sigma) = \eta(\tilde\sigma) \varepsilon_L(\DD_{\mathrm{pst}}(V), \xi),\]
  where $\eta$ is the finitely-ramified character by which $G_{\Qp}^{\mathrm{ab}}$ acts on $\det(V)(-w)$, where $w =\sum n_i$, and $\tilde\sigma$ is the unique lifting of $\sigma$ to $\Gal(\Qpnr(\mu_{p^\infty}) / \Qpnr)$.

  If we let $t$ denote the element $\log([\xi])$ of $\BdR$, then it follows that multiplying by $t^w \varepsilon_L(\DD_{\mathrm{pst}}(V), \xi)$ defines an isomorphism $\Lt \otimes_L \Det_L(\DdR(V)) \to \Lt\otimes_L \Det_L(V)$
  (regarding both as submodules of $\BB_{\dR} \otimes_{\Qp} \Det(V)$), and we take this to be the definition of $\varepsilon_{L, \xi, \dR}(V)$.

  Thirdly, we define $\theta_L(V)$. The general definition is rather complicated. Let $C(\Qp, V)$ be the complex of continuous cochains with values in $V$ (so $R\Gamma(\Qp, V)$ is the image of $C(\Qp, V)$ in the derived category). Then $C_f(\Qp, V)$ is a certain subcomplex of $C(\Qp, V)$ which is nonzero only in degrees 0 and 1, and whose cohomology in degree $0$ is $H^0(\Qp, V)$ and in degree 1 is $H^1_f(\Qp, V)$. Hence
  \[ \Det_L C_f(\Qp, V) = \left\{\Det_L H^0(\Qp, V)\right\} \cdot \left\{ \Det_L H^1_f(\Qp, V) \right\}^{-1}.\]
  The fundamental exact sequence
  \begin{equation}
   \label{eq:fundseq}
   0 \to H^0(\Qp, V) \to \Dcris(V) \rTo^{(1 - \vp, 1)} \Dcris(V) \oplus t(V) \to H^1_f(\Qp, V) \to 0
  \end{equation}
  gives rise to a quasi-isomorphism
  \[ C_f(\Qp, V) \cong [ \Dcris(V) \rTo^{(1-\vp, 1)} \Dcris(V) \oplus t(V)], \]
  where $t(V)$, the ``tangent space'' of $V$, is defined as $\DdR(V) / \Fil^0 \DdR(V)$. This gives an isomorphism of determinants
  \[ \eta(\Qp, V) : \Det_L(0) \to \left\{ \Det_L C_f(\Qp, V)\right\} \cdot \left\{ \Det_L t(V)\right\}.\]
  We also have a corresponding isomorphism $\eta(\Qp, V^*(1))$. Furthermore, there is an isomorphism
  \[ \Psi_f(\Qp, V) : C_f(\Qp, V) \cong \left( C(\Qp, V^*(1)\right) / \left( C_f(\Qp, V^*(1)) \right)^*[-2].\]
  On homology groups this says that
  \begin{gather*}
   H^0(\Qp, V)^* = H^2(\Qp, V^*(1))\\
   H^1_f(\Qp, V)^* =\frac{ H^1(\Qp, V^*(1))} {H^1_f\left(\Qp, V^*(1)\right).}
  \end{gather*}
  Also, there is a canonical exact sequence
  \begin{equation}
  \label{tangent}
   0 \rTo t(V^*(1))^* \rTo \DdR(V) \rTo t(V) \rTo 0,\end{equation}
  arising from the compatibility of the functor $\DdR(-)$ with tensor products and the canonical isomorphism $\DdR(L(1)) \cong L$. Putting all of these together, we have an isomorphism
  \[ \theta_L(V) : \Det_L(0) \rTo \Det_L R\Gamma(\Qp, V) \cdot \Det_L \DdR(V)\]
  defined by
  \[ \eta(\Qp, V) \cdot \left\{\eta(\Qp, V^*(1))^*\right\}^{-1} \cdot \left\{ \Det_L \Psi_f(\Qp, V^*(1))^*\right\}^{-1}.\]

  This completes the definition of the Fukaya--Kato $\varepsilon$-isomorphism $\varepsilon_{L, \xi}(V)$ for a de Rham representation $V$.

  For future reference, we observe that the objects defined in this section are well-behaved in short exact sequences:

  \begin{lemma}
   \label{lem:SEScompatible}
   Suppose that we have a short exact sequence
   \[ 0\rTo V'\rTo V\rTo V''\rTo 0\]
   of crystalline representations of $G_{\Qp}$ with coefficients in a finite extension $L$ of $\Qp$.
   Then
   \begin{itemize}
    \item $\varepsilon_{L,\xi,\dR}(V)=\varepsilon_{L,\xi,\dR}(V')\cdot\varepsilon_{L,\xi,\dR}(V'')$,
    \item $\theta_L(V)=\theta_L(V')\cdot\theta_L(V'')$,
    \item $\Gamma_L(V)=\Gamma_L(V')\cdot\Gamma_L(V'')$.
   \end{itemize}
  \end{lemma}
  \begin{proof}
   The equalities $\varepsilon_{L,\xi,\dR}(V)=\varepsilon_{L,\xi,\dR}(V')\cdot\varepsilon_{L,\xi,\dR}(V'')$ and $\Gamma_L(V)=\Gamma_L(V')\cdot\Gamma_L(V'')$ are true by construction.
   Recall that $\theta_L(V)$ is defined as
   \[ \eta(\Qp,V)\cdot\{\eta(\Qp,V^*(1))^*\}^{-1}\cdot \{\Det_L\Psi_f(\Qp,V^*(1))^*\}^{-1}.\]
   It follows from tedious diagram chasing that each of the factors is multiplicative in short exact sequences, which implies the result for $\theta_L(\Qp,V)$.
  \end{proof}

  \begin{remark}
   The statements for $\varepsilon_{L,\xi,\dR}(V)$ and $\Gamma_L(V)$ are clearly also true for de Rham representations. However, the statement for $\theta_L(V)$ is less clear in this generality, since the sequence
   \[ 0\rTo \Dcris(V')\rTo \Dcris(V)\rTo \Dcris(V'')\rTo 0\]
   is not necessarily exact in this case. Of course, if all of the $V$'s are purely non-crystalline (in the sense that their $\Dcris$ is zero) then we evidently have such an exact sequence, and the result holds in this case too.
  \end{remark}

  As a corollary, we obtain the following result:

  \begin{proposition}\label{prop:SEScompatible}
   Under the same assumptions as in Lemma \ref{lem:SEScompatible}, we have
   \[ \varepsilon_{L,\xi}(V)=\varepsilon_{L,\xi}(V')\cdot\varepsilon_{L,\xi}(V'').\]
  \end{proposition}
  \begin{proof}
   Immediate from Lemma \ref{lem:SEScompatible} and \eqref{eq:epsilon}.
  \end{proof}


 \subsection{Two special cases}
  \label{sect:specialcases}

  In this section we'll give a different description of the isomorphism $\theta_L(V)$ when $V$ is either purely crystalline or purely non-crystalline, which are the only two cases we shall need to consider.

  Recall that the dual exponential map $\exp^*_{\Qp,V^*(1)}: H^1(\Qp, V) \rTo \Fil^0 \DdR(V)$ is defined by the commutativity of the following diagram
  \begin{equation}
  \label{defexptranspose}
  \xymatrix{
    H^1(\Qp,V) \ar[d]_{\exp^*_{\Qp,V^*(1)}}  \ar@{}[r]|{\times}  & H^1(\Qp,V^*(1))   \ar[rr]^{ (\;,\;)_V} && L\phantom{.} \ar@{=}[d]^{ } \\
    {\DdR(V)}  \ar@{}[r]|{\times} & \DdR(V^*(1))\ar[u]_{\exp_{\Qp,V^*(1)}} \ar[rr]^{[\;,\;]_V} && L.   }\end{equation} The involved pairings induce isomorphisms
  \begin{gather*}
   \psi_V:H^1(\Qp,V)\to H^1(\Qp,V^*(1))^*,\;\;\;h\mapsto (h,-)_V,\\
   \psi_{V,/f}:H^1(\Qp,V)/H^1_f(\Qp,V)\to H^1_f(\Qp,V^*(1))^*,\;\;\;[h]\mapsto (h,-)_V,\\
   \psi_{V,f}:H^1(\Qp,V^*(1))/H^1_f(\Qp,V^*(1))\to H^1_f(\Qp,V)^*,\;\;\;[h]\mapsto (-,h)_V,\\
   \psi_{\DdR(V)}:\DdR(V)\to \DdR(V^*(1))^*,\;\;\;d\mapsto [d,-]_V,\\
  \end{gather*}
  and
  \[
   \psi_{\Fil^0\DdR(V)}:\Fil^0 \DdR(V)\to t(V^*(1))^*,\;\;\;d\mapsto [d,-]_V,.
  \]
  Hence, $\exp^*_{\Qp,V^*(1)}$ is the composite
  \begin{equation}
  \label{exptranspose}
  H^1(\Qp,V)/H^1_f(\Qp,V)\rTo^{\psi_{V,/f}} H^1_f(\Qp,V^*(1))^*\rTo^{\left(\exp_{\Qp,V^*(1)}\right)^*}\DdR(V^*(1))^*\rTo^{\psi_{\Fil^0\DdR(V)}^{-1}} \Fil^0\DdR(V).
  \end{equation}

  \begin{proposition}
   \label{prop:noncrystallinedet}
   If $V$ is a de Rham representation such that $\Dcris(V) = \Dcris(V^*(1)) = 0$, then we have
   \begin{align*}
    H^0(\Qp, V) & = H^2(\Qp, V) = 0,\\
     H^1_e(\Qp, V) & = H^1_f(\Qp, V) = H^1_g(\Qp, V),
     \end{align*}
   and the morphism $\theta_L(V)$ is given by
   \begin{multline*}
       \Det_L \left( \log_{\Qp, V} : H^1_f(\Qp, V) \rTo^\cong \frac{\DdR(V)}{\Fil^0 \DdR(V)}\right) \\ \cdot \Det_L \left(-\exp^*_{\Qp, V^*(1)}: \frac{H^1(\Qp, V)}{H^1_f(\Qp, V)} \rTo^\cong \Fil^0 \DdR(V)\right).
   \end{multline*}
  \end{proposition}

  \begin{proof}
   Since the complex $C_f(\Qp, V)$ is reduced to $H^1_f(\Qp, V)[1]$, the isomorphism $\eta(\Qp, V)$ is the map induced by $\log_{\Qp, V}$. Similarly, $ \eta(\Qp, V^*(1))^* \cdot \Det_L \Psi_f(\Qp, V^*(1))^* $ is induced by the map
   \[\psi_{\Fil^0\DdR(V)}^{-1}\circ\left(\exp_{\Qp,V^*(1)}\right)^*\circ \psi_{V^*(1),f}=- \exp^*_{\Qp,V^*(1)}, \]
   where the last equality follows from comparison with \eqref{exptranspose} and taking into account that $\psi_{V^*(1),f}=-\psi_{V,/f}$ by the skew-symmetry of the cup-product.
  \end{proof}

  \begin{proposition}\label{prop:nophidet}
   If $V$ is a crystalline representation such that $\Dcris(V)^{\vp = 1} = \Dcris(V)^{\vp = p^{-1}} = 0$, then we again have
   $H^0(\Qp, V) = H^2(\Qp, V) = 0, H^1_e(\Qp, V) = H^1_f(\Qp, V) = H^1_g(\Qp, V)$, and the morphism $\theta_L(V)$ is given by
   \begin{multline*}
       \Det_L\left( (1 - \vp)(1 - p^{-1}\vp^{-1})^{-1} : \Dcris(V) \rTo \Dcris(V) \right) \\ \cdot \Det_L \left( \log_{\Qp, V} : H^1_f(\Qp, V) \rTo^\cong \frac{\Dcris(V)}{\Fil^0 \Dcris(V)}\right) \\ \cdot \Det_L \left(-\exp^*_{\Qp, V^*(1)}: \frac{H^1(\Qp, V)}{H^1_f(\Qp, V)} \rTo^\cong \Fil^0 \Dcris(V)\right).
   \end{multline*}
  \end{proposition}

  \begin{proof}
   Since $(1 - \vp)$ is invertible on $\Dcris(V)$, there is a natural morphism of complexes between $C_f(\Qp, V)$ and the corresponding complex with $(1 - \vp, 1)$ replaced by $(1-\vp, 0)$, which is a chain homotopy and hence induces an isomorphism on determinants; and similarly for $C_f(\Qp, V^*(1))$. The proof now concludes via the same argument as in Proposition \ref{prop:noncrystallinedet}.
  \end{proof}

  To handle the bad cases when $V$ is crystalline but the hypotheses of Proposition \ref{prop:nophidet} are not satisfied, it will be convenient to introduce a slight modification of the exponential and dual-exponential maps.

  \begin{definition}
   For $V$ a crystalline representation, let
   \[ \widetilde\exp_{\Qp, V} : \frac{\Dcris(V)}{(1 - \vp)\Fil^0 \Dcris(V)} \rTo^\cong H^1_f(\Qp, V)\]
   be the map obtained by restricting the boundary map of the fundamental exact sequence \eqref{eq:fundseq} to the summand $\Dcris(V) \subseteq \Dcris(V) \oplus t(V)$, while we write $\widetilde\log_{\Qp, V}$ for its inverse.
  \end{definition}

  \begin{remark}
   The map
   \[ \widetilde\exp_{\Qp, V} \oplus \exp_{\Qp, V} : \Dcris(V) \oplus t(V) \rTo H^1_f(\Qp, V)\]
   (which is just the boundary map of \eqref{eq:fundseq}) is the map denoted by $\exp_{V, f}$ in \cite{perrinriou99} (see p.~231).
  \end{remark}

  It is clear that the kernel of $\widetilde\exp_{\Qp, V}$ is exactly the subspace $(1 - \vp) \Fil^0 \Dcris$, and we have
  \begin{equation}
   \label{eq:signexptilde}
   \widetilde\exp_{\Qp, V} \circ (1 - \vp) = -\exp_{\Qp, V}.
  \end{equation}
  However, $\widetilde\exp_{\Qp, V}$ may be non-trivial even when $\exp_{\Qp, V}$ is the zero map, as the following example shows:

  \begin{proposition}
   If $V = L$ is the trivial representation, then for any $x \in V = \Dcris(V)$, $\widetilde\exp_{\Qp, V}(x)$ is the element of $\Hom(G_{\Qp}, L)$ which is trivial on inertia and maps the arithmetic Frobenius to $-x$.
  \end{proposition}

  \begin{proof}
   It suffices to assume $x = 1$. By Hensel's lemma there exists $y \in \widetilde\cO = W(\overline{\FF}_p)$ such that $(1 - \vp)y = 1$, where $\vp$ is the arithmetic Frobenius. Since $\widetilde\cO$ is $(\vp, G_{\Qp})$-equivariantly a submodule of $\Bcris^+$, the class $\widetilde\exp_{\Qp, V}(x)$ is given by $\sigma \mapsto (\sigma - 1) y$, which is clearly unramified and maps the arithmetic Frobenius to $-1$.
  \end{proof}

  We define
  \[ \widetilde\exp^*_{\Qp, V^*(1)} : \frac{H^1(\Qp, V)}{H^1_f(\Qp, V)} \rTo^\cong (1 - p^{-1}\vp^{-1})^{-1} \Fil^0 \Dcris(V)\]
  to be the transpose of $\widetilde{\exp}_{\Qp, V^*(1)}$, by the analogue of diagram \eqref{defexptranspose}. By construction, we have
  \begin{equation}
   \label{eq:signexptildestar}
   \exp^*_{\Qp, V^*(1)} = -(1 - p^{-1}\vp^{-1})\, \widetilde\exp^*_{\Qp, V^*(1)},
  \end{equation}
  since $(1 - p^{-1} \vp^{-1})$ is the transpose of $(1 - \vp)$.

  \begin{definition}
   We define the following subspaces of $H^1(\Qp, V)$:
   \begin{itemize}
    \item The space $H^1_a(\Qp, V)$ is defined by
    \[ H^1_a(\Qp, V) = \left\{ x \in H^1(\Qp, V) : \widetilde\exp^*_{\Qp, V^*(1)}(x) \in \Dcris(V)^{\vp = 1}\right\}.\]
    By construction, this contains the kernel of $\widetilde\exp^*_{\Qp, V^*(1)}$, which is $H^1_f(\Qp, V)$.
    \item The space $H^1_b(\Qp, V)$ is defined by
    \[ H^1_b(\Qp, V) = \left\{ x \in H^1_f(\Qp, V): \widetilde\log_{\Qp, V}(x) \in \frac{(1 - p^{-1}\vp^{-1})\Dcris(V) + (1 - \vp)\Fil^0 \Dcris(V)}{(1 - \vp)\Fil^0 \Dcris(V)}\right\}.\]
   \end{itemize}
  \end{definition}

  We now note that $\widetilde\exp^*_{\Qp, V^*(1)}$ defines an isomorphism
  \begin{align*}
   \frac{H^1_a(\Qp, V)}{H^1_f(\Qp, V)} \rTo^\cong& \Dcris(V)^{\vp = 1} \cap (1 - p^{-1}\vp^{-1})^{-1} \Fil^0 \Dcris(V)\\
   =& \Dcris(V)^{\vp = 1} \cap \Fil^0 \Dcris(V)\\
   =& H^0(\Qp, V).
  \end{align*}
  On the other hand $\widetilde\log_{\Qp, V}$ gives an isomorphism
  \[ \frac{H^1_f(\Qp, V)}{H^1_b(\Qp, V)} \rTo^\cong \frac{\Dcris(V)}{(1 - p^{-1}\vp^{-1})\Dcris(V) + (1 - \vp)\Fil^0 \Dcris(V)}\]
  which Tate duality identifies with $H^0(\Qp, V^*(1))^* = H^2(\Qp, V)$.

  We therefore have the following isomorphisms:
  \begin{subequations}
  \begin{align}
   \label{iso1}
   -(1 - \vp)\widetilde\exp^*_{\Qp, V^*(1)} &:& \dfrac{H^1(\Qp, V)}{H^1_a(\Qp, V)} \rTo^\cong & (1 - \vp)(1 - p^{-1}\vp^{-1})^{-1} \Fil^0 \Dcris(V)\\
   \label{iso2}
   -\widetilde\exp^*_{\Qp, V^*(1)} & : & \dfrac{H^1_a(\Qp, V)}{H^1_f(\Qp, V)} \rTo^\cong & H^0(\Qp, V) \\
   \label{iso3}
   \widetilde\log_{\Qp, V} & : & \dfrac{H^1_f(\Qp, V)}{H^1_b(\Qp, V)} \rTo^\cong & H^2(\Qp, V) \\
   \label{iso4} (1 - p^{-1} \vp^{-1})^{-1} \widetilde\log_{\Qp, V} & : & H^1_b(\Qp, V) \rTo^\cong & \frac{\Dcris(V)}{(1 - p^{-1} \vp^{-1})^{-1} (1 - \vp) \Fil^0 \Dcris(V)}.
  \end{align}
  \end{subequations}

  For future use we note the following corollary of these isomorphisms:

  \begin{corollary}
   \label{corr:dimsv}
   Let $V$ be crystalline. Let $s(V)$ denote the quotient of $\Dcris(V)$ given by
   \[ s(V) \coloneqq \frac{\Dcris(V)}{(1 - \vp)(1 - p^{-1} \vp^{-1})^{-1} \Fil^0 \Dcris(V)}.\]
   Then we have
   \[ \dim s(V) = \dim t(V) + \dim H^0(\Qp, V) - \dim H^2(\Qp, V),\]
   where $t(V) = \DdR(V) / \Fil^0 \DdR(V) \cong \Dcris(V) / \Fil^0 \Dcris(V)$ as above.
  \end{corollary}

  \begin{proof}
   From the isomorphisms \eqref{iso1} and \eqref{iso2}, we have
   \[ \dim s(V) + \dim H^2(\Qp, V) = \dim H^1_f(\Qp, V).\]
   However, the exactness of the sequence \eqref{eq:fundseq} shows that
   \[ \dim H^1_f(\Qp, V) = \dim t(V) + \dim H^0(\Qp, V).\]
   Combining these gives $\dim s(V) = \dim t(V) + \dim H^0(\Qp, V) - \dim H^2(\Qp, V)$ as required.
  \end{proof}

  As promised above, we can use the isomorphisms \eqref{iso1}-\eqref{iso4} to give a simpler expression for $\theta_L(V)$.

  \begin{theorem}
   \label{thm:sigmas}
   The isomorphism
   \[ \Det_L(0) \rTo^\cong   \Det_L R\Gamma(\Qp, V)\cdot\Det_L \Dcris(V)\]
   defined by composing the determinants of \eqref{iso1}-\eqref{iso4}     coincides with the isomorphism $\theta_L(V)$ defined above up to the factor $(-1)^{\dim_L(V)} $. In particular, if  $H^0(\Qp, V) = H^2(\Qp, V) = 0$, i.e., $H^1_a(\Qp, V) = H^1_b(\Qp, V) = H^1_f(\Qp, V)$, then $\theta_L(V)$ coincides with the isomorphism
   defined by
   \begin{multline*}
         \Det\left[ (1 - \vp)\widetilde\exp^*_{\Qp, V^*(1)}: \dfrac{H^1(\Qp, V)}{H^1_f(\Qp, V)} \rTo^\cong (1 - \vp)(1 - p^{-1}\vp^{-1})^{-1} \Fil^0 \Dcris(V)\right]\\
    \cdot \Det\left[ -(1 - p^{-1} \vp^{-1})^{-1} \widetilde\log_{\Qp, V} : H^1_f(\Qp, V) \rTo^\cong \frac{\Dcris(V)}{(1 - p^{-1} \vp^{-1})^{-1} (1 - \vp) \Fil^0 \Dcris(V)} \right].
   \end{multline*}
       \end{theorem}

  \begin{remark}
   Note that this result extends Proposition \ref{prop:nophidet}, because of equations \eqref{eq:signexptilde} and \eqref{eq:signexptildestar}.
  \end{remark}

  As the proof of this theorem requires some rather elaborate diagram-chasing, we shall not give it here but relegate it to Appendix \ref{appendix:diagrams} below.

\section{Regulator maps}
 \label{sect:regulators}

 \subsection{The cyclotomic regulator map}

  For any de Rham character $\eta$ of $\Gamma$ (with values in $L$) we write $\pi_{\eta}: \cH(\Gamma)\to L$ for the $L$-algebra homomorphism which sends $g \in \Gamma$ to $\eta(g)$, and similarly for $\Lambda(\Gamma)$. Then we have the two projection maps
  \[ \operatorname{pr}_\eta:\cH(\Gamma) \otimes_{\Lambda(\Gamma)} H^1_{\Iw}(\Qp, V) \rTo L\otimes_{\Lambda(\Gamma),\pi_{\eta}} H^1_{\Iw}(\Qp, V)\rTo H^1(\Qp, V(\eta^{-1}))\]
  sending $h\otimes x$ to $\pi_{\eta}(h)\otimes x$ followed by the canonical map (arising from the associated spectral sequence) and where the tensor product is formed via $\pi_{\eta}$ as indicated;
  \[ \cH(\Gamma) \otimes_{\Qp} \Dcris(V) \rTo L\otimes_{\cH(\Gamma),\pi_{\eta}} \cH(\Gamma) \otimes_{\Qp} \Dcris(V)\cong L \otimes_{\Qp} \Dcris(V)\]
  sending $h\otimes d$ to $\pi_{\eta}(h)\otimes d$. We also write $z(\eta) $ for the image of $z\in \cH(\Gamma) \otimes_{\Qp} \Dcris(V)$ under the latter map.

  We now let $V$ be a crystalline representation, with all Hodge--Tate weights $\ge 0$. We shall write $H^1_{\Iw}(\Qpi, V)_0$ for the submodule of $H^1_{\Iw}(\Qpi, V)$ whose image under the Fontaine isomorphism
  \[ H^1_{\Iw}(\Qpi, V) \rTo^\cong \DD(V)^{\psi = 1}\]
  is contained in the Wach module $\NN(V)$. By the results of \cite[Appendix A]{berger03}, if $V$ has all Hodge--Tate weights $\ge 0$, then the quotient $H^1_{\Iw}(\Qpi, V) / H^1_{\Iw}(\Qpi, V)_0$ is identified with
  \[ \frac{\left(\pi^{-1} \NN(V)\right)^{\psi = 1}}{\NN(V)^{\psi = 1}} \hookrightarrow \Dcris(V)^{\vp = 1}(-1),\]
  and in particular if $V$ has no quotient isomorphic to $\Qp$, then $H^1_{\Iw}(\Qpi, V)_0 = H^1_{\Iw}(\Qpi, V)$.

  \begin{theorem}[Perrin-Riou, Berger]
   \label{thm:cycloregulator}
   Let $V$ be a crystalline representation of $G_{\Qp}$ with all Hodge--Tate weights $\ge 0$. Then there is a homomomorphism of $\Lambda_L(\Gamma)$-modules
   \[ \cL^\Gamma_{V, \xi}: H^1_{\Iw}(\Qpi, V)_0 \rTo \cH_{L}(G) \otimes_{L} \Dcris(V) \]
   whose values at de Rham characters $\eta$ are given by the following formulae. Let $W = V(\eta^{-1})$, and let $\eta = \chi^j \eta_0$ with $j \in \ZZ$ and $\eta_0$ a finite-order character of conductor $n$.
   \begin{enumerate}[(1)]
    \item If $\eta_0$ is non-trivial, with conductor $n \ge 1$, then
    \[ \cL^\Gamma_{V, \xi}(x)(\eta) = \Gamma^*(1 + j) \varepsilon_L(\eta^{-1}, -\xi) \vp^n
     \begin{cases}
      \exp^*_{\Qp, W^*(1)}(x_\eta) \otimes t^{-j} e_j & \text{if $j \ge 0$,} \\
      \log_{\Qp, W}(x_\eta) \otimes t^{-j} e_j & \text{if $j \le -1$.}
     \end{cases}
    \]
    \item If $\eta_0$ is trivial, so $\eta = \chi^j$, then
    \[
     (1 - p^{-1-j} \vp^{-1}) \cL^\Gamma_{V, \xi}(x)(\eta) = \Gamma^*(1 + j) (1 - p^j \vp)
      \begin{cases}
       \exp^*_{\Qp, W^*(1)}(x_\eta) \otimes t^{-j} e_j & \text{if $j \ge 0$,} \\
       \log_{\Qp, W}(x_\eta) \otimes t^{-j} e_j & \text{if $j \le -1$.}
     \end{cases}
    \]
   \end{enumerate}
  \end{theorem}

  In the above theorem $e_j$ denotes the basis of $\Qp(j)$ given by $\left[(\xi_n)_{n \ge 1}\right]^{\otimes j}$, and $t$ the element $\log([\xi])$ of $\Bcris$. (Thus both $t$ and $e_j$ depend on the choice of $\xi$, but $t^{-j} e_j \in \Dcris(\Qp(j))$ does not.)

  \begin{proof}
   Well-known; for a proof of the special value formulae in this form see e.g.~\cite[Appendix B]{loefflerzerbes11}.
  \end{proof}

  The presence of the factors $(1 - p^j \vp)$ and $(1 - p^{-1-j} \vp^{-1})$ is awkward for our present purposes, since they may fail to be invertible. We will therefore use the following strengthened version of the formulae of Theorem \ref{thm:cycloregulator}, using the map $\widetilde\exp^*$ introduced in \S \ref{sect:specialcases} above.

  \begin{theorem}
   \label{thm:strongcycloregulator}
   For any $j \ge 0$ we have
   \[ \cL^\Gamma_{V, \xi}(x)(\chi^j) = -\Gamma^*(1 + j) (1 - p^j \vp) \left[\widetilde\exp^*_{\Qp, V^*(1+j)}(x_{\chi^j}) \otimes t^{-j} e_{j}\right]. \]
  \end{theorem}

  \begin{proof}
   See Theorem \ref{thm:strongcycloregulator2} in the appendix.
  \end{proof}

  From the above results we see that if $\eta$ has Hodge--Tate weight $\ge 0$, then $\cL^\Gamma_{V, \xi}(x)(\eta) = 0$ if and only if $x_\eta \in H^1_f(\Qp, V(\eta^{-1}))$ when $\eta$ is non-crystalline, and if and only if $x_\eta \in H^1_a(\Qp, V(\eta^{-1}))$ when $\eta$ is crystalline. In this case we have a formula for the derivative of $\cL^\Gamma_{V, \xi}(x)$ at $\eta$:

  \begin{theorem}
   \label{thm:regderiv}
   Suppose $x \in H^1_{\Iw}(\Qpi, V)_0$ satisfies $\cL^{\Gamma}_V(x)(\eta) = 0$, where $\eta$ has Hodge--Tate weight $j \ge 0$ and conductor $n$.
   \begin{enumerate}[(1)]
    \item If $n \ge 1$, then we have
    \[ \cL^{\Gamma}_V(x)'(\eta) = \Gamma^*(1 + j) \varepsilon_L(\eta^{-1}, -\xi) \vp^{n} \left[\log_{\Qp, V(\eta^{-1})}(x_\eta) \otimes t^{-j} e_{j}\right] \pmod{ \vp^n \Fil^{-j} \Dcris(V)}.\]
    \item If $n = 0$, so $\eta = \chi^j$, and $H^0(\Qp, V(-j)) = 0$, then we have
    \begin{multline*}
     \cL^\Gamma_{V, \xi}(x)'(\eta) = -\Gamma^*(1 + j) (1 - p^{-1-j} \vp^{-1})^{-1}\left[\widetilde\log_{\Qp, V(-j)}(x_\eta) \otimes t^{-j} e_{j}\right] \\
     \pmod{ (1 - p^{-1-j} \vp^{-1})^{-1} (1 - p^j \vp) \Fil^{-j} \Dcris(V)}
    \end{multline*}
   \end{enumerate}
  \end{theorem}

  We note also the following twist-compatibility property:

  \begin{proposition}
   \label{prop:cyclo regulator twist invariance}
   The regulator maps for $V$ and $V(1)$ are related by
   \[ \mathcal{L}^\Gamma_{V(1), \xi}(x \otimes e_1) = \ell_0 \cdot \left( \mathrm{Tw}_{\chi^{-1}}(\mathcal{L}^\Gamma_{V, \xi}(x))\otimes t^{-1} e_1 \right).\]
  \end{proposition}

  \begin{proof}
   Recall that the regulator is defined by
   \[ \cL^\Gamma_{V, \xi}(y) = \mathfrak{M}^{-1}( (1 - \vp) y),\]
   where $\mathfrak{M}$ is the Mellin transform $\cH(\Gamma) \to (\Brig)^{\psi = 0}$, and
   \[ y \in \left(\Brig \otimes \Dcris(V)\right)^{\psi = 1} \cong H^1_{\Iw}(\Qpi, V)_0.\]

   We have the identity
   \[ \ell_0 (f) = t \partial(f)\]
   for $f \in \Brig$, where $\partial$ is the differential operator $(1 + \pi) \tfrac{\mathrm{d}}{\mathrm{d}\pi}$. Consequently we have
   \[ x \otimes e_1 = \ell_0 (\partial^{-1} x) \otimes t^{-1} e_1 \]
   for any $x \in (\Brig \otimes \Dcris(V))^{\psi = 0}$. But $\partial$ corresponds under $\mathfrak{M}$ to $\operatorname{Tw}_{\chi}$, so applying $\mathfrak{M}^{-1}$ to both sides of the above we have
   \[ \mathfrak{M}^{-1}(x \otimes e_1) = \ell_0 \operatorname{Tw}_{\chi^{-1}}\left(\mathfrak{M}^{-1}(x)\right) \otimes t^{-1} e_1.\]
   Letting $x = (1 - \vp) y$ for $y \in \left(\Brig \otimes \Dcris(V)\right)^{\psi = 1}$ gives the claimed formula.
  \end{proof}

  \begin{remark}
   One can also see this twist-compatibility as a consequence of the evaluation formulae above, since the two sides of Proposition \ref{prop:cyclo regulator twist invariance} must agree under evaluation at $\chi^j$ for all but finitely many $j \in \ZZ$ by Theorem \ref{thm:cycloregulator}.
  \end{remark}

 \subsection{The two-variable regulator map}

  We now recall the main result from \cite{loefflerzerbes11}. Let $F$ be any finite unramified extension of $\Qp$, and let $F_\infty$ be the unramified $\Zp$-extension of $F$. We set $K_\infty = F_\infty(\mu_{p^\infty})$ and $G = \Gal(K_\infty / \Qp)$. We regard $\Gamma$ as a subgroup of $G$, by identifying it with $\Gal(K_\infty / F_\infty)$. Let $V$ be a crystalline $L$-linear representation of $G_{\Qp}$ such that $\dim_L V = d$. Let $\Lt = L \otimes_{\Qp} \widehat{\Qp^\mathrm{nr}}$, as above.

  \begin{theorem}
   Assume that $V$ is crystalline with Hodge-Tate weights $\geq 0$. Then there exists a regulator map
   \[ \cL^G_{V, \xi}: H^1_{\Iw}(K_\infty,V)\rTo \mathcal{H}_{\Lt}(G)\otimes\DD_{\cris}(V)\]
   such that for any finite extension $E / \Qp$ contained in $F_\infty$, we have a commutative diagram
   \[
    \begin{diagram}
     H^1_{\Iw}(K_\infty / \Qp, V) &\rTo^{\mathcal{L}_{V,\xi}^G} &\cH_{\Lt}(G) \otimes_{\Qp} \Dcris(V)\\
     \dTo & &\dTo \\
     H^1_{\Iw}(E(\mu_{p^\infty}), V )& \rTo^{\mathcal{L}_{V,\xi}^{G'}} & \cH_{\Lt}(G') \otimes_{\Qp} \Dcris(V).
    \end{diagram}
   \]
   Here $G' = \Gal(E(\mu_{p^\infty}) / \Qp)$, the right-hand vertical arrow is the map on distributions corresponding to the projection $G \rightarrow G'$, and the map $\mathcal{L}_V^{G'}$ is defined by
   \[ \mathcal{L}_{V,\xi}^{G'} = \sum_{\sigma \in \Gal(E / \Qp)} [\sigma] \cdot \mathcal{L}^{\Gamma}_{E, V,\xi}(\sigma^{-1} \circ x),\]
   where $ \mathcal{L}^{\Gamma}_{E, V,\xi}$ is the cyclotomic regulator map for $E(\mu_{p^\infty}) / E$. Moreover, the map $\cL^G_{V, \xi}$ is injective.
  \end{theorem}

  By abuse of notation we also write $\cL^G_{V, \xi}$ for the induced map with source $\mathcal{H}_{\Lt}(G)\otimes H^1_{\Iw}(K_\infty,V)$.

  We will need a twist-compatibility property which extends Proposition \ref{prop:cyclo regulator twist invariance} to the two-variable regulator map. To state this, we need to introduce a map relating $\Dcris$ for unramified twists:

  \begin{definition}
   \label{def:b_eta}
   If $\eta$ is a crystalline character of $G$, we let $b_\eta$ denote the unique isomorphism
   \[ \Lt \otimes_L \Dcris(V) \otimes_L L(\eta) \rTo^\cong \Lt \otimes_L \Dcris(V(\eta))\]
   such that extending scalars to $\Bcris$ gives a commutative diagram
   \begin{diagram}
    \Bcris \otimes_{\Qp} \Dcris(V) \otimes_L L(\eta) & \rTo^{b_\eta}_{\cong} & \Bcris \otimes_{\Qp} \Dcris(V(\eta))\\
    \dTo^{\operatorname{can}(V) \otimes 1} & & \dTo^{\operatorname{can}(V(\eta))}\\
    \Bcris \otimes_{\Qp} V \otimes_L L(\eta) & \rTo^{\times t^{-j}}_\cong & \Bcris \otimes_{\Qp} V(\eta).
   \end{diagram}
   where $\operatorname{can}(V)$ is the canonical isomorphism $\Bcris \otimes V \cong \Bcris \otimes \Dcris(V)$, and the bottom row denotes multiplication by $t^{-j}$ in $\Bcris$.
  \end{definition}

  The existence of such an isomorphism is not \emph{a priori} obvious, but it follows from the fact that we may write $\eta = \chi^j \eta_1$ for some unramified $\eta_1$, and the periods of $\eta_1$ lie in $\Lt \subseteq L \otimes \Bcris$ (see \cite[\S 4.3]{loefflerzerbes11}). Note that $b_\eta$ in fact depends on $\xi$ (since $\xi$ determines the cyclotomic period $t$).

  We can now state the twist-compatibility of $\cL^G_{V, \xi}$:

  \begin{proposition}
   \label{prop:regulator twist invariance}
   Let $V$ be crystalline with non-negative Hodge--Tate weights, and let $\eta$ be any crystalline $L$-valued character with Hodge--Tate weight $j \ge 0$. Then we have a $G$-equivariant commutative diagram:
   \begin{diagram}
    H^1_{\Iw}(K_\infty / \Qp, V) \otimes L(\eta) & \rTo^{\cL^G_{V, \xi} \otimes 1} & \cH_{\Lt}(G) \otimes \Dcris(V) \otimes L(\eta) \\
    \dTo^\cong & & \dTo_{a_\eta}\\
    H^1_{\Iw}(K_\infty / \Qp, V(\eta)) & \rTo^{(\ell_0 \cdots \ell_{j-1})^{-1} \cL^G_{V(\eta), \xi}} & \cH_{\Lt}(G) \otimes \Dcris(V(\eta))
   \end{diagram}
   Here the right-hand vertical map $a_\eta$ is given by
   \[ x \otimes y \otimes z \mapsto \mathrm{Tw}_{\eta^{-1}}(x) \otimes b_\eta(y \otimes z),\]
   where $b_\eta$ is as in Definition \ref{def:b_eta}.
  \end{proposition}

  \begin{remark}
   The diagram above is $G$-equivariant, if one equips $H^1_{\Iw}(K_\infty / \Qp, V) \otimes L(\eta)$ with the diagonal action of $G$, and $\cH_{\Lt}(G) \otimes \Dcris(V) \otimes L(\eta)$ with the action of $G$ given by
   \[ g \cdot (x \otimes y \otimes z) = ([g] x) \otimes y \otimes (\eta(g) z).\]
  \end{remark}

  \begin{proof}
   It suffices to consider two cases separately: the case where $\eta$ is an unramified character (so $j = 0$), and the case where $\eta = \chi^j$. In the unramified case, the statement to be proven is Proposition 4.13 of \cite{loefflerzerbes11}. In the case of a power of the cyclotomic character, the result follows from the twist-compatibility of the cyclotomic regulator (Proposition \ref{prop:cyclo regulator twist invariance}) and the compatibility of the cyclotomic and two-variable regulator maps.
  \end{proof}

  \begin{remark}
   We take the opportunity to correct an error in the preprint \cite{loefflerzerbes11}: in \S 4.5 of \emph{op.cit} the formula
   \[ \ell_0 \mathcal{L}^G_{V, \xi}(x)\otimes t^{-1}e_1 = \mathrm{Tw}_{\chi}(\mathcal{L}_{V(1)}(x\otimes e_1))\]
   is incorrect -- the $\ell_0$ should be $\ell_{-1} = \tw_{\chi} \ell_0$. The corrected version of the formula is equivalent to the case $\eta = \chi$ of the above proposition.
  \end{remark}

  We shall also need the following simple properties of the maps $\cL^G_{V, \xi}$:

  \begin{proposition}
   \label{prop:regulatorproperties}
   \mbox{~}
   \begin{enumerate}
    \item Let $c \in \Zp^\times$, and let $\gamma_c$ denote the unique element of $\Gamma$ such that $\chi(\gamma_c) = c$. Then we have
    \[ \cL^G_{V, c\xi} = [\gamma_c]^{-1} \cL^G_{V, \xi}.\]
    \item If $\vp_{\Lt}$ denotes the $L$-linear automorphism of $\Lt = L \otimes_{\Qp} \widehat{\Qpnr}$ given by extending scalars from the arithmetic Frobenius automorphism of $\widehat{\Qpnr}$, then we have
    \[ \left[\cL^G_{V, \xi}\right]^{\vp_{\Lt}} = [\sigma_p] \cL^G_{V, \xi}\]
    where $\sigma_p$ is the arithmetic Frobenius element of $\Gal(K_\infty / \Qpi)$.
   \end{enumerate}
  \end{proposition}

  \begin{proof}
   See Proposition 4.10 and Remark 4.17 of \cite{loefflerzerbes11}.
  \end{proof}

 \subsection{The matrix of the cyclotomic regulator}

  We now study the matrix of the cyclotomic regulator map $\cL^\Gamma_{V, \xi}$, for a crystalline $L$-linear representation $V$ of $G_{\Qp}$. We shall first define an element $\ell(V) \in \cH_{\Qp}(\Gamma)$, depending only on the Hodge--Tate weights of $V$:

  \begin{definition}\label{def:ellV}
   \mbox{~}
   \begin{enumerate}
    \item For $n \in \ZZ$, define the element $\mu_n \in \operatorname{Frac} \cH_{\Qp}(\Gamma)$ by
    \[ \mu_n =
     \begin{cases}
      \ell_0 \cdots \ell_{n-1} & \text{ if $n \ge 1$}\\
      1 & \text{if $n = 0$}\\
      \left(\ell_{-1} \cdots \ell_{n}\right)^{-1} & \text{if $n \le -1$.}
     \end{cases}
    \]
    \item For $V$ a Hodge--Tate representation of $G_{\Qp}$, with Hodge--Tate weights $n_1, \dots, n_d$, let
    \[ \ell(V) = \prod_{i = 1}^d \mu_{n_i}.\]
   \end{enumerate}
  \end{definition}

  Our goal is the following proposition:

  \begin{proposition}
   \label{prop:det of cyclo regulator}
   Let $V$ be any $d$-dimensional crystalline $L$-linear representation of $G_{\Qp}$. Let $y_1, \dots, y_d \in H^1_{\Iw}(\Qpi, V)$ be such that the quotient
   \[ Q = \frac{H^1_{\Iw}(\Qpi, V)}{\langle y_1, \dots, y_d\rangle_{\Lambda_{L}(\Gamma)}}\]
   is $\Lambda_L(\Gamma)$-torsion. Then for any $L$-basis $v_1, \dots, v_d$ of $\Dcris(V)^\vee$, the determinant of the matrix with $(i, j)$ entry $\langle \cL^\Gamma_{V, \xi}(y_i), v_j\rangle$ is equal to
   \[ \frac{\ell(V) f_Q}{f_{H^2_{\Iw}(\Qpi, V)}} \pmod{\cH_L(\Gamma)^\times},\]
   where $f_Q \in \Lambda_L(\Gamma)$ is any characteristic element of the torsion $\Lambda_L(\Gamma)$-module $Q$, and similarly for $f_{H^2_{\Iw}(\Qpi, V)}$.
  \end{proposition}

  \begin{remark}\mbox{~}
   \begin{enumerate}[(i)]
    \item Under the simplifying hypothesis that no eigenvalue of $\varphi$ on $\Dcris(V)$ is a power of $p$, which forces $H^2_{\Iw}(\Qpi, V) = H^1_{\Iw}(\Qpi, V)_{\mathrm{tors}} = 0$, this result is a consequence of \cite[Theorem D]{leiloefflerzerbes11}, which determines the elementary divisors of the matrix of $\cL^\Gamma_{V, \xi}$.
    \item One can also formulate a version of this statement with the regulator $\cL^\Gamma_{V, \xi}$ replaced by the Perrin-Riou exponential $\Omega_{V, h, \xi}$ for suitable $h$ (see Appendix \ref{sect:app-formulary}). In \cite{perrinriou94}, Perrin-Riou has calculated the determinant of $\Omega_{V, h, \xi}$, assuming her local reciprocity conjecture (later proved by Colmez \cite{colmez98}). We shall give an direct proof of the formula for $\cL^\Gamma_{V, \xi}$, rather than deducing it from the result for $\Omega_{V, h, \xi}$, but modulo this difference our argument follows Perrin-Riou's rather closely.
   \end{enumerate}
  \end{remark}

  We recall the well-known fact (cf. \cite[Proposition II.1.4]{colmez98} for example) that for any $p$-adic representation $V$ we have $H^2_{\Iw}(\Qpi, V) = H^0(\Qpi, V^*(1))^*$, and the $\Lambda_L(\Gamma)$-torsion submodule of $H^1_{\Iw}(\Qpi, V)$ is isomorphic to $H^0(\Qpi, V)$.

  \begin{lemma}\label{lemma:semisimple}
   If $V$ is crystalline, then $V^{G_{\Qpi}}$ is semisimple as a $\Lambda_L(\Gamma)$-module, isomorphic to a direct sum of powers of the cyclotomic character.
  \end{lemma}

  \begin{proof}
   This follows from the fact that $H^1_f(\Qp, \Qp)$ has dimension 1, corresponding to the unramified $\Zp$-extension of $\Qp$, which is linearly disjoint from the cyclotomic extension; hence there are no nontrivial crystalline extensions of $\Qp(j)$ by itself (for any $j \in \ZZ$) which factor through $\Gamma$.
  \end{proof}

  From this lemma and the remarks above, we see that if $V$ is crystalline, both $H^2_{\Iw}(\Qpi, V)$ and the torsion submodule of $H^1_{\Iw}(\Qpi, V)$ are semisimple.

  We note also the following facts regarding the structure of the cyclotomic Iwasawa cohomology. For each $n$ we have an injective inflation map \[ \left(V^{G_{\Qpi}}\right)_{\Gamma_n} \hookrightarrow H^1(\QQ_{p, n}, V),\]
  and these assemble to an injection
  \[ V^{G_{\Qpi}} \hookrightarrow H^1_{\Iw}(\Qpi, V).\]

  \begin{proposition}[Perrin-Riou, {\cite[\S 2.1.6]{perrinriou92}}]
   \label{prop:defPRpairing}
   There is a pairing
   \[ \langle -, -\rangle_{\Qpi,\Iw} : H^1_\Iw(\Qpi, V) \times H^1_\Iw(\Qpi, V^*(1)) \to \Lambda_L(\Gamma),\]
   linear in the first variable and antilinear (i.e.~$\Lambda(\Gamma)$ acts via the involution $\iota$ induced by $\gamma\mapsto \gamma^{-1}$) in the second; and this pairing induces an exact sequence
   \[ 0 \to V^{G_{\Qpi}} \to H^1_\Iw(\Qp, V) \to \Hom_{\Lambda_L(\Gamma)}(H^1_\Iw(\Qpi, V^*(1)), \Lambda_L(\Gamma))^\iota \to 0.\]
   (Here the superscript $\iota$ indicates that $\Gamma$ acts via the involution $\iota$.)
  \end{proposition}

  \begin{corollary}\label{corr:prpairingperfect}
   Let $x_1, \dots, x_d$ be any basis of $H^1_\Iw(\Qpi, V) / V^{G_{\Qpi}}$, and let $x'_1, \dots, x'_d$ any basis of the corresponding module for $V^*(1)$. Then the matrix whose $(i, j)$ entry is $\langle x_i, x_j'\rangle_{\Qpi,\Iw}$ is invertible in $M_{d \times d}(\Lambda_L(\Gamma))$.
  \end{corollary}

  We define a pairing
  \[ \langle-,-\rangle_{\cris} : \Big( \cH_L(\Gamma) \otimes_L \Dcris(V)\Big) \times \Big( \cH_L(\Gamma) \otimes_L \Dcris(V^*(1))\Big) \to \cH_L(\Gamma)\]
  by extending scalars from the natural pairing $\Dcris(V) \times \Dcris(V^*(1)) \to \Dcris(L(1)) = L$. We again make this linear in the first variable and antilinear in the second variable.

  The following theorem is Perrin-Riou's reciprocity formula, stated in terms of the regulator maps $\cL^\Gamma_{V, \xi}$, rather than their relatives the Perrin-Riou exponential maps $\Omega_{V,h}^\xi$. In this theorem, we extend the definition of $\cL^\Gamma_{V, \xi}$ to representations with arbitrary Hodge--Tate weights, at the cost of taking values in the total ring of fractions $\cK_L(\Gamma)$ of $\cH_L(\Gamma)$, as described \emph{loc.cit.}.

  \begin{theorem}[{\cite[Appendix B]{loefflerzerbes11}}]\label{thm:localrecip}
   For $x \in H^1_{\Iw}(\Qpi, V)$ and $y \in H^1_{\Iw}(\Qpi, V^*(1))$, we have
   \[ \langle \cL^\Gamma_{V, \xi}(x), \cL^\Gamma_{V^*(1), \xi}(y) \rangle_{\cris} = -\gamma_{-1} \cdot \ell_0 \cdot \langle x, y \rangle_{\Qpi, \Iw},\]
   where $\gamma_{-1}$ is the unique element of $\Gamma$ such that $\chi(\gamma_{-1}) = -1$.
  \end{theorem}

  We use Theorem \ref{thm:localrecip} to calculate the determinant of the regulator map. One readily verifies that $\mu_{n + 1} = \ell_0 \operatorname{Tw}_{-1}(\mu_{n})$, where $\operatorname{Tw}_{-1}$ is the twisting map $\gamma \to \chi(\gamma)^{-1} \gamma$ on $\cH_{\Qp}(\Gamma)$. Let us write $f_{V}$ for a generator of the characteristic ideal of $V^{G_{\Qpi}}$, and similarly for $f_{V^*(1)}$.

  \begin{proposition}
   Let $x_1, \dots, x_d \in H^1_{\Iw}(\Qp, V)$, and $v_1, \dots, v_d$ a basis of $\Dcris(V^*(1))$. Then the matrix $A \in M_{d \times d}(\cK_L(\Gamma))$ with $(i, j)$ entry
   \[ A_{ij} = \langle \cL^\Gamma_{V, \xi}(x_i), v_j \rangle_{\cris}\]
   satisfies
   \[ \det(A) \in \ell(V) \cdot \frac{f_V}{\iota(f_{V^*(1)})} \cdot \cH_L(\Gamma).\]
  \end{proposition}

  \begin{proof}
   Since the statement is invariant under twisting, we may assume that the Hodge--Tate weights of $V$ are non-negative and that $V$ has no quotient isomorphic to $\Qp$, so in particular $A \in M_{d \times d}(\cH_L(\Gamma))$.

   We note that for each character $\eta$ of Hodge--Tate weight $j \ge 0$ and conductor $n$, the subspace of $\Dcris(V) \otimes_L L(\eta)$ spanned by the elements $\cL^\Gamma_{V, \xi}(x_i)(\eta)$ is contained in
   \[ L(\eta) \otimes_L \begin{cases}
       \vp^n \Fil^{-j} \Dcris(V) & \text{if $n \ge 1$,}\\
       (1 - p^j \vp)(1 - p^{-1-j} \vp^{-1})^{-1} \Fil^{-j} \Dcris(V) & \text{if $n = 0$.}
      \end{cases}\]
   The codimension of this subspace is exactly $\dim t(W) + \dim H^0(\Qp, W) - \dim H^2(\Qp, W)$, where $W = V(\eta^{-1})$. (This is obvious for $n > 0$, and follows from Corollary \ref{corr:dimsv} for $n = 0$.) From the fact that each of the $\Gamma_{\mathrm{tors}}$-isotypical components of $\cH_L(\Gamma)$ is an elementary divisor domain, we deduce that $\det(A)$ vanishes to order at least $\dim t(W) + \dim H^0(\Qp, W) - \dim H^2(\Qp, W)$ at $\eta$ (cf.~\cite[Proposition 4.2]{leiloefflerzerbes11}).

   By construction, $\ell(V)$ vanishes at $\eta$ to order $\dim t(W)$; and by the semisimplicitly lemma \ref{lemma:semisimple}, $\frac{f_V}{\iota(f_{V^*(1)})}$ vanishes at $\eta$ to order $\dim H^0(\Qp, W) - \dim H^2(\Qp, W)$. It follows that $\det(A)$ is divisible by $\frac{f_V}{\iota(f_{V^*(1)})} \ell(V)$, as required.
  \end{proof}

  To bound the determinant from above, we shall use Perrin-Riou's pairing and the local reciprocity formula.

  \begin{lemma}
   If $x_1, \dots, x_d$ map to a basis of the free module $H^1_{\Iw}(\Qpi, V) / V^{G_{\Qpi}}$, then $\det(A)$ generates the fractional ideal $\ell(V) \frac{f_V}{\iota(f_{V^*(1)})} \cH_L(\Gamma)$.
  \end{lemma}

  \begin{proof}
   Let us choose $x_1', \dots, x_d' \in H^1_{\Iw}(\Qpi, V^*(1))$. Let $v_1', \dots, v_d'$ be the basis of $\Dcris(V)$ dual to $v_1, \dots, v_d$. Then the matrix $A'$ with $A_{ij} = \langle \cL^\Gamma_{V^*(1), \xi}(x_i'), v_j'\rangle_{\mathrm{cris}}$ has determinant in $\ell(V^*(1))\frac{f_{V^*(1)}}{\iota(f_{V})} \cH_L(\Gamma)$, by applying the previous proposition to $V^*(1)$.

   The matrix $A \cdot \iota(A')^\mathrm{T}$ has $i, j$ entry
   \[ \langle \cL^\Gamma_{V, \xi}(x_i), \cL^\Gamma_{V^*(1), \xi}(x_i')\rangle_{\cris} = -\sigma_{-1} \cdot \ell_0 \cdot B_{ij},\]
   where $B_{ij} = \langle x_i, x_j'\rangle_{\Qpi, \Iw} \in \Lambda_L(\Gamma)$, by the reciprocity formula \ref{thm:localrecip}. Hence we have
   \[ \left( \frac{\det A}{\frac{f_V}{\iota(f_{V^*(1)})} \ell(V)}\right) \cdot \iota\left(\frac{\det A'}{\frac{f_{V^*(1)}}{\iota(f_{V})} \ell(V^*(1))}\right) = \frac{(-\sigma_{-1} \cdot \ell_0)^d}{\ell(V) \iota(\ell(V^*(1)))} \det B.\]
   Since $\iota(\ell_j) = -\ell_{-j}$, we have
   \[ \ell(V) \iota(\ell(V^*(1))) = (-1)^d \prod_{i = 1}^d \mu_{n_i} \iota(\mu_{1 - n_i}) = (-1)^{\sum_{i=1}^d n_i} \ell_0^d.\]
   Thus
   \[ \left( \frac{\det A}{\frac{f_V}{\iota(f_{V^*(1)})} \ell(V)}\right) \cdot \iota\left(\frac{\det A'}{\frac{f_{V^*(1)}}{\iota(f_{V})} \ell(V^*(1))}\right) = (-1)^{\sum_{i=1}^d n_i} \cdot (-\sigma_{-1})^d \cdot \det B.\]

   We may choose $x_1', \dots, x_d'$ to be a basis of $H^1_{\Iw}(\Qpi, V^*(1))$ modulo torsion, in which case $\det B \in \Lambda_L(\Gamma)^\times = \cH_L(\Gamma)^\times$, by Proposition \ref{corr:prpairingperfect}. So it follows that
   \[\det(A) \in \ell(V) \cdot \frac{f_V}{\iota(f_{V^*(1)})} \cdot \cH_L(\Gamma)^\times.\]
  \end{proof}

  Since $f_V = f_Q$ and $f_{V^*(1)} = f_{H^2_{\Iw}(\Qpi, V)}$ in the notation of Proposition \ref{prop:det of cyclo regulator}, this proves the proposition in the special case when $x_1, \dots, x_d$ are a basis of $H^1_{\Iw}$ modulo torsion; but the statement of Proposition \ref{prop:det of cyclo regulator} is clearly independent of the choice of $x_1, \dots, x_d$, so the proposition holds for all choices of the $x_i$. This completes the proof of Proposition \ref{prop:det of cyclo regulator}.

 \subsection{The matrix of the two-variable regulator}

  We now consider the two-variable regulator $\cL^G_{V, \xi}$. The result we shall prove is formally very close to Proposition \ref{prop:det of cyclo regulator}:

  \begin{proposition}
   \label{prop:det of regulator}
   Let $x_1, \dots, x_d \in H^1_{\Iw}(K_\infty, V)$, and let $v_1, \dots, v_d$ be an $L$-basis of $\Dcris(V)^\vee$. If the quotient
   \[ Q = H^1_{\Iw}(K_\infty, V) / \langle x_1, \dots, x_d \rangle\]
   is $\Lambda_L(G)$-torsion, then the determinant of the matrix $A$ whose $i, j$ entry is $\langle \cL^G_{V, \xi}(x_i), v_j \rangle$ lies in $\ell(V) \cdot f_Q \cdot \cH_{\Lt}(G)^\times$, where $f_Q$ is a characteristic element of $Q$.
  \end{proposition}

  Since $\cH_{\Lt}(G)$ is not an elementary divisor domain, we cannot proceed as before, so we shall prove the proposition by reducing it to Proposition \ref{prop:det of cyclo regulator} applied to the twist of $V$ by every unramified character of $G$.

  Let $x_1,\dots,x_d \in H^1_{\Iw}(K_\infty,V)$ be as in the proposition, and denote by $f_Q$ a generator of the characteristic ideal of the quotient $Q$. Define
  \[ C=\left[ \det\langle \cL^G_{V, \xi}(x_i), v_j\rangle\slash f_Q\right]\in \cK_{\Lt}(G)^\times\slash \cH_{\Lt}(G)^\times\]
  where as usual $\langle \cL^G_{V, \xi}(x_i), v_j\rangle$ denotes the matrix of $\cL^G_{V, \xi}$ with respect to the basis $(x_i)$ and $(v_i)$, respectively. So our goal is to prove that $C = [\ell(V)]$.

  \begin{lemma}
   \label{lem:independentofbases}
   The definition of $C$ is independent of the choice of $\{x_1,\dots,x_d\}$ and $\{v_1,\dots,v_d\}$.
  \end{lemma}

  \begin{proof}
   Clear from the definition of $C$.
  \end{proof}

  \begin{definition}
   Denote by $\cK_{L}(G)^\circ$ the set of $f \in \cK_L(G)$ with the property that for each character $\tau$ of $U$, we can find an expression for $f$ in the form $u/v$ where the images of $u$ and $v$ in $\cK_{L(\tau)}(\Gamma)$ under the evaluation-at-$\tau$ map are not zero-divisors, so $f$ has a well-defined image in $\cK_{L(\tau)}(\Gamma)^\times$. It is clear that $\cK_L(G)^\circ$ is a subgroup of $\cK_L(G)^\times$ containing $\cH_L(G)^\times$.
  \end{definition}

  \begin{proposition}
   \label{prop:jacobson}
   For any coefficient field $L$, the natural map
   \[ \frac{\cK_L(G)^\circ}{\cH_L(G)^\times}
   \to \prod_{\tau} \frac{\cK_{L(\tau)}(\Gamma)^\times} {\cH_{L(\tau)}(\Gamma)^\times}\]
   is injective (where the product on the right is over all characters $\tau$ of $U$).
  \end{proposition}

  \begin{proof}
   We identify $\cH_L(G)$ with the algebra of $L[\Delta]$-valued rigid-analytic functions on the product of two copies of the rigid-analytic open disc $B(0, 1)$, where $\Delta$ is the torsion subgroup of $\Gamma$. By passing to $\Delta$-isotypical components, it suffices to replace $\cH_L(G)$ with $\cO(B(0, 1)_L^{2})$.

   Let $X_n$ be an ascending family of open affinoid subdomains of $B(0, 1)_L^2$. Since $\cO(B(0, 1)_L^2) = \varprojlim_n \cO(X_n)$, we are reduced to proving the corresponding statement for the rings $\cO(X_n)$, which are isomorphic to Tate algebras (at least after a suitable extension of the field $L$). However, Tate algebras are Jacobson rings and unique factorization domains \cite[\S 5.2.6]{boschguentzerremmert84}. Thus for any non-unit $f \in \operatorname{Frac} \cO(X_n)$, we may write $f = u / v$ where (without loss of generality) there exists some irreducible element $h$ such that $h \mid u$ but $h \nmid v$. By the Jacobson property, there exists a maximal ideal containing $h$ but not containing $v$. This maximal ideal must correspond to a point of $X_n$ at which $u$ vanishes but $v$ does not; since all points of $X_n$ are of the form $(\tau, \tau')$ for $\tau$ a character of $U$ and $\tau'$ a character of $\Gamma$, this shows that the image of $f$ in $\cK_{L(\tau)}(\Gamma)^\times$ is not a unit
   either, as required.
  \end{proof}

  To apply this in our case, we need the following proposition.

  \begin{proposition}\label{prop:Cnicesubset}
   We have $C\in \mathcal{K}_{\Lt}(G)^\circ\slash \mathcal{H}_{\Lt}(G)^\times$.
  \end{proposition}

  In order to prove the proposition, we need the following result:

  \begin{proposition}\label{lem:smallcokernel}
   Let $\tau$ be an unramified character of $G$. Then the corestriction map
   \[ \cores:H^1_{\Iw}(K_\infty,V(\tau^{-1}))\rTo H^1_{\Iw}(\QQ_{p,\infty},V(\tau^{-1}))\]
   induces an injection
   \[ \cores: H^1_{\Iw}(K_\infty,V(\tau^{-1}))_U\hookrightarrow H^1_{\Iw}(\QQ_{p,\infty},V(\tau^{-1})),\]
   and the cokernel is a finite dimensional $L$-vector space, where $U=\mathrm{Gal}(K_\infty/\QQ_{p,\infty})$.
  \end{proposition}
  \begin{proof}
   Let $T$ be a Galois-stable lattice in $V$. Recall that local Tate duality induces isomorphisms
   \[ H^1_{\Iw}(K_\infty,T(\tau^{-1}))^*\cong H^1(K_\infty,V^*/T^*(1)(\tau))\hspace{3ex}\text{and}\hspace{3ex}H^1_{\Iw}(\QQ_{p,\infty},T(\tau^{-1}))^*\cong H^1(\QQ_{p,\infty},V^*/T^*(1)(\tau)).\]
   Moreover, the transpose of corestriction is restriction. By inflation-restriction and the fact that $U\cong\Zp$, we have a short exact sequence
   \[ 0\rTo H^1(U,V^*/T^*(1)(\tau)^{H_\infty})\rTo H^1\big(\QQ_{p,\infty},V^*/T^*(1)(\tau))\rTo H^1(K_\infty,V^*/T^*(1)(\tau))^U\rTo 0,\]
   which we can dualize to get a short exact sequence
   \begin{equation}\label{ses}
    0\rTo H^1_{\Iw}(K_\infty,T)_U\rTo^{\cores} H^1_{\Iw}(\QQ_{p,\infty},T)\rTo H^1(U,V^*/T^*(1)(\tau)^{H_\infty})^*\rTo 0.
   \end{equation}
   Now $H^1(U,V^*/T^*(1)(\tau)^{H_\infty})=\big(V^*/T^*(1)(\tau)^{H_\infty}\big)_U$ is a confinitely generated $\cO$-module, which finishes the proof.
  \end{proof}

  \begin{corollary}
   The $\Lambda_L(\Gamma)$-module $H^1_{\Iw}(K_\infty,V(\tau^{-1}))_U$ is finitely generated of rank $d=\dim_LV$.
  \end{corollary}
  \begin{proof}
   Immediate from Proposition \ref{lem:smallcokernel} and the well-known fact that $H^1_{\Iw}(\QQ_{p,\infty},V(\tau^{-1}))$ has $\Lambda_L(\Gamma)$-rank $d$.
  \end{proof}

  To simplify the notation, let $M=H^1_{\Iw}(K_\infty,V(\tau^{-1}))$, and let $I\subset \Lambda_L(G)$ be the augmentation ideal of $U$. As $I=\langle u-1\rangle$ for any topological generator $u$ of $U$, we have $M_U=M/IM$.

  \begin{lemma}\label{lem:compatibleunderprojection}
   Let $y_1,\dots,y_d\in M/IM$ such that $(M/IM)/\langle y_i\rangle$ is $\Lambda_L(\Gamma)$-torsion. If $x_1,\dots,x_d$ are any lifts of the $y_i$ to $M$, then $M/\langle x_1,\dots,x_d\rangle$ is $\Lambda_L(G)$-torsion. Moreover, if $f$ (resp. $g$) denotes the characteristic ideal of $M / \langle x_i \rangle$ (resp. $(M/IM) / \langle y_i \rangle$), then $g=f\pmod I$. In particular, $f\neq 0\pmod I$.
  \end{lemma}

  \begin{proof}
   Suppose that $M/\langle x_1,\dots,x_d\rangle$ is not $\Lambda_L(G)$-torsion. Then the rank of the submodule of $M$ spanned by the $x_i$ is $< d$, so there exist $a_1,\dots,a_d\in \Lambda_L(G)$ such that
   \[ a_1x_1 + \dots + a_d x_d = 0.\]
   Denote by $\bar{a}_i$ the image of $a_i$ in $\Lambda_L(\Gamma)$. As $M$ is torsion-free, the $I$-torsion submodule of $M$ is zero, so we may divide out powers of the generator of $I$ and assume without loss of generality that $\bar{a}_i\neq 0$ for some $i$. We then get a relation
   \[ \bar{a}_1y_1+\dots+\bar{a}_dy_d=0,\]
   which gives a contradiction since the $y_i$ span a $\Lambda_L(\Gamma)$-submodule of $M/I$ of rank $d$.

   We now prove the statement regarding the characteristic ideal of the quotients. We claim that if $Q$ is any torsion $\Lambda(G)$ module such that $Q^U = 0$, then $\operatorname{char}_{\Lambda_L(\Gamma)}(Q_U) = \operatorname{char}_{\Lambda_L(G)}(Q) \pmod I$.

   By the structure theorem for finitely-generated $\Lambda_L(G)$-modules, there exists a submodule $Q' \subseteq Q$ and short exact sequences of the form
   \[ 0 \rTo P \rTo \bigoplus_{i=1}^d (\Lambda_L(G) / f_i^{n_i}) \rTo Q' \rTo 0\]
   and
   \[ 0 \rTo Q' \rTo Q \rTo P' \rTo 0\]
   where $f_i$ are irreducible elements of $\Lambda_L(G)$, $n_i \in \NN$ and $P$ and $P'$ are pseudonull. Since $Q$ and hence $Q'$ have zero $U$-invariants, we have
   \[ 0 \rTo (P')^U \rTo Q'_U \rTo Q_U \rTo P'_U \rTo 0.\]
   Since $P'$ is pseudo-null as a $\Lambda_L(G)$-module, it is torsion as a $\Lambda_L(\Gamma)$-module, and from the short exact sequence
   \[ 0 \rTo (P')^U\rTo P' \rTo^{u-1} P' \rTo P'_U \rTo 0\]
   and the multiplicativity of characteristic elements in short exact sequences we have $\Char_{\Gamma}(P')^U = \Char_\Gamma P'_U$ and hence $\Char_\Gamma Q'_U = \Char_\Gamma Q_U$.

   Turning our attention to the first short exact sequence, and letting $R = \bigoplus_i (\Lambda_L(G) / f_i^{n_i})$, we see that since $Q^U = 0$ we have
   \[ 0 \rTo P^U \rTo R^U \rTo 0 \rTo P_U \rTo R_U \rTo Q'_U \rTo 0\]
   and hence $\Char_{\Gamma} Q'_U = \Char_\Gamma R_U / \Char_\Gamma P_U$; and since $P$ is pseudo-null, as above we have $\Char_\Gamma P_U = \Char_\Gamma P^U = \Char_\Gamma R^U$. Combining the above we conclude that
   \[ \Char_{\Gamma} Q_U = \Char_\Gamma R_U / \Char_\Gamma R^U.\]
   So it suffices to check that our hypotheses force $R^U$ to be zero. If $R^U$ is nonzero, then one of the $f_i$ must be divisible by $u-1$; but since the $f_i$ are irreducible, this forces $f_i = u-1$ and hence $R^U$ is infinite-dimensional over $L$, contradicting the fact that $R^U = P^U$ and $P$ is finite-dimensional.
  \end{proof}

  We can now complete the proof of Proposition \ref{prop:Cnicesubset}. Let $\tau$ be a character of $U$, inducing a homomorphism of $\Lambda(U)$-modules $\Lambda(U) \to L(\tau)$. Let $\pi_{\Gamma, \tau}:\Lambda_L(G)\rightarrow \Lambda_L(\Gamma)(\tau)$ denote the homomorphism of $\Lambda(G)$-modules induced by taking the completed tensor product of the above map with $\Lambda_L(\Gamma)$. Using the compatibility of the regulator with crystalline twists (Proposition \ref{prop:regulator twist invariance}), we have
  \[ \pi_{\Gamma, \tau} \big(\det\langle \cL^G_{V, \xi}(x_i), v_j\rangle\big)=\det\langle \cL^\Gamma_{V(\tau^{-1}), \xi}(y_i),v_j\rangle \otimes e_\tau \neq 0.\]
  Now Lemma \ref{lem:compatibleunderprojection} implies that the image of a characteristic element of $H^1_{\Iw}(K_\infty,V)/\langle x_1,\dots,x_d\rangle$ under $\pi_{\Gamma, \tau}$ is non-zero, which implies that $C$ has a well-defined image in $\cK(\Gamma)^\times$, as required.

 \begin{proposition}
  Let $\tau$ be any continuous character of $U$, as above. Then the image of $C$ in $\cK_{\Lt}(\Gamma)^\times / \cH_{\Lt}(\Gamma)^\times$ is $[\ell(V)]$.
 \end{proposition}

  \begin{proof}
   We shall deduce this from Proposition \ref{prop:det of cyclo regulator} using Proposition \ref{prop:jacobson}.

   More precisely, let $A$ be as above, and let $\chi$ be a character of $U$. We claim that the image of $\det(A)$ in $\cK_{\Lt(\tau)}(\Gamma)^\times$ under the evaluation-at-$\chi$ map (which is well defined by proposition \ref{prop:Cnicesubset}) lies in $\ell(V) \cdot \overline{f_Q} \cdot \cH_{\Lt(\tau)}(\Gamma)^\times$, where $\overline{f_Q}$ is the image of $f_Q$; by the injectivity result of Proposition \ref{prop:jacobson}, this claim implies the statement we want. By twisting, it suffices to consider the case $\chi = 1$.

   Now, the class $C = \left[\frac{\det(A)}{f_Q}\right] \in \cK_L(G)^\circ / \cH_L(G)^\times$ is known to be independent of the choice of the $x_i$, so we may assume the $x_i$ map to elements $y_i$ of $H^1_{\Iw}(\Qpi, V)$ which span a full rank submodule. Then, as shown in Lemma \ref{lem:compatibleunderprojection}, the characteristic element of the quotient $H^1_{\Iw}(K_\infty, V)_U / \langle y_i \rangle$ is $\overline{f_Q}$; and as in \eqref{ses}, we have a short exact sequence of torsion modules
   \[ 0 \rTo H^1_{\Iw}(K_\infty, V)_U / \langle y_i \rangle \rTo H^1_{\Iw}(\Qpi, V)_U / \langle y_i \rangle \rTo \big(V^*/T^*(1)(\tau)^{H_\infty}\big)_U\rightarrow 0.\]
   By Tate duality, we have $\big(V^*/T^*(1)(\tau)^{H_\infty}\big)_U=H^2_{\Iw}(K_\infty, V)^U$, so we can rewrite the short exact sequence as
   \[ 0 \rTo H^1_{\Iw}(K_\infty, V)_U / \langle y_i \rangle \rTo H^1_{\Iw}(\Qpi, V) / \langle y_i \rangle \rTo H^2_{\Iw}(K_\infty, V)^U\rightarrow 0.\]

   Hence if we define
   \[ Q' = H^1_{\Iw}(\Qpi, V) / \langle y_i \rangle, \]
   we have
   \[ \Char_{\Gamma}(Q') = \overline{f_Q} \cdot \Char_{\Gamma}(H^2_{\Iw}(K_\infty, V)^U).\]
   On the other hand, we have
   \[ \Char_{\Gamma}(H^2_{\Iw}(K_\infty, V)_U) = \Char_{\Gamma}(H^2_{\Iw}(K_\infty, V)^U),\]
   since $H^2_{\Iw}(K_\infty, V)$ is finite-dimensional over $L$, and $H^2_{\Iw}(K_\infty, V)_U = H^2_\Iw(\Qpi, V)$; hence we have
   \[ \frac{\Char_{\Gamma}(Q')}{\Char_{\Gamma}(H^2_\Iw(\Qpi, V))} = \overline{f_Q}.\]

   So the claim above, and hence the proposition, follows from Proposition \ref{prop:det of cyclo regulator}.
  \end{proof}

 Combining this with Proposition \ref{prop:jacobson} completes the proof of Proposition \ref{prop:det of regulator}.


\section{Construction of the isomorphism}
 \label{sect:constructionG}

 The aim of this section is as follows. Let $T$ be a $\cO$-lattice in a crystalline $L$-linear Galois representation $V$. Let $\cOt = \widehat{\Zp}^{\mathrm{nr}} \otimes_{\Zp} \cO$ be the ring of integers of $\Lt$. We shall construct a canonical isomorphism of determinants over the ring $\Lambda_{\cOt}(G)$,
 \begin{multline*}
  \Theta_{\Lambda_{\cO}(G), \xi}(T) : \Det_{\Lambda_{\cOt}(G)}(0) \rTo^\cong \\ \Lambda_{\cOt}(G) \otimes_{\Lambda_{\cO}(G)} \left\{ \Det_{\Lambda_{\cO}(G)} R\Gamma_{\Iw}(K_\infty, T) \cdot \Det_{\Lambda_{\cO}(G)} \left(\Lambda_{\cO}(G) \otimes_{\cO} \Dcris(T)\right)\right\},
 \end{multline*}
 where $\Dcris(T)$ is a certain $\cO$-lattice in $\Dcris(V)$ determined by $T$ (see \S \ref{sect:powerofp} below). Our construction is to descend step-by-step in the following tower of ring extensions:
 \[ \Lambda_{\cOt}(G) \subset \Lambda_{\tilde L}(G) \subset \cK_{\tilde L}(G)\]
 where $\cH_{\tilde L}(G)$ is the algebra of locally analytic $\Lt$-valued distributions on $G$, and $\cK_{\tilde L}(G)$ its total ring of quotients.


 \subsection{Construction over \texorpdfstring{$\mathcal{K}_{\Lt}(G)$}{a large ring}}
  \label{sect:construction1}

  Over $\mathcal{K}_{\Lt}(G)$, the construction of the isomorphism $\Theta$ is very simple:

  \begin{proposition}
   The $\Lambda_L(G)$-module $H^1_{\Iw}(K_\infty, V)$ is finitely generated of rank $d$, and $H^i_{\Iw}(K_\infty, V)$ is torsion for $i \ne 1$.
  \end{proposition}

  \begin{proof}
   See \cite[Proposition A.6]{loefflerzerbes11}.
  \end{proof}

  \begin{proposition}
   The regulator $\cL^G_{V, \xi}$ induces an isomorphism
   \[ \Det_{\mathcal{K}_{\Lt}(G)}(0) \rTo^\cong \Det_{\mathcal{K}_{\Lt}(G)}\left(\mathcal{K}_{\Lt}(G) \otimes_{\Lambda_L(G)} R\Gamma_{\Iw}(K_\infty, V)\right) \cdot \Det_{\mathcal{K}_{\Lt}(G)}\left(\mathcal{K}_{\Lt}(G) \otimes_{L} \Dcris(V)\right).\]
  \end{proposition}

  \begin{proof}
   From the previous proposition we have
   \[ \Det_{\mathcal{K}_{\Lt}(G)}\left(\mathcal{K}_{\Lt}(G) \otimes_{\Lambda_L(G)} R\Gamma_{\Iw}(K_\infty, V)\right) \cong \Det_{\mathcal{K}_{\Lt}(G)}\left(\mathcal{K}_{\Lt}(G) \otimes_{\Lambda_L(G)} H^1_{\Iw}(K_\infty, V)\right)^{-1}.\]
   However, the regulator map is injective on $\Lt \htimes_L H^1_{\Iw}(K_\infty, V)$, so it is an isomorphism after tensoring with $\cK_{\Lt}(G)$.
  \end{proof}

  \begin{definition}
   \label{def:isooverK}
   We define an isomorphism
   \[ \Theta_{\mathcal{K}_{\Lt}(G), \xi}(V): \Det_{\mathcal{K}_{\Lt}(G)}(0) \rTo^\cong \Det_{\mathcal{K}_{\Lt}(G)}\left(\mathcal{K}_{\Lt}(G) \otimes_{\Lambda_L(G)} R\Gamma_{\Iw}(K_\infty, V)\right) \cdot \Det_{\mathcal{K}_{\Lt}(G)}\left(\mathcal{K}_{\Lt}(G) \otimes_{L} \Dcris(V)\right)\]
   by
   \[ \Theta_{\mathcal{K}_{\Lt}(G), \xi}(V) = \ell(V)^{-1} \Det\left(\cL^G_{V, \xi}\right)\]
   where $\ell(V) \in \mathcal{K}_{L}(G)^\times$ is as defined in \S \ref{sect:gamma} above.
  \end{definition}


 \subsection{Descent to \texorpdfstring{$\Lambda_{\Lt}(G)$}{the Iwasawa algebra}}
  \label{sect:descenttoLambda}

  We now use Proposition \ref{prop:det of regulator} to show that the isomorphism $\Theta_{\mathcal{K}_{\Lt}(G), \xi}(V)$ descends to $\Lambda_{\Lt}(G)$:

  \begin{theorem}
   \label{thm:existence of Theta}
   There exists a canonical isomorphism in $\underline{\Det}(\Lambda_{\Lt}(G))$
   \begin{multline*}
    \Theta_{\Lambda_L(G), \xi}(V) :
    \Det_{\Lambda_{\Lt}(G)} (0) \rTo^\cong \\
    \left[ \Det_{\Lambda_{\Lt}(G)} \left(\Lt \htimes_L R\Gamma_{\Iw}(K_\infty, V)\right)\right]
    \left[ \Det_{\Lambda_{\Lt}(G)} \left(\Lambda_{\Lt}(G) \otimes_L \Dcris(V) \right) \right]
   \end{multline*}
   with the property that the isomorphism in $\underline{\Det}(\cK_{\Lt}(G))$ obtained by extending scalars is the isomorphism of Definition \ref{def:isooverK}.
  \end{theorem}

  \begin{remark}
   By base change, we obtain a canonical isomorphism in $\underline{\Det}(\Lambda_{\Lt}(\Gamma))$,
   \begin{equation}\label{ThetaGamma}
    \Theta_{\Lambda_L(\Gamma), \xi}(V) :
    \Det_{\Lambda_{\Lt}(\Gamma)} (0) \rTo^\cong
    \left[ \Det_{\Lambda_{\Lt}(\Gamma)} \left(\Lt \htimes_L R\Gamma_{\Iw}(\QQ_{p,\infty}, V)\right)\right]
    \left[ \Det_{\Lambda_{\Lt}(\Gamma)} \left(\Lambda_{\Lt}(\Gamma) \otimes_L \Dcris(V) \right) \right].
   \end{equation}
   Since this coincides with $\ell(V)^{-1} \Det(\cL^\Gamma_{V, \xi})$ after base extension to $\cK_{\Lt}(\Gamma)$, the scalars descend from $\Lt$ to $L$: the above isomorphism is the image of an isomorphism
   \[ \Det_{\Lambda_{\Lt}(\Gamma)} (0) \rTo^\cong
    \left[ \Det_{\Lambda_{L}(\Gamma)} \left(R\Gamma_{\Iw}(\QQ_{p,\infty}, V)\right)\right]
    \left[ \Det_{\Lambda_{L}(\Gamma)} \left(\Lambda_{L}(\Gamma) \otimes_L \Dcris(V) \right) \right],
   \]
   in the category $\underline{\Det}(\Lambda_L(\Gamma))$, which we denote also by $\Theta_{\Lambda_L(\Gamma), \xi}(V)$. Note that $\Theta_{\Lambda_L(G), \xi}(V)$ does \emph{not} descend to $\underline{\Det}(\Lambda_L(G))$.
  \end{remark}

  Since for $i=0$ or $2$ the module $H^i_\Iw(K_\infty, V)$ is pseudo-null, its class in $\underline{\Det}(\Lambda_{\Lt}(G))$ is canonically isomorphic to the trivial object, by \cite[Lemma 2.2]{venjakob11}, and this canonical isomorphism is compatible under base extension with the isomorphism arising from the fact that $\cK_{\Lt}(G) \otimes_{\Lambda_{\Lt}(G)} H^i_\Iw(K_\infty, V) = 0$. Thus it suffices to construct an isomorphism
  \[ \Det_{\Lambda_{\Lt}(G)}\left( \Lt \htimes_L H^1_{\Iw}(K_\infty, V)\right) \rTo^\cong \Det_{\Lambda_{\Lt}(G)} \left(\Lambda_{\Lt}(G) \otimes_L \Dcris(V) \right).\]

  For this, we will need the following auxilliary result:

  \begin{lemma}
   \label{lemma:detcriterion}
   Let $R \into S$ be a morphism of commutative rings such that $SK_1(R) = SK_1(S) = 0$. Let $P, Q$ be free $R$-modules of equal rank $d$.

   Then any choice of $R$-bases of $P$ and $Q$ determines isomorphisms $\Isom_{\underline{\Det}(R)}(\Det_R(P), \Det_R(Q)) \cong R^\times$ and $\Isom_{\underline{\Det}(S)}(\Det_S(S \otimes_R P), \Det_S(S \otimes_R Q)) \cong S^\times$, and under these identifications, the image of $\Isom_{\underline{\Det}(R)}(\Det_R(P), \Det_R(Q))$ in $\Isom_{\underline{\Det}(S)}(\Det_S(S \otimes_R P), \Det_S(S \otimes_R Q))$ corresponds to the image of $R^\times$ in $S^\times$.
  \end{lemma}

  \begin{proof}
   This is clear from the definitions of the categories $\underline{\Det}(R)$ and $\underline{\Det}(S)$ and of the base-extension functor.
  \end{proof}

  \begin{remark}
   More generally, without the assumption on $SK_1$ we can assert that the image corresponds to the image of $K_1(R)$ in $K_1(S)$, but we shall not need this.
  \end{remark}

  \begin{lemma}
   \label{lemma:torsiondet}
   Let $M$ be a torsion $\Lambda_{L}(G)$-module. Then the morphism
   \[ Q_{L}(G)^\times \times_{\Lambda_{L}(G)^\times} \Isom_{\underline{\Det}(\Lambda_{L}(G))}(0, \Det_{\Lambda_{L}(G)} M) \cong \operatorname{Aut}_{\underline{\Det}(Q_{L}(G))}(0)\]
   given by the fact that $Q_{L}(G) \otimes_{\Lambda_{L}(G)} M = 0$ identifies $\Isom_{\underline{\Det}(\Lambda_{L}(G))}(0, \Det_{\Lambda_{L}(G)} M)$ with $f_Q^{-1} \Lambda_{L}(G)^\times$, where $f_Q$ is any characteristic element of $M$.
  \end{lemma}

  \begin{proof}
   For pseudo-null modules this is \cite[Lemma 2.2]{venjakob11}, and since the statement is compatible with direct sums, it suffices to consider the case of $M = \Lambda_{L}(G) / f$ for a single irreducible element $f$; but this follows immediately from tensoring the short exact sequence
   \[ 0 \rTo \Lambda_L(G) \rTo^{\times f} \Lambda_L(G) \rTo M \rTo 0\]
   with $Q_{L}(G)$.
  \end{proof}

  \begin{proposition}
   We have $\cH_{\Lt}(G)^\times = \Lambda_{\Lt}(G)^\times$.
  \end{proposition}

  \begin{proof}
   This is standard; see e.g.~\cite[(4.8)]{lazard1962}.
  \end{proof}

  \begin{proof}[Proof of Theorem \ref{thm:existence of Theta}]
   We will consider the base-extension maps corresponding to the successive ring extensions $\Lambda_{\Lt}(G) \into Q_{\Lt}(G) \into \cK_{\Lt}(G)$, where $Q_{\Lt}(G)$ is the total ring of quotients of $\Lambda_{\Lt}(G)$. (It would perhaps be more natural to consider $\Lambda_{\Lt}(G) \into \cH_{\Lt}(G) \into \cK_{\Lt}(G)$, but we do not know if $\cH_{\Lt}(G)$ has trivial $SK_1$.) These give homomorphisms (cf.~\cite[1.2.5]{fukayakato06})
   \begin{diagram}
    Isom_{\underline{Det}(\Lambda_{\Lt(G)})}
    \left[
    \Det_{\Lambda_{\Lt}(G)} \Big(\Lt \htimes_L H^1_{\Iw}(K_\infty, V) \Big), \Det_{\Lambda_{\Lt}(G)} \Big(\Lambda_{\Lt}(G) \otimes_L \Dcris(V)\Big)
    \right]\\
    \dTo\\
    Isom_{\underline{Det}(Q_{\Lt}(G))}
    \left[
    \Det_{Q_{\Lt}(G)} \Big(Q_{\Lt}(G) \otimes_{\Lambda_L(G)} H^1_{\Iw}(K_\infty, V) \Big), \Det_{Q_{\Lt}(G)} \Big(Q_{\Lt}(G) \otimes_L \Dcris(V)\Big)
    \right]\\
    \dTo\\
    Isom_{\underline{Det}(\cK_{\Lt}(G))}
    \left[
    \Det_{\cK_{\Lt}(G)} \Big(\cK_{\Lt}(G) \otimes_{\Lambda_{\Lt}(G)} H^1(K_\infty, V) \Big), \Det_{\cK_{\Lt}(G)} \Big(\cK_{\Lt}(G) \otimes_L \Dcris(V)\Big)
    \right]\\
   \end{diagram}
   All of these maps are evidently injective, and $\ell(V)^{-1} \Det(\cL^G_{V, \xi})$ evidently defines an element of the third group, so it suffices to show that it lies in the image of the composed map. Moreover, all three rings have trivial $SK_1$.

   We first show that it lies in the image of the second row. Since the two modules $Q_{\Lt}(G) \otimes_{\Lambda_L(G)} H^1_{\Iw}(K_\infty, V)$ and $Q_{\Lt}(G) \otimes_L \Dcris(V)$ are free (of equal ranks), it suffices to show that for any choice of bases of these modules, the determinant of the matrix of $\cL^G_{V, \xi}$ with respect to these bases is an element of $\ell(V) Q_{\Lt}(G)^\times \subset \cK_{\Lt}(G)^\times$. However, this is immediate from Proposition \ref{prop:det of regulator}, since $\cH_{\Lt}(G)^\times = \Lambda_{\Lt}(G)^\times$ is contained in $Q_{\Lt}(G)^\times$.

   To show that $\ell(V)^{-1} \Det(\cL^G_{V, \xi})$ actually lies in the image of the top row, we choose $x_i$ and $v_j$ satisfying the hypotheses of Proposition \ref{prop:det of regulator}. Writing $\Lambda$ for $\Lambda_{L}(G)$ and $\widetilde\Lambda$ for $\Lambda_{\Lt}(G)$ to lighten the notation, we have an exact sequence
   \[
    0 \rTo \widetilde\Lambda^{\oplus d} \rTo \Lt \htimes_L H^1_{\Iw}(K_\infty, V) \rTo \Lt \htimes_L N \rTo 0,
   \]
   where $N$ is torsion, from which we deduce a commutative diagram of multiplication maps (where, for sanity, we let $\Lambda = \Lambda_L(G)$ and $\widetilde\Lambda = \Lambda_{\Lt}(G)$, and $\widetilde Q = Q_{\Lt}(G)$)
   \begin{diagram}
    \left(\begin{array}{c}\Isom_{\underline{\Det}(\widetilde\Lambda)} (\Det N, 0) \\
     \times\\ \Isom_{\underline{\Det}(\widetilde\Lambda)}(\Det \widetilde\Lambda^d, \Det (\widetilde\Lambda \otimes_L \Dcris(V))
    \end{array}\right)
 & \rTo & \Isom_{\underline{\Det}(\widetilde\Lambda)} \Big(\Det \widetilde\Lambda \otimes_\Lambda H^1_{\Iw}(K_\infty, V), \Det \widetilde\Lambda \otimes_L \Dcris(V)\Big)\\
    \dInto & & \dInto\\
    \left(\begin{array}{c}
    \operatorname{Aut}_{\underline{\Det}(\widetilde Q)} (0) \\ \times \\ \Isom_{\underline{\Det}(\widetilde Q)}(\Det \widetilde Q^d, \Det \widetilde Q \otimes_{L} \Dcris(V))\end{array}\right) & \rTo & \Isom_{\underline{\Det}(\widetilde Q)} \Big(\Det \widetilde Q \otimes_{\Lambda} H^1_{\Iw}(K_\infty, V), \Det \widetilde Q \otimes_{L} \Dcris(V)\Big)
   \end{diagram}

   By Lemma \ref{lemma:torsiondet}, the image of $\Isom_{\underline{\Det}(\widetilde\Lambda)} (\Det N, 0)$ in $\operatorname{Aut}_{\underline{\Det}(\widetilde Q)} (0) \cong \widetilde Q^\times$ is $f_N^{-1} \Lambda^\times$; but, on the other hand, Proposition \ref{prop:det of regulator} shows that the element $\ell(V)^{-1} \Det \cL^G_{V, \xi}$ of $\Isom_{\underline{\Det}(\widetilde Q)} \Big(\Det \widetilde Q \otimes_{\Lambda} H^1_{\Iw}(K_\infty, V)$ is $f_Q$ times the image of an element of $\Isom_{\underline{\Det}(\widetilde\Lambda)}(\Det \widetilde\Lambda^d, \Det (\widetilde\Lambda \otimes_L \Dcris(V))$. Cancelling the factors of $f_N$ and $f_{N}^{-1}$, their product is the image of an element of the top right-hand corner, as required.
  \end{proof}


 \subsection{Integral coefficients}
  \label{sect:powerofp}

  Let $V$ be a crystalline representation with non-negative Hodge-Tate weights. We assume for now that $H^0(\Qpi, V) = 0$. Let $\NN(V)$ be the Wach module of $V$ (cf.~\cite{berger04}). Then the inclusion
  \[ \NN(V)\subset \BB^+_{\rig,\Qp} \otimes_{\BB^+_{\Qp}} \Dcris(V)\]
  induces an isomorphism of $\vp$-modules
  \[ i:\NN(V)/\pi\NN(V)\cong \Dcris(V).\]
  Let $T$ be a Galois-stable lattice in $V$; then there is a corresponding $\left(\cO \otimes \AA^+_{\Qp}\right)$-lattice $\NN(T) \subseteq \NN(V)$. We use this to define a $\cO$-lattice $\Dcris(T) \subseteq \Dcris(V)$, which is simply the image of $\NN(T)$ in $\Dcris(V)$ under the map $i$.

  \begin{proposition}\label{prop:fg}
   $\big(\vp^*\NN(T)\big)^{\psi=0} / (1-\vp)\NN(T)^{\psi=1}$ is a finitely-generated $\cO$-module.
  \end{proposition}

  \begin{proof}
   Let $h$ be the largest Hodge-Tate weight of $V$, and let $d=\dim_LV$. By the definition of a Wach module (see \cite[Definition II.4.1]{berger04}), $\vp^*\NN(T)/\NN(T)$ is annihilated by $q^h$, where $q=\vp(\pi)/\pi$. Hence the module
   \[ M = \vp(\pi)^h \big(\vp^*\NN(T)\big)\]
   satisfies $\vp(M) \subseteq M$.

   Suppose that $x \in \vp(\pi) M$. Then by induction we see that $\vp^n(x) \in \vp^{n+1}(\pi) M$; but since $\vp^n(\pi)\rightarrow 0$ as $n \to +\infty$, this implies that the series $\sum_{n \ge 0} \vp^n(x)$ converges to an element $y \in M$ satisfying $(1 - \vp)y = x$. Moreover, if $\psi(x) = 0$, then we evidently have $\psi(y) = y$. This implies that the image of
   \[ (1 - \vp) : \NN(T)^{\psi=1} \to \big(\vp^*\NN(T)\big)^{\psi=0}\]
   contains $\left(\vp(\pi) M\right)^{\psi = 0} = \vp(\pi)^{1 + h} \big(\vp^*\NN(T)\big)^{\psi = 0}$.

   Now, as shown in the proof of \cite[Proposition 3.11]{leiloefflerzerbes10}, for any $k \ge 0$ we have an isomorphism of $\Lambda(\Gamma)$-modules
   \[ \Lambda_{\cO}(\Gamma)^{\oplus d}/p_k\Lambda_{\cO}(\Gamma)^{\oplus d}\cong \big(\vp^*\NN(T)\big)^{\psi=0}/\vp(\pi)^k\big(\vp^*\NN(T)\big)^{\psi=0},\]
   where
   \[ p_k=(1-\gamma_1)(1-\chi(\gamma_1)^{-1}\gamma_1)\cdots(1-\chi(\gamma_1)^{1-k}\gamma_1),\]
   with $\gamma_1$ a generator of $\Gamma_1 \subseteq \Gamma$. So if $k \ge 1 + h$, the quotient $\big(\vp^*\NN(T)\big)^{\psi=0}/(1-\vp)\NN(T)^{\psi=1}$ is a finitely-generated $\Lambda_{\cO}(\Gamma)$-module annihilated by the element $p_k$. However, $\Lambda_{\cO}(\Gamma) / p_k \Lambda_{\cO}(\Gamma)$ is finitely-generated over $\cO$; hence $\big(\vp^*\NN(T)\big)^{\psi=0} / (1-\vp)\NN(T)^{\psi=1}$ must also be finitely-generated over $\cO$.
  \end{proof}

  \begin{corollary}
   \label{cor:charidealintegral}
   The characteristic ideal of the quotient
   \[ \frac{\big(\vp^*\NN(T)\big)^{\psi=0}}{(1-\vp)\NN(T)^{\psi=1}}\]
   is equal to $\prod_{i = 1}^d p_{n_i}$, where $n_i$ are the Hodge--Tate weights of $V$ as above.
  \end{corollary}

  \begin{proof}
   The characteristic ideal of the $\Lambda_L(\Gamma)$-module $\frac{\big(\vp^*\NN(V)\big)^{\psi=0}}{(1-\vp)\NN(V)^{\psi=1}}$ is known to be generated by $\prod_{i = 1}^d p_{n_i}$, by \cite[Theorem 4.12]{leiloefflerzerbes11}. Hence the characteristic ideal of $\frac{\big(\vp^*\NN(T)\big)^{\psi=0}}{(1-\vp)\NN(T)^{\psi=1}}$ must be equal to this up to a power of $p$ for each $\Gamma / \Gamma_1$-isotypical component. However, this module is finitely-generated over $\cO$ by Proposition \ref{prop:fg}, so it has zero $\mu$-invariant.
  \end{proof}

  We now recall a theorem (due to Laurent Berger) giving a convenient basis of $\NN(T)$:

  \begin{theorem}[{Berger, cf.~\cite[Theorem 3.5]{leiloefflerzerbes10}}]
   There exists an $\cO \otimes \AA^+_{\Qp}$-basis $(x_1, \dots, x_d)$ of $\NN(T)$ with the property that
   \[ \big((1 + \pi) \vp(x_1), \dots, (1 + \pi) \vp(x_d)\big)\]
   is a $\Lambda_{\cO}(\Gamma)$-basis of $\left( \vp^* \NN(T)\right)^{\psi = 0}$.
  \end{theorem}

  \begin{remark}
   In fact the second author together with Peter Schneider and Ramdorai Sujatha have recently shown that \emph{every} $\cO \otimes \AA^+_{\Qp}$-basis of $\NN(T)$ has this property. However, this is not immediately obvious from the definitions (as is erroneously claimed in \cite{benoisberger08}).
  \end{remark}

  \begin{proposition}\label{prop1}
   Let $x_1, \dots, x_d$ be a basis of $\NN(T)$ as in the previous theorem. Then the images $v_i$ of $x_i$ in $\Dcris(T) = \NN(T) / \pi \NN(T)$ are a basis of $\Dcris(T)$ over $\cO$. Moreover, if we define a matrix $M \in \operatorname{M}_{d \times d}(\cH_L(\Gamma))$ by
   \[ (1 + \pi) \vp(x_j) = \sum_{i = 1}^d m_{ij} \cdot (1 + \pi)\vp(v_i),\]
   then the determinant of $M$ is
   \[ \frac{p^{m(V)} \ell(V)}{\prod_{i = 1}^d p_{n_i} }\]
   up to a unit in $\Lambda_{\cO}(\Gamma)$.
  \end{proposition}

  \begin{proof}
   The fact that the $v_i$ span $\Dcris(T)$ is clear, since $\Dcris(T) = \NN(T) / \pi \NN(T)$.

   For the second statement, we know that
   \[ \det(M) \cdot \frac{\prod_{i = 1}^d p_{n_i}}{p^{m(V)}\ell(V)}\]
   is an element of $\Lambda_L(\Gamma)^\times$, as a consequence of Corollary \ref{cor:charidealintegral} and Proposition \ref{prop:det of cyclo regulator} (or, alternatively, by Theorem A of \cite{leiloefflerzerbes11}). Let us calculate its image under a character $\eta$ of $\Gamma$ that is trivial on $\Gamma_1$. By construction, $\det(M)(\eta) = 1$, since $M$ is in $\operatorname{Mat}_{2 \times 2}(\cH_L(\Gamma_1))$ and is congruent to 1 modulo the trivial character.

   On the other hand, $(\gamma_1 - 1)$ has a simple zero at $\eta$ with derivative $\log \chi(\gamma_1)$, which has valuation 1, and for $i \ge 1$, the valuation of $1 - \chi^i(\gamma_1)$ is $1 + v_p(i)$; so the valuation of $ \left(\prod_{i = 1}^d p_k\right)^*(\eta)$
   is
   \[ m(V) + \sum_{\substack{1 \le i \le d \\ n_i \ge 1}} v_p( (n_i - 1)! ).\]
   On the other hand, $\ell(V)^*(\eta)$ is $\prod_{i = 1}^d (n_i - 1)!$ up to a sign, so the value at $\eta$ of $\frac{p^{m(V)} \ell(V)^*(\eta)}\prod_{i = 1}^d p_k $ lies in $\Zp^\times$.
  \end{proof}

  \begin{proposition}\label{prop2}
   The determinant of $\vp: \Dcris(V) \to \Dcris(V)$ is $p^{-m(V)}$ up to a $p$-adic unit.
  \end{proposition}

  \begin{proof}
   It suffices to consider the case where $V$ is 1-dimensional, in which case $V$ is the product of an unramified character (mapping Frobenius to a unit) and a power of the cyclotomic character, for which the result is obvious.
  \end{proof}

  \begin{corollary}\label{corr:overGamma}
   If $H^0(\Qpi,V)=0$, then there is a unique isomorphism
   \[ \Theta_{\Lambda_\cO(\Gamma), \xi}(T): \Det_{\Lambda_{\cO}(\Gamma)} H^1_{\Iw}(K_\infty, T) \rTo \Det_{\Lambda_{\cO}(\Gamma)} (\Lambda_{\cO}(\Gamma) \otimes_{\Zp} \Dcris(T))\]
   whose base extension to $\Lambda_{L}(\Gamma)$ is the morphism $\Theta_{\Lambda_L(\Gamma), \xi}(V)$ constructed in \eqref{ThetaGamma}.
  \end{corollary}

  \begin{proof}
   If $\eta$ is a character of $\Gamma/\Gamma_1$, denote by $e_\eta$ the idempotent corresponding to $\eta$ in $\Lambda_{\Zp}(\Gamma)$.
   It follows from Lemma \ref{lemma:detcriterion} that for each such character $\eta$ there is a unique integer $a_\eta$ such that $p^{a_\eta} e_\eta \Theta_{\Lambda_L(\Gamma), \xi}(V)$ lies in the image of the base-extension map. Applying Propositions \ref{prop1} and \ref{prop2} shows that we must have $a_\eta = 0$ for all $\eta$.
  \end{proof}

  We can now pass to $\Lambda_{\cO}(G)$ coefficients.

  \begin{corollary}
   \label{corr:integralTheta}
   Let $V$ be an arbitrary $L$-linear crystalline representation of $G_{\Qp}$. As before, let $\Lambda = \Lambda_{\cO}(G)$ and $\widetilde\Lambda = \Lambda_{\cOt}(G)$. There is a unique isomorphism
   \[
    \Theta_{\Lambda, \xi}(T): \Det_{\widetilde\Lambda} (0) \rTo \widetilde\Lambda \otimes_{\Lambda} \left\{ \Det_{\Lambda} R\Gamma_{\Iw}(K_\infty, T) \cdot \Det_{\Lambda} (\Lambda \otimes_{\cO} \Dcris(T))\right\}
   \]
   whose base extension to $\Lambda_{\tilde L}(G)$ is the morphism $\Theta_{\Lambda_L(G), \xi}(V)$ constructed in Theorem \ref{thm:existence of Theta}.
  \end{corollary}

  \begin{proof}
   Since the isomorphisms $\Theta_{\Lambda_L(G), \xi}(V)$ are compatible under crystalline twists, we can replace $V$ with a suitable crystalline twist and assume without loss of generality that $H^0(\Qpi,V) = 0$ and that the Hodge--Tate weights of $V$ are $\ge 0$. As the morphism $\Theta_{\Lambda_L(\Gamma), \xi}(V)$ is obtained from $\Theta_{\Lambda_L(G), \xi}(V)$ by base change, the result is now immediate from Corollary \ref{corr:overGamma}.
  \end{proof}

 \subsection{Definition of the epsilon-isomorphism}

  Having constructed the isomorphism $\Theta_{\Lambda_\cO(G), \xi}(T)$ of Corollary \ref{corr:integralTheta}, it is an easy step to define the epsilon-isomorphism whose construction is the main purpose of this paper. We need only the following result, which is essentially a restatement of a theorem of Berger:

  \begin{theorem}
   Let $V$ be a crystalline $L$-linear Galois representation, $T \subseteq V$ a $G_{\Qp}$-stable $\cO$-lattice, and $\Dcris(T)$ the corresponding $\cO$-lattice in $\Dcris(V)$ defined in \S \ref{sect:powerofp} above. Then there is a unique isomorphism
   \[ \varepsilon_{\cO, \xi, \dR}(T): \Det_{\cOt}(\cOt \otimes_\cO \Dcris(T)) \rTo^\cong \Det_{\cOt}(\cOt \otimes_{\cO} T) \]
   whose image under base-extension $\Lt \otimes_{\cOt} -$ is the isomorphism $\varepsilon_{L, \xi, \dR}(V)$ of section \ref{sect:deRhamepsilon}.
  \end{theorem}

  \begin{proof}
   By Lemma \ref{lemma:detcriterion}, it suffices to check that for any $\cO$-bases of $T$ and of $\Dcris(T)$, the matrix of the canonical isomorphism
   \[ \operatorname{can}_V: \Bcris \otimes_{\Qp} \Dcris(V) \cong \Bcris \otimes_{\Qp} \Dcris(V)\]
   has determinant in $t^{m(V)} \cOt^\times$. This is precisely the result of Proposition V.1.2 of \cite{berger04}.
  \end{proof}

  \begin{definition}
   We define
   \[
    \varepsilon_{\Lambda_\cO(G), \xi}(T): \Det_{\Lambda_{\cOt}(G)}(0) \rTo^\cong
    \left[ \Det_{\Lambda_{\cOt}(G)} \left(\cOt \htimes_{\cO} R\Gamma_{\Iw}(K_\infty, T)\right)\right]
    \left[ \Det_{\Lambda_{\cOt}(G)} \left(\Lambda_{\cOt}(G) \otimes_\cO T\right)\right]
   \]
   to be the isomorphism given by
   \[ (-\gamma_{-1})^d (-1)^{m(V)} \cdot \Theta_{\Lambda_{\cO}(G), \xi}(T) \cdot \varepsilon_{\cO, \xi, \dR}(T),\]
   where we regard $\varepsilon_{\cO, \xi, \dR}(T)$ as an isomorphism
   \[ \Det_{\Lambda_{\cO}(G)}(\Lambda_{\cO}(G) \otimes_{\cO} \Dcris(T)) \rTo^\cong \Det_{\Lambda_{\cO}(G)}(\Lambda_{\cO}(G) \otimes_{\cO} T)\]
   via base-extension.
  \end{definition}

  \begin{remark}
   Note that $(-1)^{m(V)} (-\gamma_{-1})^d = \Det\left( -\gamma_{-1} : \Lambda_{\cO}(G) \otimes T\right)$. The factor $-\gamma_{-1}$ is the same as in Theorem \ref{thm:localrecip}.
  \end{remark}

 \subsection{Properties of the epsilon-isomorphism}

  We note for later use some properties of the isomorphisms $\Theta_{\Lambda_{\cO}(G), \xi}(T)$ and $\varepsilon_{\Lambda_{\cO}(G), \xi}(T)$:

  \begin{proposition}[Compatibility with short exact sequences]
   \label{prop:regulatorSEScompatible}
   Let
   \[ 0 \rTo T' \rTo T \rTo T''\rTo 0\]
   be a short exact sequence of $\cO$-linear crystalline representations of $G_{\Qp}$. Then
   \[ \Theta_{\Lambda_{\cO}(G), \xi}(T)=\Theta_{\Lambda_{\cO}(G), \xi}(T')\cdot \Theta_{\Lambda_{\cO}(G), \xi}(T'')\]
   and
   \[ \varepsilon_{\Lambda_{\cO}(G), \xi}(V) = \varepsilon_{\Lambda_{\cO}(G), \xi}(V')\cdot \varepsilon_{\Lambda_{\cO}(G), \xi}(V'').\]
  \end{proposition}

  \begin{proof}
   It suffices to check this after arbitrary base-extension; but over $\cK_{\Lt}(G)$ the result is obvious, since the regulator map $\cL^G_{V, \xi}$ is compatible with short exact sequences, as are the factors $\ell(V)$ and $\varepsilon_{L, \xi, \dR}(V)$ (the latter by Lemma \ref{lem:SEScompatible}).
  \end{proof}

  \begin{proposition}[Change of coefficient field]
   \label{prop:regulatorbaseextension}
   Let $L'$ be a finite extension of $L$ with ring of integers $\cO'$. Then
   \[ \Theta_{\Lambda_{\cO'}(G), \xi}(\cO' \otimes_\cO T)=\cO' \otimes_\cO \Theta_{\Lambda_\cO(G), \xi}(T)\]
   and
   \[ \varepsilon_{\Lambda_{\cO'}(G), \xi}(\cO' \otimes_\cO T) = \cO' \otimes_\cO \varepsilon_{\Lambda_{\cO'}(G), \xi}(T).\]
  \end{proposition}

  \begin{proof}
   Clear.
  \end{proof}

  The next compatibility property takes a little more notation to state. For brevity let us write $\Lambda$ for $\Lambda_{\cO}(G)$. For $\eta$ a continuous $L$-valued (hence $\cO$-valued) character of $G$, we have a twisting homomorphism $\tw_\eta: \Lambda \to \Lambda$ which maps a group element $g \in G$ to $\eta(g) g$. Hence we obtain a pullback functor $(\tw_\eta)^*$ from the category of $\Lambda$-modules to itself,
  \[ (\tw_\eta)^* M = \Lambda \otimes_{\Lambda, \tw_\eta} M.\]
  This can also be described in terms of tensoring with the $\Lambda$-bimodule $\Lambda \otimes_{\Lambda, \tw_\eta} \Lambda$, which is free of rank one as a $\Lambda$-module; hence the twisting functor extends to a functor from the category $\underline{\Det}(\Lambda)$ to itself which is compatible with the functor $\Det$.

  Note that we have an isomorphism
  \[ (\tw_\eta)^* (\Lambda \otimes_{\cO} T) \cong \Lambda \otimes_{\cO} T(\eta^{-1}),
   \; a\otimes b\otimes v\mapsto a \operatorname{Tw}_{\eta}(b)\otimes (v \otimes t_{\eta^{-1}})\]
  as $(\Lambda, G_{\Qp})$-modules, if $\Lambda$ acts on $\Lambda \otimes_{\cO} T$ via left multiplication on the left factor, while $g \in G_{\Qp}$ sends $\lambda \otimes v$ to $\lambda \bar{g}^{-1} \otimes gv$ where $\bar{g}$ denotes the image of $g$ in $G$ (and analogously for the action on $\Lambda \otimes_{\cO} T(\eta^{-1})$).

  We clearly have $(\tw_\eta)^* \circ (\tw_{\eta^{-1}})^* = \id$. Similar definitions apply to other coefficient rings than $\Lambda_{\cO}(G)$, including $\Lambda_L(G)$, $\cH_L(G)$ or $\Lambda_{\cOt}(G)$.

  Finally note that for a $\Lambda$-module $M$ we have a canonical isomorphism
  \[\Lambda \otimes_{\Lambda, \operatorname{Tw}_{\eta^{-1}} } M = M\otimes_\cO \cO t_\eta,\; \lambda\otimes m\mapsto \tw_\eta(\lambda) m \otimes t_\eta,\]
  of $\Lambda$-modules, where the $\Lambda$-module structure on the right hand side is induced by the diagonal action of $G$ upon it.

  \begin{proposition}[Invariance under crystalline twists]
   \label{prop:regulator twist invariance 2}
   If $T' = T(\eta)$ for a crystalline character $\eta$ with values in $\cO$, then
   \begin{diagram}
    \operatorname{Tw}_{\eta^{-1}} \left(\Det_{\Lambda_{\cOt}(G)} \left(\cOt \htimes_\cO R\Gamma_{\Iw}(K_\infty, T)\right)^{-1} \right)
    & \rTo^{\operatorname{Tw}_{\eta^{-1}}\left(\Theta_{\Lambda_\cO(G), \xi}(T)\right)} &
    \operatorname{Tw}_{\eta^{-1}}\left(\Det_{\Lambda_{\cOt}(G)} \left(\Lambda_{\cOt}(G)) \otimes_\cO \Dcris(T) \right) \right)\\
    \dTo^{\cong} & & \dTo^{\cong} \\
    \Det_{\Lambda_{\cOt}(G)} \left(\cOt \htimes_\cO R\Gamma_{\Iw}(K_\infty, T(\eta))\right)^{-1}
    & \rTo^{\Theta_{\Lambda_\cO(G), \xi}(T(\eta))} &
    \Det_{\Lambda_{\cOt}(G)} \left(\Lambda_{\cOt}(G) \otimes_\cO \Dcris(T(\eta)) \right)
   \end{diagram}
   where we take the left vertical map as natural identification while the right vertical map is -- up to interchanging the twist- and determinant-functors -- $\Det_{\Lambda_{\cOt}(G)}(a_\eta)$ with the notation as in Proposition \ref{prop:regulator twist invariance} above.
  \end{proposition}

  \begin{proof}
   It suffices to check the statement after base-extension to $L$. First note that
   \[ \operatorname{Tw}_{\eta^{-1}}(m_\lambda)=m_{\operatorname{Tw}_{\eta^{-1}}(\lambda)}\]
   where $m_\lambda$ denotes multiplication by $\lambda\in \Lambda$. In particular we have
   \[\operatorname{Tw}_{\eta^{-1}}(\ell(V))=\ell(V(\eta))\ell(L(\eta))^{-d}\]
   where we omit the $m$ for simplicity again. Hence we obtain, by applying the determinant functor to the diagram in Proposition \ref{prop:regulator twist invariance} and after base extension to $\cK_{\Lt}(G)$,
   \begin{align*}
    \Det(a_\eta) \operatorname{Tw}_{\eta^{-1}}\left(\frac{\Det\cL^G_{V, \xi}}{\ell(V)}\right)
    &= \operatorname{Tw}_{\eta^{-1}}\left(\frac{1}{\ell(V)}\right)\Det(a_\eta) \operatorname{Tw}_{\eta^{-1}}\left(\Det\cL^G_{V, \xi}\right)\\
    &= \frac{\ell(L(\eta))^d}{\ell(V(\eta))}\ell(L(\eta))^{-d} \Det\cL^G_{V, \xi} \\
    &= \ell(V(\eta))^{-1} \Det(\cL^G_{V(\eta), \xi})
   \end{align*}
   which is just the claimed statement. Here, for the second equality we used the compatibility of twisting with taking determinants.
  \end{proof}

  \begin{proposition}
   If $\eta$ is crystalline, then we have $\tw_{\eta^{-1}}\left( (-\gamma_{-1})^d (-1)^{m(V)}\right) = (-\gamma_{-1})^d (-1)^{m(V(\eta))}$, where $m(V)$ denotes the sum of the Hodge--Tate weights of $V$ (and similarly for $V(\eta)$).
  \end{proposition}

  \begin{proof}
   By the definition of the twisting map, we have $\tw_{\eta^{-1}}(\gamma_{-1}) = \eta(\gamma_{-1}) \gamma_{-1} = (-1)^j \gamma_{-1}$ (since $\eta$ is $\chi^j$ times an unramified character); so
   \[ \tw_{\eta^{-1}}\left( (-\gamma_{-1})^d (-1)^{m(V)}\right) = (-1)^{jd}(-\gamma_{-1})^d (-1)^{m(V)} = (-\gamma_{-1})^d (-1)^{m(V) + jd}.\]
   Since $m(V(\eta)) = m(V) + jd$ we are done.
  \end{proof}

  Since the map $\Det(a_\eta)$ also evidently gives the twist-compatibility of the maps $\varepsilon_{L, \xi, \dR}(V)$ and $\varepsilon_{L, \xi, \dR}(V(\eta))$, we obtain the compatibilty of $\varepsilon_{\Lambda_\cO(G), \xi}(T)$ and $\varepsilon_{\Lambda_\cO(G), \xi}(T(\eta))$; i.e.,
  \[\operatorname{Tw}_{\eta^{-1}}\left(\varepsilon_{L, \xi, \dR}(V)\right)=\varepsilon_{L, \xi, \dR}(V(\eta))\]
  up to the indicated identifications. We use these results to extend the definition of $\varepsilon_{L, \xi, \dR}$ and $\Theta_{\Lambda_\cO(G), \xi}$ to lattices $T$ in arbitrary crystalline Galois representations $V$, by tensoring the corresponding maps for $T(j)$ with $\Qp(-j)$, where $j \gg 0$ is such that $V(j)$ has non-negative Hodge--Tate weights.


 \subsection{Epsilon-isomorphisms for more general modules}

  We recall the following definition from \cite[\S1.4]{fukayakato06}.

  \begin{definition}
   A ring $R$ is of
   \begin{description}
     \item[(type 1)] if there exists a two sided ideal $I$ of $R$ such that $R/I^n$ is finite of order a power of
     $p$ for any $n \geq 1$, and such that $R \cong \varprojlim_n R/I^n$.
     \item[(type 2)] if $R$ is the matrix-algebra $M_n(F)$ of some finite extension $E$ over $\Qp$ and some dimension $n\geq 1.$
   \end{description}
  \end{definition}

  By lemma 1.4.4 in (loc.\ cit.), $R$ is of type 1 if and only if the defining condition above holds with $I$ equal to the Jacobson ideal $J=J(R).$ Such rings are always semi-local and $R/J$ is a finite product of matrix algebras over finite fields. For a ring $R$ of type (1) or (2) we define
  \[ \Rt \coloneqq \widehat{\ZZ_p^{nr}} \htimes_{\Zp} R\]
  where $\widehat{\mathbb{Z}_p^{nr}}$ denotes the completion of the ring of integers of the maximal unramified extension of $\Qp.$

  Now let $T \subseteq V$ be a Galois-stable $\cO$-lattice of a crystalline representation of $G_{\Qp}$ with coefficients in some finite extension $L/\Qp$. We set $\T(T) \coloneqq \Lambda_{\cO}(G)\otimes_{\cO} T$, which we consider as $\Lambda(G)$-module by multiplication on the left tensor factor and as $G_{\Qp}$-module via $g(\lambda\otimes t)=\lambda \bar{g}^{-1}\otimes gt$. The following isomorphism (essentially a version of Shapiro's lemma) is well known:

  \begin{proposition}
   We have
   \[ R\Gamma(\Qp, \T(T)) \cong R\Gamma_{\Iw}(K_\infty, T)\]
   as $\Lambda(G)$-modules.
  \end{proposition}

  \begin{proof}
   See e.g.~\cite[8.4.4.2 Proposition]{nekovar06}.
  \end{proof}

  Let $\Lambda = \Lambda_{\cO}(G)$, which is a ring of type 1, with $\tilde\Lambda = \Lambda_{\cOt}(G)$. Then we have constructed an isomorphism
  \[ \varepsilon_{\Lambda, \xi}(T): \Det_{\tilde\Lambda}(0) \rTo^\cong
    \tilde\Lambda \otimes_{\Lambda} \left\{ \Det_{\Lambda} R\Gamma(\Qp, \T(T)) \cdot \Det_{\Lambda} \T(T)\right\}.
  \]
  We shall establish that this satisfies the properties predicted by \cite[Conjecture 3.4.3]{fukayakato06} for the module $\T(T)$. For that purpose it is convenient to write also $\varepsilon_{\Lambda, \xi}(\T(T))$ for the above $\varepsilon$-isomorphism, and to extend it to a slightly more general class of modules.

  We consider quadruples $(R, Y, T, \xi)$ where
  \begin{itemize}
   \item $R$ is a $p$-torsion-free $\cO$-algebra which is also a ring of type (1) or (2) above,
   \item $\xi$ is a compatible system of $p^n$-th roots of unity (as before),
   \item $T$ is a $G_{\Qp}$-stable $\cO$-lattice in a crystalline $L$-linear representation of $G_{\Qp}$,
   \item $Y$ is a finitely-generated projective left $R$-module, equipped with a continuous $R$-linear action of $G$.
  \end{itemize}

  Given such a quadruple, we define $\T = Y \otimes_{\cO} T$, which we equip with the obvious left $R$-module structure and an action of $G_{\Qp}$ via $g \cdot (y \otimes t) = y\bar{g}^{-1} \otimes g t$. Then $(R, \T, \xi)$ is a triple satisfying the conditions of \cite[\S 3.4.1]{fukayakato06}. Moreover, the action of $G$ on $Y$ extends to a $\Lambda$-module structure, and we have
  \[ \T = Y \otimes_{\Lambda} \T(T)\]
  where $\T(T)$ is as above. So we may define
  \[ \varepsilon_{R, \xi}(\T) \coloneqq Y\otimes_{\Lambda} \varepsilon_{\Lambda, \xi}(\T(T)),\]
  which is an isomorphism
  \[
    \Det_{\Rt}(0) \rTo^\cong \Rt \otimes_{R} \left\{ \Det_R R\Gamma(\Qp, \T) \cdot \Det_R \T \right\};
  \]
  here we have used the fact that
  \[ Y\otimes^{\mathbb{L}}_{\Lambda} R\Gamma(\Qp, \T(T))\cong R\Gamma(\Qp, Y\otimes_{\Lambda} \T(T))\]
  by \cite[1.6.5]{fukayakato06}.

  \begin{remark}
   Note that $R$ need not be commutative, and $Y$ need not be either projective or finitely-generated as a $\Lambda$-module.
  \end{remark}

  \begin{proposition}
   \label{prop:compat2}
   Suppose $R = \cO_F$ for some finite extension $F / L$, and the finite-dimensional $F$-vector space $F \otimes_{R} Y$ is de Rham as a representation of $G$. Then $F \otimes_{\cO_F} \T$ is also de Rham, and $F \otimes_{\cO_F}\varepsilon_{R, \xi}(\T)$ coincides with the canonical isomorphism $\varepsilon_{F, \xi}(F \otimes_{\cO_F} \T)$ of \S \ref{sect:deRhamepsilon}.
  \end{proposition}

  \begin{proof}
   Since our $\varepsilon$-isomorphisms commute with base change in $L$, it suffices to assume $F = L$. We may also assume that $L$ is sufficiently large that all the Jordan--H\"older constituents of $Y$ are one-dimensional. By the compatibility with short exact sequences, it suffices to assume $Y$ is itself one-dimensional, so $Y = L(\eta)$ for a de Rham character $\eta$ of $G$. This reduces the result to Theorem \ref{thm:compatibleunderevaluation}, which we shall establish in the next section.
  \end{proof}

  \begin{corollary}
   Suppose that the pair $(R, \T)$ satisfies the following condition:
   \begin{itemize}
    \item if $\Phi_{\T}$ is the set of all $\cO$-algebra homomorphisms $\rho: R \to M_n(F)$ (where $F/L$ is a finite extension and $n$ an integer, both depending on $\rho$) such that $F^n \otimes_{R, \rho} \T$ is de Rham, then
    \[ K_1(R) \to \prod_{\rho \in \Phi_{\T}} F^\times\]
    is injective.
   \end{itemize}
   Then $\varepsilon_{R, \xi}(\T)$ depends only on $\xi$ and on the isomorphism class of $\T$ as an $R[G_{\Qp}]$-module.
  \end{corollary}

  \begin{proof}
   Clear from the preceding proposition, since the isomorphism $\varepsilon_{R, \xi}(\T)$ must be consistent with the de Rham $\varepsilon$-isomorphisms $\varepsilon_{F, \xi}(F^n \otimes_{R, \rho} \T)$, which are uniquely determined by $(R, \T, \xi)$.
  \end{proof}

  \begin{remark}
   We suspect that the statement of the corollary is true for arbitrary type 1 $\cO$-algebras $R$, but this is much more difficult to prove. For instance, if $T_1, T_2$ are two $\cO$-lattices in crystalline $L$-linear representations such that $T_1 / \varpi^n \cong T_2 / \varpi^n$ for some $n \ge 1$, then on taking $R = \cO / \varpi^n$ this would imply that the $\varepsilon$-isomorphisms for $T_1$ and $T_2$ are congruent modulo $\varpi^n$. This should certainly be true, but at present we can only prove it under the assumption that the Hodge--Tate weights of the $T_i$ all lie in some interval $[a, b]$ with $b - a < p-1$; we hope to return to this problem in a subsequent paper.
  \end{remark}

  We shall now show that the association $(R, Y, T, \xi) \to \varepsilon_{R, \xi}(\T)$ satisfies properties corresponding to conditions (i)---(iv) and (vi) of \cite[Conjecture 3.4.3]{fukayakato06}.

  \subsubsection*{Property (i) (additivity)} The first condition of \emph{op.cit.} states that for any three triples $(R, \T_i, \xi)$, $i = 1,2,3$, with common $R$ and $\xi$, and an exact sequence
  \[ 0 \rTo \T_1 \rTo \T_2 \rTo \T_3 \rTo 0, \]
  we have
  \[ \varepsilon_{R, \xi}(\T_2) = \varepsilon_{R, \xi}(\T_1)\varepsilon_{R, \xi}(\T_3).\]
  By assumption our $\T_i$ are of the form $Y_i \otimes T_i$, for crystalline $\cO$-representations $T_i$ and $R$-modules $Y_i$ with $G$-action. We shall consider only the cases when the exact sequence arises from an exact sequence of $Y_i$'s with a common $T$, or an exact sequence of $T_i$'s with a common $Y$. The first case is obvious from the construction of $\varepsilon_{R, \xi}(-)$. The latter case follows from Proposition \ref{prop:SEScompatible}.

  \subsubsection*{Property (ii) (base change)} The second condition is a compatibility with base-change in $R$; this is immediate from our construction.

  \subsubsection*{Property (iii) (change of \texorpdfstring{$\xi$}{xi})} Let $c \in \Zp^\times$, and let $\gamma_c$ be any element of $G_{\Qp}$ acting trivially on $\Qpnr$ and such that $\chi(\gamma_c) = c$. Then we must show that
  \[ \varepsilon_{R, c\xi}(\T) = [T, \gamma_c] \varepsilon_{R, \xi}(\T), \]
  where $[T, \gamma_c]$ is the class in $K_1(R)$ of the $R$-linear automorphism of $\T$ given by $\gamma_c$. (This is well-defined, as $\gamma_c$ is uniquely determined up to conjugation in $G_{\Qp}$.) It suffices to check this when $R = \Lambda$ and $\T = \T(T)$; but this is immediate from the corresponding property of the regulator map $\cL^G_{V, \xi}$, which is part (1) of Proposition \ref{prop:regulatorproperties}, and of the de Rham $\varepsilon$-isomorphism $\varepsilon_{L, \xi, \dR}(V)$.

  \subsubsection*{Property (iv) (Galois equivariance)} Let $\vp$ denote the arithmetic Frobenius automorphism of $\cO$. Then we must show that
  \[
   \varepsilon_{R, \xi}(\T) \in \Isom\left(\Det_R(0), \Det_R R\Gamma(\Qp, \T) \cdot \Det_R \T\right) \times^{K_1(R)} \left\{ x \in K_1(\Rt) : \vp(x) = [T, \sigma_p]^{-1} x\right\}
  \]
  where $\sigma_p$ is the arithmetic Frobenius element of $\Gal(\Qp^{ab} / \Qpi)$. Again, it suffices to assume $(R, \T) = (\Lambda, \T(T))$ and the result is now clear from the Galois-equivariance properties of the map $\cL^G_{V, \xi}$, cf.~Proposition \ref{prop:regulatorproperties} (2), and of the de Rham $\varepsilon$-isomorphism $\varepsilon_{L, \xi, \dR}(V)$, cf.~\cite[Proposition 3.3.7]{fukayakato06}.

  \subsubsection*{Property (v) (compatibility with de Rham $\varepsilon$-isomorphisms)}

   If $R$ is the ring of integers of a finite extension $F / L$, and $F \otimes_{R} \T$ is de Rham, we must check that $\varepsilon_{R, \xi}(\T)$ is consistent with $\varepsilon_{F, \xi}(F \otimes_{R} \T)$ as defined in \S \ref{sect:deRhamepsilon} above. This is exactly Proposition \ref{prop:compat2} above.

  \subsubsection*{Property (vi) (local duality)}

  Let $\T$ be a free $R$-module with compatible $G_{\Qp}$-action as above. Then
  \[ \T^* \coloneqq \Hom_{R}(\T,R) \]
  is a free $R^\circ$-module -- for the action $h\mapsto h(-)r,$ $r$ in the opposite ring $R^\circ$ of $R$ -- with compatible $G_{{\Qp}}$-action
  given by $h\mapsto h\circ \sigma^{-1}$. Recall that in Iwasawa theory we have the canonical involution $\iota:\Lambda^\circ\to \Lambda$, induced by $g\mapsto g^{-1},$ which allows to consider (left) $\Lambda^\circ$-modules again as (left) $\Lambda$-modules, e.g.\ one has $\T^*(T)^\iota\cong \T(T^*)$ as $(\Lambda, G_{{\Qp}})$-module, where $M^\iota \coloneqq \Lambda\otimes_{\iota,\Lambda^\circ}M$ denotes the $\Lambda$-module with underlying abelian group $M,$ but on which $g\in G$ acts as $g^{-1}$ for any $\Lambda^\circ$-module $M.$

  Given ${\varepsilon}_{R^\circ,-\epsilon}(\T^*(1))$ we may apply the dualising functor $-^*$ to obtain an isomorphism
  \[
   \varepsilon_{R^\circ,-{\xi}}(\T^*(1))^* : (\Det_{R^\circ}(R\Gamma({\Qp},\T^*(1) ))_{\widetilde{R^\circ}})^* (\Det_{R^\circ}(\T^*(1) )_{\widetilde{{R^\circ}}})^* \to \u_{\widetilde{{R^\circ}}},
  \]
  while the local Tate duality isomorphism \cite[\S 1.6.12]{fukayakato06}
  \[\psi(\T): R\Gamma({\Qp},\T)\cong R\Hom_{R^\circ}(R\Gamma({\Qp},\T^*(1)),R^\circ)[-2]\]
  induces an isomorphism
  \begin{align*}
   \overline{\Det_R(\psi(\T))_{\widetilde{R }}}^{-1}:&
   \left((\Det_{R^\circ}(R\Gamma({\Qp},\T^*(1)))_{\widetilde{R^\circ}})^*\right)^{-1}\cong\\
   &\Det_R(R\Hom_{R^\circ}(R\Gamma({\Qp},\T^*(1)),R^\circ))_{\widetilde{R }}^{-1} \to
   \Det_R(R\Gamma({\Qp},\T))_{\widetilde{R }}^{-1}.\notag
  \end{align*}
  Consider the product
  \[
   {\varepsilon}_{R,{\xi}}(\T)\cdot {\varepsilon}_{R^\circ,-{\xi}}(\T^*(1))^*\cdot \overline{\Det_R(\psi(\T))_{\widetilde{R }}}^{-1}:
   \Det_R(\T(-1))_{\widetilde{R}}\cong\Det_R(\T^*(1)^*)_{\widetilde{R }}\to\Det_R(\T)_{\widetilde{R}}
  \]
  and the isomorphism $\xymatrix{ {\T(-1)} \ar[r]^{\cdot{\xi}} & \T }$ which sends $t\otimes {\xi}^{\otimes -1}$ to $t$. The following result can be seen as another version of the reciprocity law \ref{thm:localrecip}.

  \begin{proposition}[Duality]\label{duality}
   Let $\T$ be as above such that $\T\cong Y\otimes_\Lambda\T(T)$ for some $(R,\Lambda)$-bimodule $Y$, which is projective as $R$-module. Then
   \[{\varepsilon}_{R,{\xi}}(\T)\cdot {\varepsilon}_{R^\circ,-{\xi}}(\T^*(1))^*\cdot \overline{\Det_R(\psi(\T))_{\widetilde{R }}}^{-1}=
   \Det_R\left(\xymatrix{{\T(-1) } \ar[r]^(0.6){\cdot{\xi}} & \T }\right)_{\widetilde{R}}.\]
  \end{proposition}

  \begin{proof}
   First note that the statement is stable under applying $Y'\otimes_R-$, for some $(R',R)$-bimodule $Y'$ which is projective as a $R'$-module, by the functoriality of local Tate duality and the lemma below. Thus we are reduced to the case $(R, \T)=(\Lambda, \T(T))$ where $T$ is a Galois stable lattice in some crystalline representation $V$.

   Since the morphisms between $\Det_R(\T(-1))_{\widetilde{R }}$ and $\Det_R(\T)_{\widetilde{R }}$ form a $K_1(\widetilde{\Lambda})$-torsor and the kernel
   \[SK_1(\widetilde{\Lambda}) \coloneqq \ker\left(K_1(\widetilde{\Lambda})\to \prod_{\rho\in{\rm{Irr}}({G})} K_1(\widetilde{L_\rho})\right)=1\]
   is trivial, as $G$ is abelian, it suffices to check the statement for all $(L, V(\rho)),$ which is nothing else than the content of \cite[prop.\ 3.3.8]{fukayakato06}. Here $ {\rm{Irr}}({G})\ $ denotes the set of $ \overline{\QQ}_p$-valued irreducible representations of $ {G}$ with finite image.
  \end{proof}

  \begin{lemma}
  Let $Y$ be a $(R',R)$-bimodule such that $Y\otimes_R\T\cong \T'$ as $(R',G_{{\Qp}})$-module and let
  $Y^*=\Hom_{R'}(Y,R')$ the induced $(R'^\circ,R^\circ)$-bimodule. Then there is a natural
  \begin{enumerate}
   \item equivalence of functors \[Y\otimes_R\Hom_{R^\circ}(-, R^\circ)\cong \Hom_{R'^\circ}(Y^*\otimes_{R^\circ} -,R'^\circ)\] on $P(R^\circ)$;
   \item isomorphism $Y^*\otimes_{R^\circ} \T^*\cong (\T')^*$ of $(R'^\circ,G_{{\Qp}})$-modules.
  \end{enumerate}
  \end{lemma}

  \begin{proof}
  This is easily checked using the adjointness of $\Hom$ and $\otimes.$
  \end{proof}

\section{Evaluation at characters}
\label{sect:evaluation}

 \subsection{Setup}

  For any de Rham character $\eta$ of $G$ (which we assume to take values in $L$), we write $\ev_{\eta}:\Lambda_{\cOt}(G) \to \cOt$ for the $\cOt$-linear ring homomorphism which sends $g$ to $\eta(g)$, and we abbreviate the functor
  \[\cOt \otimes_{\Lambda_{\cOt}(G), \ev_\eta}^\mathbb{L} (-) \quad\text{by}\quad {\rm sp}_\eta(-),\]
  where the tensor product is formed via $\ev_{\eta}$ as indicated.

  We want to study the image under $\sp_{\eta}$ of the isomorphism
  \[ \varepsilon_{\Lambda_\cO(G), \xi}(T): \Det_{\Lambda_{\cOt}(G)}(0) \rTo^\cong
    \Lambda_{\cOt}(G) \otimes_{\Lambda_{\cO}(G)} \left\{
     \Det_{\Lambda_{\cO}(G)} R\Gamma_{\Iw}(K_\infty, T) \cdot
     \Det_{\Lambda_{\cO}(G)} \left(\Lambda_{\cO}(G) \otimes_\cO T\right)\right\}
  \]
  constructed above. We have
  \[ \sp_\eta\left( R\Gamma_\Iw(K_\infty, T) \right) = R\Gamma(\Qp, T(\eta^{-1})), \]
  by \cite[1.6.5]{fukayakato06}, and
  \[ \sp_\eta\left( \Lambda_{\cOt}(G) \otimes_\cO T\right) = T(\eta^{-1}),\]
  (since clearly $T(\eta^{-1})$ is canonically isomorphic to $T$ as a $\cO$-module, although obviously not as a Galois representation). So $\sp_{\eta}$ is an isomorphism
  \[ \Det_{\cOt} (0) \rTo^\cong
      \cOt \otimes_{\cO} \left\{ \Det_{\cO} R\Gamma(\Qp, T(\eta^{-1})) \cdot \Det_{\cO} T(\eta^{-1})\right\}. \]

  \begin{theorem}
   \label{thm:compatibleunderevaluation}
    Let $V$ be a crystalline representation of $G_{\Qp}$ with coefficients in a finite extension $L$ of $\Qp$, $T$ a $\cO$-lattice in $V$, and $\eta$ an $L$-valued de Rham character of $G$, as above. Then the isomorphism
    \[ {\rm sp}_\eta\big(\varepsilon_{\Lambda_\cO(G), \xi}(T)\big)\]
    coincides with $\varepsilon_{L, \xi}(W)$ after extending scalars to $\Lt$, where $W$ is the de Rham representation $V(\eta^{-1})$ and $\varepsilon_{L, \xi}(W)$ is the canonical $\varepsilon$-isomorphism of \S \ref{sect:deRhamepsilon}.
  \end{theorem}

  It is clear that $\ev_{\eta}$ extends to a homomorphism $\Lambda_{\Lt}(G) \to L$, and that the composition of the exact functor $\Lt \otimes_{\cOt} (-)$ with $\sp_{\eta}$ coincides with the derived tensor product
  \[ \Lt \otimes_{\Lambda_{\Lt}(G), \ev_\eta} (-),\]
  which (in a slight abuse of notation) we shall also denote by $\sp_{\eta}$. So it suffices to show that
  \[ {\rm sp}_\eta\big(\varepsilon_{\Lambda_L(G), \xi}(V)\big) = \varepsilon_{L, \xi}(W);
  \]
  this implies Theorem \ref{thm:compatibleunderevaluation} for all lattices $T \subset V$. By the invariance of $\varepsilon_{\Lambda_L(G), \xi}$ under twists by crystalline characters, it suffices to prove the theorem under the additional hypothesis that $V$ has non-negative Hodge--Tate weights.

 We shall divide the de Rham characters of $G$ into the following three classes. Suppose $h$ is the largest Hodge--Tate weight of $V$.

 \begin{itemize}
  \item {\em Good characters:} these are characters of $G$ whose Hodge--Tate weights are $\ge h$ or $\le -1$, and such that the twisted representation $W = V(\eta^{-1})$ satisfies $H^0(\Qp,V(\eta^{-1}))=H^2(\Qp,V(\eta^{-1}))=0$.
  \item {\em Somewhat bad characters:} these are characters whose Hodge--Tate weights lie in $[0, h-1]$, but still such that $H^0(\Qp,V(\eta^{-1}))=H^2(\Qp,V(\eta^{-1}))=0$.
  \item {\em Extremely bad characters:} $\eta$ is extremely bad if $H^i(\Qp,V(\eta^{-1}))\neq 0$ for $i=0$ or $i=2$.
 \end{itemize}

 Note that somewhat bad characters almost always exist (they exist unless $V$ has all Hodge--Tate weights 0, in which case $V$ is unramified), but extremely bad characters are rarer; in particular, there are none if $V$ is irreducible and $d > 1$.


 \subsection{Evaluation of the Gamma factor}
  \label{sect:gamma}
  As a preliminary to the proof of Theorem \ref{thm:compatibleunderevaluation}, we need to compare the factor $\Gamma_L(V(\eta^{-1}))$ defined in \ref{sect:deRhamepsilon} with the factor $\ell(V)$ arising in the definition of $\Theta_{\Lambda_L(G), \xi}(V)$.

  \begin{proposition}
   \label{prop:factorials}
   Let $\eta$ be an $L$-valued character of $G$ which is de Rham, with Hodge--Tate weight $j$, and let $W=V(\eta^{-1})$. Then we have
   \[ \frac{\Gamma^*(1 + j)^d} {\ell(V)^*(\eta)} = (-1)^{\sum n_i + jd + r}\, \Gamma_L(W),\]
   where $r = \#\{ i : n_i > j\} = \dim_L t(W)$.
  \end{proposition}

  \begin{proof}
   For any $n \ge 0$, we have
   \[ (\ell_0 \dots \ell_{n-1})^*(\eta) = \prod_{\substack{0 \le k \le n-1 \\ k \ne j}} (j - k) =
    \begin{cases}
     \frac{j!}{(j-n)!} & \text{if $j \ge n$,}\\
     j! (n - 1 - j)! (-1)^{n - 1 - j} & \text{if $0 \le j \le n-1$,} \\
     (-1)^n \frac{(n-1-j)!}{(-1-j)!} & \text{if $j \le -1$.}
    \end{cases}
   \]
   Hence
   \begin{align*}
    \frac{1}{\Gamma^*(1 + j)} (\ell_0 \dots \ell_{n-1})^*(\eta) &=
    \begin{cases}
     \frac{1}{(j-n)!} & \text{if $j \ge n$,}\\
     (-1)^{n-1-j} (n - 1 - j)! & \text{if $j \le n-1$.} \\
     \end{cases}\\
    &= \begin{cases}
     (-1)^{n-j} \Gamma^*(n-j) & \text{if $j \ge n$,}\\
     (-1)^{n-1-j} \Gamma^*(n - j) & \text{if $j \le n-1$.} \\
    \end{cases}
   \end{align*}

   Taking $n$ to be each of the Hodge--Tate weights of $V$ in turn and multiplying, we obtain
   \[ \frac{1}{\Gamma^*(1 + j)^d} \ell(V)^*(\eta) = (-1)^{\sum n_i + jd + r} \prod_{i = 1}^d \Gamma^*(n_i - j) = (-1)^{\sum n_i + jd + r} \Gamma_L(W)^{-1},
   \]
   since the Hodge--Tate weights of $W$ are $\{n_i - j\}_{i = 1, \dots, d}$.
  \end{proof}

 \subsection{The good characters}\label{sect:good}

  In this section, we prove Theorem \ref{thm:compatibleunderevaluation} for good characters of $G$. As remarked above, it suffices to assume that $V$ has non-negative Hodge--Tate weights, and that the character $\eta$ takes values in $L^\times$. We write $W = V(\eta^{-1})$.

  \begin{proposition}\label{descent-good}
   Let $\eta$ be an $L$-valued de Rham character of $G$ whose Hodge--Tate weight $j$ does not lie in $[0, h-1]$, and such that $H^0(\Qp, W) = H^2(\Qp, W) = 0$.

   Then
   \begin{enumerate}[(i)]
    \item The corestriction map $H^1_{\Iw}(K_\infty, V)_{G = \eta} \to H^1(\Qp, W)$ is an isomorphism of $L$-vector spaces, so
    \[ \Det_L R\Gamma(\Qp, W) \cong \left(\Det_L H^1_{\Iw}(K_\infty, V)_{G = \eta}\right)^{-1};\]

    \item composing the regulator with the evaluation map $\ev_\eta$ induces an isomorphism of free $\Lt$-modules
    \[ \ev_{\eta} \circ \cL^G_{V, \xi} : \Lt \otimes_L H^1_\Iw(K_\infty, V)_{G = \eta} \rTo^\cong \Lt \otimes_L \Dcris(V);\]

    \item the isomorphism
    \[ \Lt \otimes_L \Det_L H^1_{\Iw}(K_\infty, V)_{G = \eta} \rTo \Lt \otimes_L \Det_L(V)\]
    coming from $\sp_{\eta} \big(\varepsilon_{\Lambda_L(G), \xi}(V)\big)$ via (i) is given by the map
    \[ (-\eta(\gamma_{-1}))^d \ell(V)(\eta)^{-1} \Det_{\Lt}(\operatorname{ev}_{\eta} \circ \cL^G_{V, \xi})) \cdot \varepsilon_{L, \xi, \dR}(V);\]

    \item the isomorphism of (iii) coincides with the canonical isomorphism $\varepsilon_{L, \xi}(W)$ of Section \ref{sect:deRhamepsilon} above, so Theorem \ref{thm:compatibleunderevaluation} holds for $\eta$.
   \end{enumerate}
  \end{proposition}

  \begin{proof}
   For (i), we have the exact sequence
   \begin{multline*}
    0 \to H^1(G, H^1_{\Iw}(K_\infty, W)) \to H^0(\Qp, W) \to H^2_{\Iw}(K_\infty, W)^G \\ \to H^1_{\Iw}(K_\infty, W)_G \rTo^{\operatorname{cores}} H^1(\Qp, W) \to H^1(G, H^2_{\Iw}(K_\infty, W)) \to 0.
   \end{multline*}
   given by the Tor spectral sequence for $\sp_{\eta}$. By Tate duality, $H^2_{\Iw}(K_\infty, W)^G\cong H^0(\Qp,W^*(1))$, which is zero by assumption. However, since $H^2_{\Iw}(K_\infty, W)$ is finite-dimensional, it decomposes as a finite direct sum of primary submodules corresponding to characters of $G$; if the $G$-invariants are zero then the trivial character cannot appear, and any other direct summands have zero $G$-cohomology in all degrees, so we also have $H^1(G, H^2_{\Iw}(K_\infty, W)) = 0$. Thus corestriction is an isomorphism $H^1_{\Iw}(K_\infty, W)_G \cong H^1(\Qp, W)$, and since $W$ and $V$ are isomorphic as $G_{K_{\infty}}$-representations, we have $H^1_{\Iw}(K_\infty, W)_G = \left(H^1_{\Iw}(K_\infty, V)(\eta^{-1})\right)_G = H^1_{\Iw}(K_\infty, V)_{G=\eta}$.

   Now let us suppose that $n \ge 1$, where $n$ is the conductor of $\eta$. By \cite[Theorem 4.16]{loefflerzerbes11} we have a commutative diagram of free $\Lt$-modules:
   \begin{diagram}[width=3cm, height=3.5cm]
    \Lt \otimes_L H^1_{\Iw}(\Qp, V)_{G = \eta} & \rTo^{\operatorname{ev}_{\eta} \circ \cL^G_V} & \Lt \otimes_L \Dcris(V) \\
    & \rdTo_{\Gamma^*(1 + j) \varepsilon_L(\eta^{-1}, -\xi) \Phi^n \genfrac{\{}{.}{0pt}{}{\exp^*}{\log}} & \dTo \\
    && \Qpn \otimes_{\Qp} \Lt \otimes_L \DdR(W)
   \end{diagram}
   Here the vertical map is given by the isomorphism
   \[ b_{\eta^{-1}} : \QQ_{p, n} \otimes_{\Qp} \Lt \otimes_L\Dcris(V) \cong \QQ_{p, n} \otimes_{\Qp} \Lt \otimes_L \DdR(W)\]
   given by multiplication by $t^j$ in $\BdR \otimes V$ (which depends on the choice of $\xi$, since $\xi$ determines $t$); $\Phi$ is the unique $\BdR$-linear endomorphism of $\BdR \otimes V$ coinciding with the crystalline Frobenius on $\Dcris(V)$; and the bracket in the diagonal map denotes either $\exp^*_{\Qp, W^*(1)}$ (if $j \ge h$) or $\log_{\Qp, W}$ (if $j \le -1$).

   In either case, the diagonal map is clearly an isomorphism, which proves (ii). Part (iii) now follows from the definition of $\varepsilon_{\Lambda_L(G), \xi}(V)$ together with the compatibility of determinant functors and Tor spectral sequences, cf.~\cite{venjakob12} (since the Tor spectral sequence for $\sp_{\eta}$ collapses in this case). Let us prove (iv). By Proposition \ref{prop:noncrystallinedet}, for $j \ge h$ the determinant of $\log$ is $\theta_L(W)$, and for $j \le -1$ the determinant of $\exp^*$ is $(-1)^d \theta_L(W)$. We write this as $(-1)^{d - r}\theta_L(W)$, where $r = \dim t(W)$ as in Proposition \ref{prop:factorials}. Passing to determinants and dividing through by the factor $\ell(V)(\eta) \in L^\times$, the diagonal arrow becomes
   \begin{gather*}
    \frac{\Gamma^*(1 + j)^d}{\ell(V)(\eta)} \varepsilon_L(\eta^{-1}, -\xi) \det(\vp)^d (-1)^{d - r} \theta_L(W)\\
    = (-1)^{(j+1)d + \sum n_i} \Gamma_L(W) \varepsilon_L(W, \xi^{-1}) \theta_L(W)\\
    = (-1)^{d + m(W)} \Gamma_L(W) \varepsilon_L(W, \xi^{-1}) \theta_L(W)
   \end{gather*}
   where we have used the formula for $\ell(V)(\eta)$ from the previous section, and written $m(W)$ for the sum of the Hodge--Tate weights of $W$.

   Hence the following diagram commutes:
   \begin{diagram}
    \Det_{(\BdR \otimes L)} (\BdR \otimes_{\Qp} H^1(\Qp, W)) & \rTo^{\operatorname{sp}_{\eta} \Theta_{\Lambda_L(G)}(V)} & \Det_{(\BdR \otimes L)} (\BdR \otimes_{\Qp} \Dcris(V)) & \rTo^{t^{m(V)} \cdot \operatorname{can}} & \Det_{\BdR \otimes L} (\BdR \otimes_{\Qp} V)\\
    & \rdTo_{(-1)^{d + m(W)} \Gamma_L(W) \varepsilon_L(W, \xi) \theta_L(W)} & \dTo & & \dEq\\
    & & \Det_{\Lt} \left(\BdR \otimes_{\Qp} \DdR(W)\right) & \rTo^{t^{m(W)} \cdot \operatorname{can}} & \Det_{\BdR \otimes L} (\BdR \otimes_{\Qp} W)
   \end{diagram}
   where the middle vertical map is multiplication by $t^{dj}$. Both the left-hand triangle and the right-hand square clearly commute. But $t^{m(V)} \operatorname{can} = \varepsilon_{L, \xi, \dR}(V)$, since $V$ is crystalline and hence the $\varepsilon$-factor $\varepsilon_L(W, \xi)$ is 1. So the composition of the two arrows on the top row is
   \[ \operatorname{sp}_{\eta}\left(\Theta_{\Lambda_L(G)}(V)\right) \varepsilon_{L, \xi, \dR}(V) = (-1)^{m(V)} (-\eta(\gamma_{-1}))^d \operatorname{sp}_{\eta}\left(\varepsilon_{\Lambda_L(G), \xi}(V)\right)\]
   by definition. On the other hand, the composition of the diagonal arrow and $t^{m(W)} \cdot \operatorname{can}$ is
   \begin{align*}
    (-1)^{d + m(W)} \Gamma_L(W) \varepsilon_L(W, -\xi) \theta_L(W) t^{m(W)} \cdot \operatorname{can} &=
    (-1)^{d + m(V)} \eta(\gamma_{-1})^d \Gamma_L(W) \varepsilon_L(W, \xi) \theta_L(W) t^{m(W)} \cdot \operatorname{can} \\
    &= (-1)^{m(V)} (-\eta(\gamma_{-1}))^d \Gamma_L(W) \theta_L(W) \varepsilon_{L, \xi, \dR}(W).
   \end{align*}
   Cancelling out the factor $(-1)^{m(V)} (-\eta(\gamma_{-1}))^d$, we deduce that $\operatorname{sp}_{\eta}\left(\varepsilon_{\Lambda_L(G), \xi}(V)\right) = \varepsilon_{L, \xi}(W)$ as required.

   We now consider the case $n = 0$, so $\eta$ is $\eta^j$ times an unramified character. In this case one obtains a diagram very similar to the above, but with $\Phi^{n}$ in the diagonal map replaced by the operator
   \[ (1 - p^j \eta(\sigma_p) \Phi) (1 - p^{-1-j} \eta(\sigma_p)^{-1} \Phi^{-1})^{-1} .\]
   Since $p^j \eta(\sigma_p) \Phi$ coincides with the Frobenius $\vp$ of $\Dcris(W)$, the determinant of
   \[ (1 - p^j \eta(\sigma_p) \Phi) (1 - p^{-1-j} \eta(\sigma_p)^{-1} \Phi^{-1})^{-1} \genfrac{\{}{.}{0pt}{}{\exp^*}{\log}\]
   is the base-extension to $\Lt$ of $(-1)^{d-r} \theta_L(W)$, by Proposition \ref{prop:nophidet}, and the proof goes through as before.
  \end{proof}


 \subsection{The somewhat bad characters}
  \label{sect:somewhatbad}

  Let us now suppose $\eta$ is ``somewhat bad'' in the sense above (recall that ``somewhat bad'' excludes ``extremely bad'', so $H^0(\Qp, W) = H^2(\Qp, W) = 0$). By the twist-compatibility of the $\varepsilon$-isomorphisms we may assume that $\eta$ factors through $\Gamma$. Let $\mathfrak{p}$ be the ideal of $\Lambda_L(\Gamma)$ corresponding to $\eta$, and define
  \[ \varpi = \frac{\gamma - \eta(\gamma)}{\eta(\gamma) \log \chi(\gamma)},\]
  so $\varpi$ is the unique uniformizer of $\mathfrak{p}$ such that $\varpi'(\eta) = 1$. Note that $\varpi$ is not a zero-divisor in $\Lambda(\Gamma)$.

  We also denote by $\mathfrak{p}$ the ideal of $\cH_L(\Gamma)$ above $\eta$. The inclusion $\Lambda_L(\Gamma) \into \cH_L(\Gamma)$ induces an isomorphism after localization at $\mathfrak{p}$ and completion (both completions are isomorphic to $L[[\varpi]]$).

  Since $\eta$ is \emph{not} ``extremely bad'', we know that the localization of $H^2_{\Iw}(\Qpi, V)$ at the prime ideal $\mathfrak{p}$ corresponding to $\eta$ is zero; the localization of $H^1_{\Iw}(\Qpi, V)$ at $\mathfrak{p}$ is free of rank $d$; and reduction modulo $\mathfrak{p}$ determines an isomorphism
  \[ \Lambda(\Gamma) / \mathfrak{p} \otimes_{\Lambda(\Gamma)} H^1_{\Iw}(\Qpi, V) \cong H^1(\Qp, W).\]
  Let $y_1, \dots, y_r$ be any basis of $H^1_f(\Qp, W)$, and $y_{r+1}, \dots, y_d$ any basis of the quotient $H^1(\Qp, W) / H^1_f(\Qp, W)$. Then there exists a lifting $x_1, \dots, x_d$ of $y_1, \dots, y_d$ to a basis of the localization $H^1_{\Iw}(\Qpi, V)_{\mathfrak{p}}$.

  By Theorem \ref{thm:cycloregulator} (resp.~Theorem \ref{thm:strongcycloregulator} if $\eta$ is crystalline) we know that for $1 \le i \le r$ we have $\cL^\Gamma_{V, \xi}(x_j)(\eta) = 0$, and hence (by the definition of $\varpi$) we have
  \[ \cL^\Gamma_{V, \xi}(x_j) = \varpi \cL^\Gamma_{V, \xi}(x_j)'(\eta) \bmod{\mathfrak{p}^2}.\]

  Let $A$ denote the unique $\Lambda(\Gamma)_{\mathfrak{p}}$-linear map
  \[ H^1_{\Iw}(\Qpi, V)_{\mathfrak{p}} \rTo \Lambda(\Gamma)_{\mathfrak{p}} \otimes \Dcris(V)\]
  such that
  \[ A(x_j) =
   \begin{cases}
    \frac{1}{\varpi} \cL^\Gamma_{V, \xi}(x_j) & \text{if $1 \le j \le r$,}\\
    \cL^\Gamma_{V, \xi}(x_j) & \text{if $r+1 \le j \le d$.}
   \end{cases}
  \]
  This is well-defined, since $x_1, \dots, x_d$ are a free basis of $H^1_{\Iw}(\Qpi, V)_{\mathfrak{p}}$ over $\Lambda(\Gamma)_{\mathfrak{p}}$. We write $B$ for the morphism obtained by reducing modulo $\mathfrak{p}$.

  \begin{proposition}
   The determinant of $A$ is equal to the image of $\varpi^{-r} \Det(\cL^\Gamma_{V, \xi})$ under localization at $\mathfrak{p}$.
  \end{proposition}

  \begin{proof}
   Clear from the definition of the map $A$.
  \end{proof}

  We shall show that the reduction $B$ is an isomorphism; it follows that $A$ is also an isomorphism, and that the image of the determinant of $\Det(A)$ modulo $\mathfrak{p}$ is just $\Det(B)$.

  \begin{proposition}
   The image of $H^1_f(\Qp, W)$ under $B$ is a subspace of $\Dcris(V)$ complementary to the subspace
   \[
    M \coloneqq \begin{cases}
     \vp^n \Fil^{-j} \Dcris(V) & \text{if $n \ge 1$,}\\
     (1 - p^j \vp)(1 - p^{-1-j} \vp^{-1})^{-1} \Fil^{-j} \Dcris(V) & \text{if $n = 0$.}
    \end{cases}
   \]
   Moreover, the induced morphism $H^1_f(\Qp, W) \rTo^B \Dcris(V) \rTo \frac{\Dcris(V)}{M}$ is an isomorphism, and it is given explicitly by
   \[ \Gamma^*(1 + j)\varepsilon(\eta^{-1}, -\xi) \cdot
    \begin{cases}
     \vp^n \left[\log_{\Qp, W} \otimes t^{-j} e_j\right] & \text{if $n \ge 1$,}\\[2mm]
     -(1 - p^{-1-j}\vp^{-1})^{-1} \left[\widetilde\log_{\Qp, W}\otimes t^{-j} e_j\right] & \text{if $n = 0$.}
    \end{cases}
   \]
  \end{proposition}

  \begin{proof}
   The fact that the composite map is given by the formula above follows directly from Theorem \ref{thm:regderiv}, since the uniformizer $\varpi$ is chosen such that $\varpi'(\eta) = 1$. It remains to check that the composite is an isomorphism. For $n \ge 1$, the map
   \[ x \mapsto \Gamma^*(1 + j) \varepsilon(\eta^{-1}, -\xi)\vp^n \left(t^{-j} x \otimes e_j\right)\]
   defines an isomorphism $t(W) \cong \frac{\Dcris(V)}{M}$, and the map of the proposition is the composite of this and
   \[ \log_{\Qp, W} : H^1_f(\Qp, W) \rTo^\cong t(W).\]
   Similarly, in the case $n = 0$ the map $H^1_f(\Qp, W) \rTo^B \Dcris(V) \rTo \frac{\Dcris(V)}{M}$ is the composite of the twisting isomorphism $s(W) \cong \frac{\Dcris(V)}{M}$ given by tensoring with $t^{-j} e_j$ and the morphism
   \[ -(1 - p^{-1}\vp^{-1})^{-1} \widetilde\log_{\Qp, W} : H^1_f(\Qp, W) \to s(W)\]
   of \eqref{iso4} above.
  \end{proof}

  \begin{proposition}
   The image of the subspace $N \subseteq H^1(\Qp, W)$ spanned by $y_{r+1}, \dots, y_d$ under $B$ is $M$, and the composite isomorphism
   \[ \frac{H^1(\Qp, W)}{H^1_f(\Qp, W)} \rTo^\cong N \rTo^B M\]
   is given by
   \[ \Gamma^*(1 + j)\varepsilon(\eta^{-1}, -\xi) \cdot
    \begin{cases}
     \vp^n \left[\exp^*_{\Qp, W^*(1)} \otimes t^{-j} e_j\right] & \text{if $n \ge 1$,}\\[2mm]
     -(1 - p^{j}\vp) \left[\widetilde\exp_{\Qp, W^*(1)} \otimes t^{-j} e_j\right] & \text{if $n = 0$.}
    \end{cases}
   \]
  \end{proposition}

  \begin{proof}
   The explicit formula follows from Theorems \ref{thm:cycloregulator} and \ref{thm:strongcycloregulator}, and it follows from equation \eqref{iso1} that the composite morphism is an isomorphism (via an argument very similar to the previous proposition).
  \end{proof}

  Combining these two propositions we have the following:

  \begin{proposition}
   \label{prop:somewhatbad}
   The image of $\Det(A)$ modulo $\mathfrak{p}$ is
   \begin{multline*}
      \Gamma^*(1 + j)^d \varepsilon_L(\eta^{-1}, -\xi)^d \cdot \Det_L \left(\varphi: \Dcris(V) \to \Dcris(V)\right)^n \\ \Det_L\left(\log : H^1_f(\Qp, W) \rTo t(W)\right) \cdot \Det_L\left(\exp^* : \frac{H^1(\Qp, W)}{H^1_f(\Qp, W)} \rTo \Fil^0 \DdR(W)\right)
   \end{multline*}
   if $n \ge 1$; and if $n = 0$ it is
   \begin{multline*}
     \Gamma^*(1 + j)^d \Det_L\left(-(1 - p^{-1}\vp^{-1})^{-1} \widetilde\log : H^1_f(\Qp, W) \rTo \frac{\Dcris(W)}{(1 - \vp)(1 - p^{-1}\vp^{-1})^{-1} \Fil^0 \Dcris(W)}\right) \\ \cdot \Det_L\left(-(1 - \vp)\widetilde\exp^* : \frac{H^1(\Qp, W)}{H^1_f(\Qp, W)} \rTo \Fil^0 \DdR(W)\right)
   \end{multline*}
  \end{proposition}

  We now combine this with the result of Proposition \ref{prop:factorials}, which shows that $\ell(V)$ has a zero of degree $r$ at $\eta$, and
  \[ \frac{\Gamma^*(1 + j)^d}{\ell(V)(\eta)} = (-1)^{m(V) + jd + r} \Gamma_L(W) \varpi^{-r} \bmod{\varpi^{1-r}}. \]
  This shows that for $n \ge 1$, the image of $\Theta_{\Lambda_L(G), \xi}(V) = \Det \frac{\cL^\Gamma_{V, \xi}}{\ell(V)}$ modulo $\mathfrak{p}$ is
  \begin{multline*}
   (-1)^{m(V) + jd + r} \Gamma_L(W) \varepsilon_L(\eta^{-1}, -\xi)^d \cdot \Det_L \left(\varphi: \Dcris(V) \to \Dcris(V)\right)^n \\
   \Det_L\left(\log : H^1_f(\Qp, W) \rTo t(W)\right) \cdot \Det_L\left(\exp^* : \frac{H^1(\Qp, W)}{H^1_f(\Qp, W)} \rTo \Fil^0 \DdR(W)\right)
  \end{multline*}
  or (grouping the $(-1)$'s differently)
  \begin{multline*} (-1)^{d + m(W)} \Gamma_L(W) \varepsilon_L(W, -\xi) \cdot \Det_L\left(\log : H^1_f(\Qp, W) \rTo t(W)\right) \\ \cdot \Det_L\left(-\exp^* : \frac{H^1(\Qp, W)}{H^1_f(\Qp, W)} \rTo \Fil^0 \DdR(W)\right)
  \end{multline*}
  In the case $n = 0$ the result becomes
  \begin{multline*} (-1)^{d + m(W)} \Gamma_L(W) \varepsilon_L(W, -\xi) \\ \cdot \Det_L\left(-(1 - p^{-1}\vp^{-1})^{-1}\log : H^1_f(\Qp, W) \rTo \frac{\Dcris(W)}{(1 - \vp)(1-p^{-1}\vp^{-1})^{-1} \Fil^0 \Dcris(W)}\right) \\ \cdot \Det_L\left((1 - \vp)\exp^* : \frac{H^1(\Qp, W)}{H^1_f(\Qp, W)} \rTo (1 - \vp)(1-p^{-1}\vp^{-1})^{-1} \Fil^0 \Dcris(W)\right).
  \end{multline*}

  Using Proposition \ref{prop:noncrystallinedet} for $n \ge 1$, and Theorem \ref{thm:sigmas} in the case $n = 0$, and invoking again the compatibility of determinants with Tor spectral sequences, we see that in both cases the specialization of $\ell(V)^{-1} \Det_{\cK(\Gamma)} \cL^\Gamma_{V, \xi}$ at $\eta$ is
  \[ (-1)^{d + m(W)} \Gamma_L(W) \varepsilon_L(W, -\xi) \theta_L(W),\]
  as in the case of good characters in the previous section. The remainder of the proof continues exactly as before, and we deduce that
  \[ \operatorname{sp}_{\eta}(\varepsilon_{\Lambda_L(G), \xi}(V)) = \varepsilon_{L, \xi}(W). \]

 \subsection{The extremely bad characters}
  \label{sect:extremelybad}

  Let $\eta$ be an extremely bad character of $G$, and let $W=V(\eta^{-1})$. Recall that our aim is to prove the following statement:

  \begin{proposition}\label{prop:restatement}
   We have
   \[ {\rm sp}_\eta\big(\Theta_{\Lambda_L(G), \xi}(V)\big)\cdot \varepsilon_{L,\xi,\dR}(V)=\varepsilon_{L(\eta),\xi}(W),\]
   where ${\rm sp}_\eta$ denotes specialisation at $\eta$.
  \end{proposition}

  We prove the proposition by induction on $d=\dim_L W$. If $d=1$, then $W=L$ or $W=L(1)$, and the result is the content of \cite[Theorem 2.13]{venjakob11}, once having checked that the epsilon-isomorphisms defined in \emph{loc. cit.} and here agree. This can be seen from the fact that both agree with the Fukaya--Kato $\varepsilon$-isomorphism after specialization at any good (or somewhat bad) character, together with the fact that good characters are Zariski-dense in $\operatorname{Spec} \Lambda_L(G)$; however, we give a more direct proof in Appendix \ref{App-comparison}.

  Now assume that $d>1$, and that the proposition is true for all $d' < d$. Then the assumption that $H^i(\Qp,W)\neq 0$ for $i=0$ or $i=2$ implies that we can find a subrepresentation $W'$ of dimension $<d$ such that we have a short exact sequence
  \[ 0\rTo W'\rTo W\rTo W/W'\rTo 0.\]
  After twisting, this induces a short exact sequence
  \begin{equation*}\label{SES2}
   0 \rTo W'(\eta)\rTo V\rTo V/W'(\eta)\rTo 0.
  \end{equation*}
  Note that as $V$ is crystalline, so are $W'(\eta)$ and $V/W'(\eta)$. By induction hypothesis and the results in Sections \ref{sect:good} and \ref{sect:somewhatbad}, Proposition \ref{prop:restatement} is true for the representations $W'(\eta)$ and $V/W'(\eta)$. As we know that
  \begin{align*}
   \Theta_{\Lambda_L(G), \xi}(V) &= \Theta_{\Lambda_L(G), \xi}(W'(\eta)) \cdot\Theta_{\Lambda_L(G), \xi}(V/W'(\eta)) \\
   \varepsilon_{L,\xi,\dR}(V) &= \varepsilon_{L,\xi,\dR}(W'(\eta))\cdot \varepsilon_{L,\xi,\dR}(V/W'(\eta))\\
   \varepsilon_{L(\eta),\xi}(W)&= \varepsilon_{L(\eta),\xi}(W')\cdot \varepsilon_{L(\eta),\xi }(W/W')
  \end{align*}
  by Proposition \ref{prop:regulatorSEScompatible}, Lemma \ref{lem:SEScompatible} and Proposition \ref{prop:SEScompatible}, this finishes the proof.

\appendix

\section{A formulary for the p-adic regulator map}
 \label{sect:app-formulary}

 In this appendix, we will prove a strengthening of the explicit formulae of \cite[Appendix B]{loefflerzerbes11} which determines, loosely speaking, the ``leading term'' of the $p$-adic regulator map at \emph{every} de Rham character of $\Gamma$, including the case of ``bad'' characters where there are Frobenius eigenvalues equal to $1$ or $p$.

 \subsection{The big exponential map}

  We will begin by quoting results regarding Perrin-Riou's big exponential map, which will allow us to study the regulator map using the reciprocity law of \ref{thm:localrecip}.

  Let $V$ be a crystalline representation whose Hodge--Tate weights lie in $[-\infty, h]$, for some integer $h \ge 1$. Define a map
  \[ \Delta: \cH(\Gamma) \otimes \Dcris(V) \to \bigoplus_{k = 0}^h \frac{\Dcris(V)}{(1 - p^k \vp)\Dcris(V)}(k)\]
  as the direct sum of the obvious projection maps.

  Then the Perrin-Riou exponential map is the map
  \[ \Omega_{V, h, \xi} : \left(\cH(\Gamma) \otimes_{\Qp} \Dcris(V)\right)^{\Delta = 0} \rTo \cH(\Gamma) \otimes_{\Lambda_{\Qp}(\Gamma)} H^1_{\Iw}(\Qp, V) \]
  defined by
  \[\label{tilde}\Omega_{V, h, \xi}(z) = \left(\ell_{h-1} \circ \dots \circ \ell_{0}\right) (1 - \vp)^{-1}(\tilde{z}),\]
  where $\tilde{z}$ denotes the image of $z$ under the isomorphism
  \[\cH(\Gamma) \otimes_{\Qp} \Dcris(V)\to \left({\Brig}\otimes \Dcris(V)\right)^{\psi = 0}\]
  which sends $\sum f_i \otimes d_i$ to $\sum f_i(1+\pi)\otimes d_i$.

  (This is not quite the definition of $\Omega_{V, h, \xi}$ originally given in \cite{perrinriou94}, but it is shown in \cite[Theorem II.13]{berger03} that the above formula does give a well-defined map and that this map agrees with Perrin-Riou's original definition, modulo choices of signs.)

  \begin{theorem}[Perrin-Riou, cf.~{\cite[Theor\'eme 3.3]{perrinriou99}}]
   \label{thm:expformula}
   There exists an extension of $\Omega_{V, h, \xi}$ to a morphism of $\cH(\Gamma)$-modules
   \[ \cH(\Gamma) \otimes_{\Qp} \Dcris(V) \rTo \cH(\Gamma) \otimes_{\Lambda_{\Qp}(\Gamma)} \frac{H^1_{\Iw}(\Qp, V)}{H^1_{\Iw}(\Qp, V)_{tors}},\]
   coinciding on $\ker(\Delta)$ with the map defined above.

   Moreover, the values of $\Omega_{V, h, \xi}(z)$ at all de Rham characters $\eta$ of Hodge--Tate weights $j \le h-1$ are given by the following formulae, which hold modulo the image of $H^1_{\Iw}(\Qp, V)_{tors}$ in $H^1(\Qp, V(\eta^{-1}))$:
   \begin{enumerate}
    \item if $\eta$ has positive conductor $n \ge 1$, then
    \[
     \operatorname{pr}_\eta \left( \Omega_{V, h, \xi}(z) \right) =
     (-1)^{h - j - 1}(h - j - 1)! \exp_{\Qp, V(\eta^{-1})}
     \left(
       \tau(\eta_0, \xi) \cdot p^{-n} \vp^{-n} \left( z(\eta) \otimes t^{j} e_{-j} \right)
      \right),
    \]
    where the Gauss sum $\tau(\eta_0, \xi)$ is as in Propsition \ref{prop:gaussums} above.
    \item if $\eta = \chi^j$, then we have
    \[ \operatorname{pr}_\eta\left( \Omega_{V, h, \xi}(z) \right) = -(-1)^{h - j - 1}(h - j - 1)!\operatorname{\widetilde\exp}_{\Qp, V(\eta^{-1})}\left[(1 - p^{-1} \vp^{-1})\left( z(\eta) \otimes t^{j} e_{-j} \right)\right],\]
    where $\widetilde\exp$ is as defined in \S \ref{sect:specialcases} above.
   \end{enumerate}
  \end{theorem}

  \begin{proof}
   If $n \ge 1$ or if $z(\eta) \in (1 - p^{-j} \vp) \Dcris(V)$, then we may assume that $z \in \ker(\Delta)$ and this is then a standard formula, equivalent to the commutative diagram relating $\Omega_{V, h}$ to the exponential maps (see e.g. page 121 of \cite{berger03}\footnote{Sadly, there seems to be a recurring ambiguity in the literature regarding the signs (cf.~Remark II.17 of \cite{berger03}). We use Berger's definitions, but we note that there are two errors in the commutative diagram on page 121 of \cite{berger03}: firstly, the map $\Omega_{V(j), h}$ in the top row should be $\Omega_{V(j), h+j}$; secondly, the sign $(-1)^{h + j - 1}$ is missing. We believe the diagram above to be the correct one.}.) The awkward case when $z(\eta) \otimes t^{j} e_{-j}$ is not in the image of $1 - \vp$ is covered in \cite{perrinriou99}.
  \end{proof}

 \subsection{The regulator at bad characters}

  We shall use Theorem \ref{thm:expformula} to study the values of the regulator map at those characters where the factor on the left hand side of formula (2) in Theorem \ref{thm:cycloregulator} is not injective. We relate these values to the extended dual exponential map $\widetilde\exp^*$; given the indirect nature of the definition of this map, we have no choice but to exploit the duality between $V$ and $V^*(1)$.

  \begin{proposition}
   The regulator map $\cL^\Gamma_{V, \xi}$ and the exponential $\Omega_{V, h, \xi}$ satisfy the formula
   \begin{equation}
    \label{eq:regexp}
    \Omega_{V, h, \xi}\left(\cL_{V}(x)\right) = \ell_{h-1} \circ \dots\circ \ell_0(x) \pmod{H^1_{\Iw}(\Qpi, V)_{tors}}
   \end{equation}
   for all $x \in \cH(\Gamma) \otimes_{\Lambda(\Gamma)} H^1_{\Iw}(\Qp, V)$.
  \end{proposition}

  \begin{proof}
   If $\cL^\Gamma_{V, \xi}(x)$ lies in $\ker(\Delta)$, then this is obvious from Berger's redefinition of $\Omega_{V, h, \xi}$ given above. However, since the target of $\Delta$ is a torsion $\Lambda(\Gamma)$-module, and we have quotiented out by the torsion in the target $H^1_{\Iw}(\Qpi, V)$, this implies that the formula holds for all $x$.
  \end{proof}

  Recall from proposition \ref{prop:defPRpairing} that there is a pairing
  \[ \langle -, - \rangle_{\Iw}: H^1_{\Iw}(\Qp, V) \times H^1_{\Iw}(\Qp, V^*(1))^\iota \rTo \Lambda_L(\Gamma),\]
  where the superscript $\iota$ indicates that the pairing is anti-linear in the second variable, with the property that
  \[ \langle x, y \rangle_{\Iw}(\eta) = \langle x_{\eta}, y_{\eta^{-1}}\rangle_{\text{Tate}}\]
  where $\langle -, -\rangle_{\text{Tate}}$ is the local Tate duality pairing $H^1(\Qp, V(\eta^{-1})) \times H^1(\Qp, V^*(1)(\eta)) \to L$.

  We extend the Perrin-Riou pairing to a pairing of $\cH(\Gamma)$ modules in the obvious way. We also define a pairing $\langle , \rangle_{\Iw, \cris}$ by extending the natural pairing $\Dcris(V) \times \Dcris(V^*(1)) \to \Dcris(L(1)) \cong L$ to a pairing
  \[ \left(\cH(\Gamma) \otimes \Dcris(V)\right) \times \left(\cH(\Gamma) \otimes \Dcris(V)\right)^\iota \rTo \cH(\Gamma).\]

  \begin{proposition}
   For any $x \in H^1_{\Iw}(\Qp, V)$ and $w \in \cH(\Gamma) \otimes \Dcris(V)$, we have
   \[ \left\langle \cL^\Gamma_{V}(x), w \right\rangle_{\Iw, \cris} = \gamma_{-1} \cdot \left\langle x, \Omega_{V^*(1), 1, \xi}(w)\right\rangle_{\Iw}.\]
  \end{proposition}

  \begin{proof}
   This is one of many possible forms of the Perrin-Riou reciprocity law. Since all the modules involved are torsion-free it suffices to prove this after inverting $\ell_j$ for all $j$, in which case the statement is identical to Theorem \ref{thm:localrecip} in the light of equation \eqref{eq:regexp}.
  \end{proof}

  \begin{theorem}[Theorem \ref{thm:strongcycloregulator}]
   \label{thm:strongcycloregulator2}
   For any $j \ge 0$, we have
   \[ \cL^\Gamma_{V, \xi}(x)(\chi^j) = -j! (1 - p^j\vp) \left[ \widetilde\exp^*_{\Qp, V^*(1 + j)}(x_{\chi^j}) \otimes t^{-j} e_{j}\right].\]
  \end{theorem}

  \begin{proof}
   Let $v \in \Dcris(V^*(1+j))$, and choose some $w \in \cH(\Gamma) \otimes \Dcris(V^*(1))$ such that $v = w(\chi^{-j}) \otimes t^{-j} e_j$. Then we have
   \[ \langle \cL^\Gamma_{V, \xi}(x)(\chi^j) \otimes t^j e_{-j}, v\rangle_{\cris} = \langle \cL^\Gamma_{V, \xi}(x), w\rangle_{\Iw, \cris}(\chi^j).\]
   By the previous proposition this is equal to
   \[ (-1)^j \langle x, \Omega_{V^*(1), 1, \xi}(w)\rangle_{\Iw}(\chi^j) = (-1)^j \langle x_{\chi^j}, \operatorname{pr}_{\chi^{-j}}\left(\Omega_{V^*(1), 1, \xi}(w)\right)\rangle_{\mathrm{Tate}}.\]
   (The term $\operatorname{pr}_{\chi^{-j}}\left(\Omega_{V^*(1), 1, \xi}(w)\right)$ is only defined modulo the image of $H^1_{\Iw}(\Qpi, V^*(1))_{tors}$ in $H^1(\Qp, V^*(1 + j))$, but this image is the orthogonal complement of the image of $H^1_{\Iw}(\Qpi, V)$ in $H^1(\Qp, V(-j))$, which by assumption contains $x_{\chi^j}$.)
   We know that
   \[ \operatorname{pr}_{\chi^{-j}}\left(\Omega_{V^*(1), 1, \xi}(w)\right) = (-1)^{j+1} j! \widetilde\exp_{\Qp, V^*(1 + j)}\left[ (1 - p^{-1} \vp^{-1}) v\right]\]
   modulo the image of the torsion. Substituting this in, we have
   \[ \langle \cL^\Gamma_{V, \xi}(x)(\chi^j), v\rangle_{\cris} = -j! \langle x_{\chi^j},\widetilde\exp_{\Qp, V^*(1 + j)}\left[ (1 - p^{-1} \vp^{-1}) v\right]\rangle_{\mathrm{Tate}}.\]
   Since $\widetilde\exp^*$ is, by definition, the adjoint of $\widetilde\exp$, the right-hand side is equal to
   \begin{align*}
    -j! \langle \widetilde\exp^*_{\Qp, V^*(1 + j)}(x_{\chi^j}),(1 - p^{-1} \vp^{-1}) v\rangle_{\cris}
    &= -j! \langle \widetilde\exp^*_{\Qp, V^*(1 + j)}(x_{\chi^j}),(1 - p^{-1} \vp^{-1}) v\rangle_{\cris}\\
    &= -j! \langle (1 - \vp) \widetilde\exp^*_{\Qp, V^*(1 + j)}(x_{\chi^j}), v\rangle_{\cris}.
   \end{align*}
   We deduce that
   \[ \langle \cL^\Gamma_{V, \xi}(x)(\chi^j), v\rangle_{\cris} = -j! \langle (1 - \vp) \widetilde\exp^*_{\Qp, V^*(1 + j)}(x_{\chi^j}), v\rangle_{\cris}\]
   for every $v \in \Dcris(V^*(1+j))$, so we must have
   \[ \cL^\Gamma_{V, \xi}(x)(\chi^j) \otimes t^j e_{-j} = -j! (1 - \vp) \widetilde\exp^*_{\Qp, V^*(1 + j)}(x_{\chi^j})\]
   as elements of $\Dcris(V(-j))$, which is clearly equivalent to the claimed formula.
  \end{proof}

 \subsection{The derivative of the regulator}

  We now use Theorem \ref{thm:expformula} to study the derivative of the regulator map $\cL^\Gamma_{V, \xi}$ at its trivial zeroes.

  \begin{proposition}
   \label{prop:regderiv}
   Let $V$ be a crystalline $L$-linear representation of $G_{\Qp}$ with all Hodge--Tate weights $\ge 0$, and $\eta$ a de Rham character of $\Gamma$ whose Hodge--Tate weight $j$ is $\ge 0$ and whose conductor is $n$. Let $W = V(\eta^{-1})$, and let $x$ be an element of $H^1_{\Iw}(\Qpi, V)$ such that $\cL^\Gamma_{V, \xi}(x)(\eta) = 0$. Then:
   \begin{enumerate}
    \item if $\eta$ has conductor $n \ge 1$, then
    \[ x_{\eta} = \frac{1}{j!}\exp_{\Qp, W}\left[\tau(\eta, \xi) \cdot p^{-n} \vp^{-n}\left(\cL^\Gamma_{V, \xi}(x)'(\eta) \otimes t^j e_{-j}\right)\right];\]
    \item if $\eta = \chi^j$, then
    \[
     x_{\eta} = -\frac{1}{j!}\widetilde\exp_{\Qp, W}\left[ (1 - p^{-1} \vp^{-1})\left(\cL^\Gamma_{V, \xi}(x)'(\eta) \otimes t^j e_{-j}\right)\right] \pmod{H^1(\Gamma, H^0(\Qp, V(-j))}.
    \]
    Moreover, $H^1(\Gamma, H^0(\Qp, V(-j))$ is non-zero if and only if $H^0(\Qp, V(-j))$ is.
   \end{enumerate}
  \end{proposition}

  \begin{proof}
   Since $\cL_{V}(x)(\eta) = 0$, we may write
   \begin{equation}
    \label{eq:kappa}
    \cL_{V}(x) = (\gamma - \eta(\gamma)) g
   \end{equation}
   for some $g \in \cH(\Gamma) \otimes \Dcris(V)$. Using equation \eqref{derivation}, one checks that this implies
   \begin{equation}
    \label{eq:deriv1}
    \cL_{V}(x)'(\eta) = g(\eta) \eta(\gamma) \log \chi(\gamma).
   \end{equation}

   We shall now find a formula for $\exp_{\Qp, V(\eta^{-1})}(g(\eta))$ (respectively $\widetilde\exp(g(\eta))$ if $n = 0$); comparing this with \eqref{eq:deriv1} will then give the proposition. We choose an integer $h > j$ such that all Hodge--Tate weights of $V$ lie in $[0, h]$, so the Perrin-Riou exponential map $\Omega_{V, h}$ is well-defined. By enlarging $h$ if necessary, we may also assume that $\Dcris(V)^{\vp = p^{-h}} = 0$.

   Applying $\Omega_{V, h}$ to both sides of equation \eqref{eq:kappa} and using equation \eqref{eq:regexp}, we obtain
   \begin{equation}
    \label{eq:omegaxg}
    \Omega_{V, h}(g) = A_{h, \eta}(x) \text{ modulo torsion,}
   \end{equation}
   where the element
   \[ A_{h, \eta} = \frac{\ell_{h-1} \circ \dots\circ \ell_0}{\gamma - \eta(\gamma)} \in \cH(\Gamma),\]
   since $0 \le j < h$. Crucially, $A_{h, \eta}$ does not vanish at $\eta$, although it vanishes at every other locally algebraic character of degree $\le h-1$.

   We now apply Theorem \ref{thm:expformula} to the element $z = g$, which is valid since $h - j \ge 1$. This tells us that the image of $\Omega_{V, h}(g)$ in $H^1(\Qp, V(\eta^{-1}))$ is given, modulo the image of the torsion in $H^1_{\Iw}(\Qpi, V)$, by
   \[ \operatorname{pr}_\eta\left(\Omega_{V, h}(g)\right) = (-1)^{h - j - 1}(h-j-1)! e_\eta(g(\eta)),\]
   where we write $e_\eta(v)$ as a shorthand for
   \[ \begin{cases}
                              \exp_{\Qp, V(\eta^{-1})}\left[\tau(\eta_0, \xi) \cdot p^{-n} \vp^{-n}\left(v \otimes t^j e_{-j}\right)\right] & \text{if $n \ge 1$,}\\
                              -\widetilde\exp_{\Qp, V(\eta^{-1})} \left[(1 - p^{-1} \vp^{-1}) \left(v \otimes t^j e_{-j}\right)\right] & \text{if $n = 0$.}
                             \end{cases}
   \]
   Plugging in equation \eqref{eq:omegaxg}, this becomes
   \[ \operatorname{pr}_\eta\left(A_{h, \eta}\cdot x \right) = (-1)^{h - j - 1}(h - j - 1)! e_\eta(g(\eta)) \pmod{H^1_{\Iw}(\Qpi, V)_{tors}}\]
   The left-hand side is easy to deal with: it is simply $A_{h, \eta}(\eta) x_\eta$, where $A_{h, \eta}(\eta)$ is a non-zero constant (which we shall evaluate shortly), and $x_\eta$ is the image of $x$ in $H^1(\Qp, V(\eta^{-1}))$ as before. Thus we have
   \[ \frac{A_{h, \eta}(\eta)}{(-1)^{h - j - 1}(h-j-1)!} x_\eta = e_\eta(g(\eta)) \pmod{H^1_{\Iw}(\Qpi, V)_{tors}}.
   \]

   Let us now evaluate the ``fudge factor'' $A_{h, \eta}(\eta)$. We have $\ell_r(\eta) = j - r$ for $r \ne j$, while for $j = r$, we obtain
   \[
    \left(\frac{\ell_j}{\gamma-\eta(\gamma)}\right)(\eta)= \lim_{s \to 0} \frac{\dfrac{\log (\eta(\gamma) \tilde{\chi}(\gamma)^s)} {\log \chi(\gamma)} - j} {\eta(\gamma) \tilde{\chi}(\gamma)^s - \eta(\gamma)}.\]
   The denominator is easily seen to be $s \eta(\gamma) \log \chi(\gamma) + O(s^2)$, while the numerator is simply $s$.

   Hence
   \[ A_{h, \eta}(\eta) = \frac{1}{\eta(\gamma) \log \chi(\gamma)} \prod_{\substack{r = 0 \\ r \ne j}}^{h-1}(j - r) = \frac{(-1)^{h-j-1}(h-j-1)!j!}{\eta(\gamma) \log \chi(\gamma)}.\]

   Putting all the pieces together, we have shown that
   \[
    \frac{j!}{\eta(\gamma) \log \chi(\gamma)} x_\eta = e_\eta(g(\eta)) = \frac{1}{\eta(\gamma) \log \chi(\gamma)} e_\eta\left(\cL^\Gamma_{V, \xi}(x)'(\eta)\right)\pmod{H^1_{\Iw}(\Qpi, V)_{tors}},
   \]
   so
   \[
    x_\eta = \frac{1}{j!} e_\eta\left(\cL^\Gamma_{V, \xi}(x)'(\eta)\right) = \frac{1}{j!}\begin{cases}
                              \exp_{\Qp, V(\eta^{-1})}\left[\tau(\eta_0, \xi) \cdot p^{-n} \vp^{-n}\left(v \otimes t^j e_{-j}\right)\right] & \text{if $n \ge 1$,}\\
                              -\widetilde\exp_{\Qp, V(\eta^{-1})} \left[(1 - p^{-1} \vp^{-1}) \left(v \otimes t^j e_{-j}\right)\right] & \text{if $n = 0$,}
                             \end{cases}\]
   again modulo the image of $H^1_{\Iw}(\Qpi, V)_{tors}$ in $H^1(\Qp, V(\eta^{-1}))$.

   We now analyse the torsion term. We know that $H^1_{\Iw}(\Qpi, V)_{tors}$ is isomorphic as a $\Gamma$-module to $H^0(\Qpi, V)$, and in particular it is crystalline as a $G_{\Qp}$-representation, and thus contains no no non-crystalline characters in its support. Thus its image in $H^1(\Qp, V(\eta^{-1}))$ is zero if $n \ge 1$. If $\eta = \chi^j$, then the image of $H^1_{\Iw}(\Qpi, V)_{tors}$ in $H^1(\Qp, V(-j))$ is precisely $H^1(\Gamma, H^0(\Qpi, V)(-j))$, and in particular its dimension is the same as that of $H^0(\Gamma, H^0(\Qpi, V)(-j)) = H^0(\Qp, V(-j))$. This completes the proof.
  \end{proof}

  \section{Proof of Theorem \ref{thm:sigmas}}
  \label{appendix:diagrams}

  In this appendix, we prove Theorem \ref{thm:sigmas}. We start with some remarks on signs.

  \begin{remark}
   \label{app-signs}
   Recall that by \cite[\S 4.3]{deligne87} to any exact sequences $C: 0 \to X_1 \to X \to X_2 \to 0$ of $L$-vector spaces there is attached a canonical isomorphism
   \[
    \Det(C):\Det(X)\cong\Det(X_1)\cdot \Det(X_2),
   \]
   which is compatible with the commutativity in $\underline{\Det}(L)$ in the following sense. If $X=X_1\oplus X_2$ and if $C_1: 0 \to X_1 \to X \to X_2 \to 0$ and $C_2: 0 \to X_2 \to X \to X_1 \to 0$ are the natural exact sequences, then we have a commutative triangle
   \[
    \xymatrix{
       & {\Det(X)}  \ar[dr]^{\Det(C_2)}  \ar[dl]_{\Det(C_1)}  \\
      {\Det(X_1)\Det(X_2)} \ar[rr]^{\psi_{\Det(X_1)\Det(X_2)}} & &     {\Det(X_2)\Det(X_1)}
     }
   \]
   where
   \[\psi_{\Det(X_1),\Det(X_2)}:\Det(X_1)\Det(X_2)\to\Det(X_2)\Det(X_1) \]
   denotes the commutativity constraint.\footnote{
    Recall that by \cite[Example in \S 4.1]{deligne87} over a field $L$ we can take the category of (graded) line bundles, i.e., one dimensional vector spaces (plus a dimension parameter), for the Picard category $\underline{\Det}(L)$ in which the determinant functor $\Det_L$ takes its values. Then the commutativity constraint is given as
    \[\psi_{\Det_L(V),\Det_L(W)}:\Det_L(V)\Det_L(W)\to\Det_L(W)\Det_L(V), \;\;\;\nu\omega\mapsto (-1)^{\dim_L(V)\dim_L(W)}\omega\nu. \]
    Moreover, by the natural isomorphism \[\Det_L(V)\Det_L(V^*)\cong\Det_L(V)\Det_L(V)^*\cong\Det_L(0)\] we may identify the inverse $\Det_L(V)^{-1}$ of $\Det_L(V)$ with $\Det_L(V^*)$. But note that it differs from the identification using the natural isomorphism \[\Det_L(V^*)\Det_L(V)\cong\Det_L(V)^*\Det_L(V)\cong\Det_L(0)\]
    by the sign $(-1)^{\dim(V)}\id_{\Det_L(0)}$.
   }
   Hence, usually these commutativity constraints do not give rise to any sign ambiguities -- and we often suppress them from the notation -- except in the case that $X_1=X_2$ or if inverses $\Det(X)^{-1}$ are involved, for the latter see e.g.\ (4.1.1) and 4.11 (b) in (loc.\ cit.).
   In particular, for every $L$-vector space $X$, the symmetry automorphism   $\Det(S_X):\Det(X\oplus X)\cong \Det(X\oplus X)$ corresponds to $\Det_L(-1|X)=(-1)^{\dim_L(X)}$ under the isomorphism $\mathrm{Aut}_{\underline{\Det(L)}}(\Det_L(X\oplus X))\cong\mathrm{Aut}_{\underline{\Det}(L)}(\Det_L(0))\cong\mathrm{Aut}_{\underline{\Det}(L)}(\Det_L(X))$, see \S 4.9 of (loc.\ cit.), and one immediately checks the commutativity of the following diagram
   \begin{diagram}
    \Det(X \oplus X) & \rTo^{\Det(C_1)} & \Det(X) \cdot \Det(X) \\
    \dTo^{\Det(S_X)} & \rdTo^{\Det(C_2)} & \dTo_{\psi_{\Det(X), \Det(X)}}\\
    \Det(X \oplus X) & \rTo^{\Det(C_1)} & \Det(X) \cdot \Det(X)
   \end{diagram}
   where $C_i$ correspond to the above short exact sequences for $X_1=X_2=X$.

   Upon replacing $\dim_L(V)$ by the Euler-Poincar\'e characteristic $\chi(C) \coloneqq \sum_i (-1)^i\dim_L(C^i)$ these remarks extend immediately to (perfect) complexes $C$ of $L$-vector spaces.
  \end{remark}

  Let $V$ be a crystalline $L$-linear representation of $G_{\Qp}$. Consider the four filtered $L$-vector spaces $\mathcal{D}_i=(D_i, F_i)$, with $D_i = \Dcris(V)$ for each $i$, and the subspaces $F_i$ defined by
  \begin{align*}
   F_3 &= \Fil^0 \Dcris(V)\\
   F_1 &= h^{-1}(F_3)\\
   F_4 &= g(F_3)\\
   F_2 &= g(F_1) = h^{-1}(F_4)
  \end{align*}
  where $g=1-\vp$ and $h=1-p^{-1}\vp^{-1}$. We obtain a commutative square of filtered $L$-vector spaces \begin{equation}
   \label{square}
   \begin{diagram}
    \cD_1 & \rTo^g & \cD_2 \\
    \dTo^h & & \dTo^h \\
    \cD_3 & \rTo^g & \cD_4
   \end{diagram}
  \end{equation}

  \begin{lemma}
   \label{lemma-squares}
   \mbox{~}
   \begin{enumerate}
    \item Let
    \[
     \xymatrix{
      A' \ar[d]_{ h_A} \ar[r]^{g' } & B' \ar[d]^{h_B } \\
      A \ar[r]^{g } & B
     }
    \]
    be a commutative square of $K$-vector spaces considered also as complexes (concentrated in degree $0$). Then this can be extended to a ``$3\times 3$''-diagram in the derived category of $K$-vector spaces
    \[
     \xymatrix{
      A' \ar[d]_{ h_A} \ar[r]^{g' } & B' \ar[d]^{h_B } \ar[r]^{u } & C(g') \ar[d]_{ H} \ar[r]^{v } & TA' \ar[d]^{Th_A } \\
      A \ar[r]^{g }\ar[d]_{\nu } & B   \ar[d]_{ } \ar[r]^{ } & C(g) \ar@{.>}[d]_{ } \ar[r]^{ } & TA \ar[d]^{T\nu } \\
      C(h_A) \ar[d]_{\omega } \ar[r]^{G } & C(h_B) \ar[d]_{ } \ar[r]^{ } & C(G) \ar@{.>}[d]_{ } \ar[r]^{ }\ar @{} [dr] |{-} & TC(h_A) \ar[d]^{T\omega } \\
       TA'   \ar[r]^{Tg' } & TB'   \ar[r]^{ Tu} & TC(g')   \ar[r]^{Tv } & T^2A'
     },
    \]
    such that the following diagram of determinants commutes (with the obvious commutativity and associativity constraints which we have suppressed)
    \[
     \xymatrix{
      [B] \ar[d]_{ } \ar[r]^{ } & [B'][C(h_B)] \ar[d]^{ } \\
      [A][C(g)] \ar[r]^{ } & [A'][C(g')][C(h_A)][C].
     }
    \]
    Here $C(f)$ denotes the mapping cone of a map $f$, $T$ is the shift by one functor and the right lower square above anti-commutes. Equivalently the following natural diagram commutes:
    \[
     \xymatrix{
      {\phantom{[C(h_B)]^{-1}}}[B']\phantom{[C(h_B)]^{-1}} \ar[d]_{ } \ar[r]^{ } & {\phantom{[C(h_B)]^{-1}}}[A'][C(g')]\phantom{([C(h_A)][C])^{-1}} \ar[d]^{ } \\
      {\phantom{[C(h_B)]^{-1}}}[B][C(h_B)]^{-1} \ar[r]^{ } & {\phantom{[C(h_B)]^{-1}}}[A][C(g)]([C(h_A)][C])^{-1}
     }
    \] Moreover, in the above diagram all continuous arrows arise naturally from the cone-construction, while the dotted arrows arise from the isomorphism of complexes between
    \[
     C(H): A' \rTo^{(h_A, -g')}A \oplus B \rTo^{\langle g, h_B\rangle} B
    \]
    starting in degree $-2$ on the left and
    \[
     C(H): A' \rTo^{(g', -h_A)} B' \oplus A \rTo^{\langle h_B, g\rangle} B
    \]
    which is given by $\id_{A'}$ in degree $-2$ and by the identity (and permutation of summands) otherwise. Alternatively, one can replace $C(G)$ by $C(H)$ and instead adjust the horizontal arrows ending or starting in it.
    \item Applying the first item to each of the squares occurring in the ``$2\times 2\times 2$''-cube \eqref{square} we obtain an (anti-)commutative ``$3\times 3\times 3$''-cube, some faces/sheets (consisting of distinguished triangles) of which are given as follows:
    \begin{subequations}
     \begin{equation}
      \label{face1}
      \xymatrix{
       F_1 \ar[d]_{g } \ar@^{^(->}[r]^{ } & D_1 \ar[d]_{g } \ar[r]^{ } & D_1/F_1 \ar[d]^{ } \\
       F_2 \ar[d]_{ } \ar@^{^(->}[r]^{ } & D_2 \ar[d]_{ } \ar[r]^{ } & D_2/F_2 \ar[d]^{ } \\
       F_{12} \ar[r]^{ } & D_{12} \ar[r]^{ } & (D/F)_{12}
      }
     \end{equation}

     \begin{equation}
      \label{face2}
      \xymatrix{
       D_1/F_1 \ar[d]_{g } \ar[r]^{ h} & D_3/F_3 \ar[d]_{ g} \ar[r]^{ } & (D/F)_{13} \ar[d]^{ } \\
       D_2/F_2 \ar[d]_{ } \ar[r]^{ h} & D_4/F_4 \ar[d]_{ } \ar[r]^{ } & (D/F)_{24} \ar[d]^{ } \\
       (D/F)_{12} \ar[r]^{ } & (D/F)_{34} \ar[r]^{ } & 0
      }
     \end{equation}

     \begin{equation}
      \label{face3}
      \xymatrix{
       F_3 \ar[d]_{ g} \ar@^{^(->}[r]^{ } & D_3 \ar[d]_{g } \ar[r]^{ } & D_3/F_3 \ar[d]^{ } \\
       F_4 \ar[d]_{ } \ar@^{^(->}[r]^{ } & D_4 \ar[d]_{ } \ar[r]^{ } & D_4/F_4 \ar[d]^{ } \\
       F_{34} \ar[r]^{ } & D_{34} \ar[r]^{ } & (D/F)_{34}
      }
     \end{equation}

     \begin{equation}
      \label{face4}
      \xymatrix{
       F_1 \ar[d]_{h } \ar[r]^{g } & F_2 \ar[d]_{h } \ar[r]^{ } &  F_{12} \ar[d]^{ } \\
       F_3 \ar[d]_{ } \ar[r]^{g } & F_4 \ar[d]_{ } \ar[r]^{ } & F_{34} \ar[d]^{ } \\
       F_{13} \ar[r]^{ } & F_{24} \ar[r]^{ } & 0
      }
     \end{equation}
    \end{subequations}
    Moreover we have canonical isomorphisms
    \begin{align*}
     F_{12} &\rTo_{\cong}^h F_{34},\\
     F_{13} &\rTo_{\cong}^g F_{24},\\
     D_{34} &\rTo_{\cong}^\id D_{12},\\
     (D/F)_{12} &\rTo_{\cong}^h (D/F)_{34}.
    \end{align*}
   \end{enumerate}
  \end{lemma}

  \begin{proof}
   A similar statement as in (i) can be found in \cite[lem.\ 3.9]{breuning11} (see also \cite[cor. 1.10]{knudsen02}) for general triangular categories, but with mapping cones possibly replaced by quasi-isomorphic complexes. This is usually proved using the octahedral axiom. In our specific simple setting one can alternatively verify both statements explicitly. The zeroes in (ii) are the consequence of the specific choice of the $F_i$'s.
  \end{proof}

  \begin{proposition}
   Using canonical isomorphisms induced by the above cube we obtain the following commutative diagram
   \[ \xymatrix{
     \u \ar[d]_{ } \ar[r]^{ } & [F_{34}] [F_1][(D/F)_{13}]^{-1} [F_4]^{-1} \ar[d]^{ } \\
     \u   \ar[r]^{ } & [F_{12}] [F_1][(D/F)_{24}]^{-1} [F_4]^{-1}  }
   \]
   where the the isomorphism in the first line arises by using the first lines of the above faces \eqref{face1}, \eqref{face2}, \eqref{face3} and the second line of \eqref{face4} while the isomorphism in the second line arises by using the second lines of the above faces \eqref{face1}, \eqref{face2}, \eqref{face3} and the first line of \eqref{face4}.
  \end{proposition}

  \begin{proof}
   Applying Lemma \ref{lemma-squares} to each of the four involved faces gives the following commutative diagrams:
   \[\xymatrix{
     [D_1] \ar[d]_{ } \ar[r]^{ } & [F_1][D_1/F_1] \ar[d]^{ } \\
     [D_2]([D_{12}])^{-1} \ar[r]^{ } & [F_2][D_2/F_2]([F_{12}][(D/F)_{12}])^{-1}   }
   \]
   \[ \xymatrix{
     [D_3/F_3]^{-1} \ar[d]_{ } \ar[r]^{ } & [(D/F)_{13}]^{-1}[D_1/F_1]^{-1} \ar[d]^{ } \\
     [D_4/F_4]^{-1}([(D/F)_{34}]) \ar[r]^{ } & [(D/F)_{24}]^{-1}[D_2/F_2]^{-1} ([(D/F)_{12}])   }
   \]
   \[\xymatrix{
     [F_3][D_3/F_3] \ar[d]_{ } \ar[r]^{ }& [D_3] \ar[d]^{ } \\
    [F_4][D_4/F_4]([F_{34}][(D/F)_{34}])^{-1} \ar[r]^{ } &[D_4]([D_{34}])^{-1}    }
   \]
   \[\xymatrix{
     [F_3]^{-1}[F_{34}]^{-1} \ar[d]_{ } \ar[r]^{ }& [F_4]^{-1} \ar[d]^{ } \\
     [F_1]^{-1}[F_{12}]^{-1}([F_{13}] )^{-1} \ar[r]^{ } & [F_2]^{-1}([F_{24}])^{-1}
    }
   \]

   Multiplying all the diagrams up and cancelling corresponding objects leads in the first line to the isomorphism
   \[[D_1][F_{34}]^{-1}\cong [F_1][(D/F)_{13}]^{-1}[D_3][F_4]^{-1}\]
   while in the second line we obtain
   \[[D_2][F_4][F_1]^{-1}[F_{12}]^{-1}\cong [(D/F)_{24}]^{-1}[D_4]\]
   times the error terms in parentheses
   \[[D_{12}]^{-1}[F_{34}]^{-1}[F_{13}]^{-1}\cong [F_{12}]^{-1}[D_{34}]^{-1}[F_{24}]^{-1}.\]
   Multiplication with the inverse of the left vertical map and cancellation of the $[D_i]=[\Dcris(V)]$ among each other using the identity map leads to the commutative diagram
   \[
    \xymatrix{
    {\u} \ar[d]_{ } \ar[r]^{ } & [F_{34}] [F_1][(D/F)_{13}]^{-1} [F_4]^{-1} \phantom{([D_{12}]^{-1}[D_{34}][F_{34}]^{-1}[F_{12}][F_{13}]^{-1}[F_{24}])}\ar@<-3cm>[d]^{ } \\
    {\u}   \ar[r]^{ } & [F_{12}] [F_1][(D/F)_{24}]^{-1} [F_4]^{-1} \left([D_{12}]^{-1}[D_{34}][F_{34}]^{-1}[F_{12}][F_{13}]^{-1}[F_{24}]\right)^{-1} }
   \]
   Using the canonical identities
   \[[D_{12}]^{-1}[D_{34}]=[F_{34}]^{-1}[F_{12}]=[F_{13}]^{-1}[F_{24}]\cong \u\] we see that this error term is canonically isomorphic to the unit object $\u$, whence the claim.
  \end{proof}

  \begin{corollary}
   The isomorphism $\u \cong [R\Gamma(\Qp,V)][\Dcris(V)]$ given by isomorphisms \eqref{iso1}--\eqref{iso4} in \S \ref{sect:specialcases} above coincides with
   \[  (-1)^{\dim V} \theta(V).\]
  \end{corollary}

  \begin{proof}
   Using the canonical isomorphism $[F_4]\cong[D_4][D_4/F_4]^{-1}$ the Proposition induces also a
   commutative diagram
   \begin{equation}
    \label{comparison-diagram}
    \xymatrix{
     {\u} \ar[d]_{ } \ar[r]^{ } & [F_{34}] [F_1][(D/F)_{13}]^{-1} [D_4]^{-1}[D_4/F_4] \ar[d]^{ } \\
     {\u}   \ar[r]^{ } & [F_{12}] [F_1][(D/F)_{24}]^{-1} [D_4]^{-1}[D_4/F_4]
    }
   \end{equation}
   Now we claim that the (inverse of the) upper line defines $\theta(V)$ times $[-\id_{\Dcris(V)}]$ while the (inverse of the) lower one corresponds to the isomorphism in the Conjecture. Indeed, first note that we have natural isomorphisms
   \[ TF_{12}\rTo^{h}_{\cong} TF_{34} \cong \H^0(\Qp,V),\]
   \[ (D/F)_{13 }\rTo^{g}_{\cong} (D/F)_{24 }\cong\H^2(\Qp,V),\]
   \[-\widetilde{\exp}^*_{\Qp,V^*(1)}:\H^1(\Qp,V)/\H^1_f(\Qp,V)\cong F_1\]
   (the sign has the same origin as that in Proposition \ref{prop:noncrystallinedet})
   and
   \[\widetilde{\log}_V:\H^1_f(\Qp,V)\cong D_4/F_4.\]
   Secondly, up to the identification $\widetilde{\exp}^*_{\Qp,V^*(1)}$, the exact sequence
   \[ \Sigma(\Qp, V):\quad 0 \rTo H^0(\Qp, V) \rTo \frac{H^1(\Qp, V)}{H^1_f(\Qp, V)} \rTo^{-(1-\vp)\circ\widetilde\exp^*} \Dcris(V) \rTo \frac{\Dcris(V)}{s(V)} \rTo 0\]
   (where $s(V) = (1- \vp)(1 - p^{-1} \vp^{-1})^{-1} \Fil^0 \Dcris(V)$) corresponds to the combination of the triangles
   \[F_1\rTo^g F_2 \rTo F_{12}\]
   and
   \[F_2 \hookrightarrow D_2 \to D_2/F_2\]
   using that $F_2=s(V)$. Similarly, the exact sequence
   \[ \Sigma(\Qp, V^*(1))^*:\quad 0 \rTo s(V) \rTo \Dcris(V) \rTo^{\widetilde\exp \circ (1 - p^{-1}\vp^{-1})} H^1_f(\Qp, V) \rTo H^2(\Qp, V) \rTo 0,\]
   corresponds to
   \[D_2/F_2\rTo^h D_4/F_4 \rTo (D/F)_{24}.\]
   Altogether we just obtain the identifications used for the second line. Concerning the first line consider the following commutative diagram
   \begin{diagram}
    0 & \rTo & \H^0(\Qp,V) & \rTo & F_3       & \rTo^{1-\vp}        & D_4                  & \rTo^{\widetilde{\exp}} & \H^1_f(\Qp,V) & \rTo & 0 \\
      &      &  \dTo^{\id} &      & \dTo      &                     & \dTo^{(\id,0)}                 &                         & \dTo^\id \\
    0 & \rTo & \H^0(\Qp,V) & \rTo & \Dcris(V) & \rTo^{(1-\vp, \bar{\id})} & \Dcris(V)\oplus t(V) & \rTo^{\widetilde{\exp} \oplus \exp} & \H^1_f(\Qp,V) & \rTo & 0 \\
      &      &             &      & \dTo      &                     & \dTo^{<0,\id>} \\
      &      &             &      & t(V)      & \rTo^\id            & t(V) \\
   \end{diagram}

   or in the derived category
   \begin{diagram}
    C_3:  & &  C_4: \\
    C_1:\phantom{mm}    F_3 \phantom{mmm}  & \rTo^{1-\vp} & D_4 & \rTo & C(1-\vp) & \rTo \\
    \dTo & & \dTo^{(\id,0)} & & \dTo^{\widetilde{\exp}} \\
    C_2:   \phantom{m} \Dcris(V)\phantom{m} & \rTo^{(1-\vp, \bar{\id})} & \Dcris(V)\oplus t(V) & \rTo & C\Big( (1 - \vp, \bar{\id}) \Big) & \rTo \\
    \dTo & & \dTo^{<0,\id>} & & \dTo \\
    t(V)      & \rTo^\id            & t(V) & \rTo & 0 & \rTo\\
    \dTo & & \dTo & & \dTo \\
   \end{diagram}

  For simplicity we identify $C(1-\vp)\cong C\Big( (1 - \vp, \bar{\id})$. It follows again from Lemma \ref{lemma-squares} that there is a commutative diagram
   \[\xymatrix{
     [\Dcris\oplus t(V)] \ar[d]_{ [C_2]} \ar[r]^{[C_4] } & [\Dcris][t(V)]   \ar[r]^{[C_1]\id_{[t(V)]} }\ar[dl]_{\id_{[D]}\theta_1} & [F_3] [C(1-\phi)] [t(V)]\ar[d]^{\id_{[F_3]}\psi_{[C(1-\phi)], [t(V)]}} \\
     [\Dcris][C(1-\phi )] \ar[rr]^{[C_3] {\id_{[C(1-\phi )]}} } && [F_3]  [t(V)] [C(1-\phi)]    }\]
   where $\theta_1$ is defined by the commutativity of the left and right subdiagrams. Now it is easy to check that the  right half of the diagram can be described also by
   \[\xymatrix{
     [\Dcris][t(V)]\ar[drr]^{(\psi_{[F_3],[t(V)]}\circ [C_3])\theta_1} \ar[d]_{\id_{[D]}\theta_1} \ar[rr]^{\psi_{[\Dcris],[t(V)]}} && [t(V)][\Dcris]  \ar[d]^{\id_{[t(V)]}[C_1]}     \\
     [\Dcris][C(1-\phi))] \ar[r]^{[C_3]\id_{[C(1-\phi )]}} &[F_3][t(V)] [C(1-\phi)]\ar[r]^{\psi_{[F_3],[t(V)]}\id_{[C(1-\phi)]}} &  [t(V)][F_3] [C(1-\phi)]}\]
   i.e., by looking at the diagonal we see that $\id_{[t(V)][F_3]}\theta_1$ equals the composite of the upper line in the following commutative diagram
   \[\xymatrix{
     [t(V)][F_3] [t(V)]\ar[r]^{\psi_{[t(V)],[F_3]}\id_{[t(V)]}}  & [F_3][t(V)] [t(V)] \ar[d]^{\id_{[F_3]}\psi_{[t(V)],[t(V)]}} \ar[r]^{[C_3]^{-1}\id_{[t(V)]}} & [\Dcris][t(V)] \ar[r]^{\psi_{[\Dcris],[t(V)]} } & [t(V)][\Dcris] \ar[d]^{\id_{[t(V)]} [C_3]} \ar[r]^{\id_{[t(V)]}[C_1]} & [t(V)][F_3] [C(1-\phi)]\\
    & [F_3][t(V)] [t(V)]\ar[rr]^{\psi_{[F_3],[t(V)]}\id_{[t(V)]}} &  & [t(V)][F_3][t(V)]. }\]

    Since  on the other hand the following diagram commutes by Remark \ref{app-signs}
    \[ \xymatrix{
      [t(V)][F_3] [t(V)]\ar@{=}[d]\ar[r]^{\psi_{[t(V)],[F_3]}\id_{[t(V)]}}  & [F_3][t(V)] [t(V)]  \ar[r]^{\id_{[F_3]}\psi_{[t(V)],[t(V)]}} & [F_3][t(V)] [t(V)] \ar[r]^{\psi_{[F_3],[t(V)]}\id_{[t(V)]}} & [t(V)][F_3] [t(V)] \ar@{=}[d] \\
      [t(V)][F_3] [t(V)] \ar[rrr]^{[-\id_{t(V)}]}  &   &   & [t(V)][F_3] [t(V)],   }\]

    we see that $\id_{[t(V)][F_3]}\theta_1$ is just given by

    \[\xymatrix{
      [t(V)][F_3] [t(V)] \ar[r]^{\id_{[t(V)]}[C_3]^{-1}} & [t(V)][D] \ar[r]^{\id_{[t(V)]}[C_1]} & [t(V)][F_3] [C(1-\phi)],   }\]

   i.e., upon identifying $D_3$ and $D_4$ the map $\theta_1$ is induced by the triangles
   \[F_3\to D_3\to D_3/F_3\to,\]
   \[F_3\to F_4\to F_{34}\to, \]
      \[F_4\to D_4\to D_4/F_4\to \]
   and the factor $[-\id_{t(V)}]$. Dually there exists a map $\theta_2:[F_3]\to [C(1-p^{-1}\phi^{-1})^*]$  which is induced by
   \[F_3\to D_3\to D_3/F_3\to,\]
   \[D_1/F_1\to D_3/F_3\to (D/F)_{13}\to, \]
      \[F_1\to D_1\to D_1/F_1\to  \]
   and the factor $[-\id_{t(V^*(1))^*}]$, and such that $\theta(V)$ is induced by $\theta_1$ and $\theta_2$
   together with the canonical exact sequence
   \[0\to\H^1_f(\Qp,V)\to \H^1(\Qp,V) \to \H^1(\Qp,V)/\H^1_f(\Qp,V)\to 0\]
   and cancellation of the $\Dcris$'s. Altogether we see that $\theta(V)$ equally can be expressed upon cancelling the various $D_i$'s by the combination of the triangles
   \[F_1\to D_1\to D_1/F_1\to ,\]
   \[D_1/F_1\to D_3/F_3\to (D/F)_{13}\to, \]
   \[F_3\to D_3\to D_3/F_3\to,\]
   \[F_3\to F_4\to F_{34}\to, \]
   \[F_4\to D_4\to D_4/F_4\to, \]
   which altogether just define the first line of \eqref{comparison-diagram}, times $[-\id_{t(V^*(1))^*}][-\id_{t(V)}]=[-\id_{\Dcris(V)}]$. This completes the proof.
  \end{proof}

\section{Comparison with Kato's rank one epsilon-isomorphisms}
 \label{App-comparison}

 In this section we explain why the construction of the epsilon-isomorphism in \cite{venjakob11} actually turns out to be the same as that in this paper. This relies on the fact that roughly speaking both the regulator map in \cite{loefflerzerbes11} and the epsilon-isomorphism in \cite{venjakob11} arise by taking inverse limits in the unramified direction from objects defined over cyclotomic $\Zp$-extensions.

 Let $K$ be a finite, unramified extension of $\Qp$, $K_\infty \coloneqq K(\mu(p))$ and $G=G(K_\infty/\Qp)$, $H=<\tau_p>=G(K/\Qp)$, $\Gamma=G(\Qpi / \Qp)$, $\Zp(r)=\Zp t_r$, $e_r \coloneqq t^{-r}\otimes t_r\in \Dcris(\Qp(r))$, $r\geq 0$. Then we have the following commutative diagram

\[
\xymatrix{
  {\lu(K_\infty)(r-1)} \ar[d]_{D\log\; g_K} \ar[rr]^{Kummer(r-1)} && {\H^1(\Qp,\T_K(\Zp(r)))} \ar[d]^{\cong} \\
   {(\cO_K[[X]]\otimes t_r)^{\psi=1}} \ar[d]_{(1-\varphi)\otimes \id} \ar@{^(->}[r]^{ } & (X^{-r}\cO_K[[X]]\otimes t_r)^{\psi=1}\ar@{=}[r] \ar[d]^{(1-\varphi)\otimes \id} & N(\Zp(r))^{\psi=1} \ar[d]^{1-\varphi}\\
  {\cO_K[[X]]^{\psi=0}\otimes t_r} \ar[d]_{\partial^{-r}\otimes t^{-r}} \ar@{^(->}[r]^{ } & {\varphi(X)^{-r}\cO_K[[X]]^{\psi=0}\otimes t_r} \ar[d]^{t^r\otimes t^{-r}}\ar@{=}[r] & {\varphi^*(N(\Zp(r)))^{\psi=0}\ar[d]^{comp}}\\
    {\cO_K[[X]]^{\psi=0}\otimes e_r}\ar[dd]_{\mathfrak{M}^{-1}\otimes\id} \ar[r]_(0.45){=t^r \partial^r\otimes\id }^(0.45){\ell_0 \cdots \ell_{r-1}\otimes \id} & {\phantom{m}\left(\frac{t}{\varphi(X)}\right)^r\cO_K[[X]]^{\psi=0}\otimes e_r}\ar@{^(->}[r] & {\left( \BB^+_{rig,K}\right)^{\psi=0}\otimes \Dcris(\Qp(r))} \ar[d]^{ \mathfrak{M}^{-1}\otimes\id}\\
      & &{\mathcal{H}_K(\Gamma)\otimes_{\Qp} \Dcris(\Qp(r))} \ar[d]_{\mu_r^{-1}= \ell(\Qp(r))^{-1}} \\
  {\cO_K[[\Gamma]]\otimes e_r} \ar[dd]_{\Theta_r} \ar[rr]^{ } && {\mathcal{H}_K(\Gamma)\otimes_{\Qp} \Dcris(\Qp(r))} \ar[d]_{\omega\otimes \id } \\
      & & {\mathcal{H}(\Gamma)\otimes S_K \otimes_{\Qp} \Dcris(\Qp(r)) }\ar[d]_{\epsilon_{\Qp,\xi,dR}(\Qp(r))\otimes (\mathcal{H}(\Gamma)\otimes S_K) } \\
   {\T(\Zp(r))\otimes_\Lambda \Lambda_{\tau_p}} \ar[rr]^{ } & &{\T(\Zp(r))\otimes_\Lambda \mathcal{H}(\Gamma)\otimes S_K } } \]

  where $\ell_i \coloneqq t \partial - i$, $\partial= (1+X)\frac{d}{dX}$, $t=\log(1+X)$ and the twisted ring \[\Lambda_{\tau_p}=\{x\in \Lambda(G)\hat{\otimes}\cO_K | (\tau_p\otimes 1)\cdot x= (\id\otimes \tau_p)(x)\}=S_K\] should be compared to $S_n$ in \cite{loefflerzerbes11}. Furthermore
  \[\T(\Zp(r))=\Lambda(G)^\natural\otimes_{\Zp} \Zp(r)\] and the map $\Theta$ is defined by
  \[\Theta_r(\lambda\otimes e_r)=(1\otimes t_r)\otimes \sum_{i=0}^{\#H-1} \tau_p^i\;\tau_p^{-i}(\lambda)\] while
\[\omega: \mathcal{H}_K(\Gamma)=\mathcal{H}(\Gamma)\otimes \cO_K\cong \mathcal{H}(\Gamma)\otimes S_K\] sends $h\otimes o$ to $h\otimes \sum_{i=0}^{\#H-1} \tau_p^i\;\tau_p^{-i}(o)$. Finally the map
\[comp: \varphi^*(N(\Zp(r)))^{\psi=0}\to \left( \BB^+_{rig,K}\right)^{\psi=0}\otimes_K \Dcris(\Qp(r))\]
is induced from the inverse of Berger's comparison isomorphism
\[ \BB^+_{rig,K}[\tfrac{1}{t}]\otimes_K \Dcris(\Qp(r))=\BB^+_{rig,K}[\tfrac{1}{t}]\otimes_{\BB_K^+} N(\Qp(r)),\]
 (cf.~\cite{berger04}), where $\BB^+_K=\cO_K[[X]][\frac{1}{p}]$.

Taking the limit over $K$ within some unramified extension $K'/\Qp$ with $G(K'(\mu(p))/\Qp)$ being of dimension $2$ and embedding $\mathcal{H}(\Gamma)\hat{\otimes} S_\infty $ into $\mathcal{H}_{\hat{K'}}(G)$ we obtain the commutative diagram

\[\xymatrix{
  {\lu(K'_\infty)(r-1)} \ar[d]_{\tilde{\cL}_\xi(\T_{K'}(\Zp(r))) } \ar[r]^{Kummer(r-1)}_{\cong} & {\H^1(\Qp,\T_{K'}(\Zp(r)))} \ar[d]^{(\epsilon_{\Qp,\xi,dR}(\Qp(r))\otimes \mathcal{H}_{\hat{K'}}(G) ) \ell(\Qp(r))^{-1}\mathcal{L}^G_{\Qp(r),\xi} } \\
   {\T_{K'}(\Zp(r))\otimes_\Lambda \Lambda_{\tau_p}} \ar[r]^{ } & {\T_{K'}(\Zp(r))\otimes_\Lambda \mathcal{H}_{\hat{K'}}(G) } }
    \]

In particular
\[ \epsilon_{\Qp,\xi,dR}(\Qp(r)) \cdot \ell(\Qp(r))^{-1}\mathcal{L}^G_{\Qp(r),\xi} = -{\epsilon'}_{\Lambda(\Gamma),\xi^{-1}}(\mathbb{T}(\Zp(r))),\]
where the latter $\epsilon$-isomorphism is the one defined in \cite[Def. 2.5]{venjakob11}. For this we use also \cite[lemma A.4]{venjakob11}, note the signs $-\mathcal{L}_{\epsilon^{-1}}$ in the definition (2.13) in (loc.\ cit.). Multiplying by $(-1)^r \gamma_{-1}$, which gives the action of $\gamma_{-1}$ on $\T_{K'}(\Zp(r))$, has the effect of replacing $\xi$ with $-\xi$ on the right-hand side, and thus we obtain
\[
 (-1)^{r + 1} (\gamma_{-1})^r \epsilon_{\Qp,\xi,dR}(\Qp(r)) \cdot \ell(\Qp(r))^{-1}\mathcal{L}^G_{\Qp(r),\xi}
 = {\epsilon'}_{\Lambda(\Gamma),\xi}(\mathbb{T}(\Zp(r))).
\]
The quantity $(-1)^{r + 1} (\gamma_{-1})^r$ is precisely the factor appearing in the definition of the $\varepsilon$-isomorphism $\varepsilon_{\Lambda_L(G), \xi}(V)$ in the present paper for $V = \Qp(r)$, so the isomorphisms in the present paper and in \cite{venjakob11} coincide as required.

\section*{Acknowledgements}

 We are grateful to Marco Schlichting for his explanations of $K_1$ and $SK_1$.

\newcommand{\etalchar}[1]{$^{#1}$}
\providecommand{\bysame}{\leavevmode\hbox to3em{\hrulefill}\thinspace}
\renewcommand{\MRhref}[2]{%
  \href{http://www.ams.org/mathscinet-getitem?mr=#1}{#2}.
}

\end{document}